\documentclass[11pt]{article}

\usepackage{amsfonts, amsmath, amsthm, amssymb,  graphicx, color}

\usepackage[british]{babel}
\usepackage[utf8]{inputenc}

\usepackage{enumerate}
\usepackage{a4wide}
\usepackage[usenames,dvipsnames]{xcolor}
\usepackage[all,cmtip]{xy}
\usepackage[normalem]{ulem}
\usepackage{xspace}
\usepackage{verbatim}
\usepackage{cancel}

\usepackage{hyperref}

\newtheorem {theorem}{Theorem}
\newtheorem*{theorem*}{Theorem}
\numberwithin{theorem}{section}
\numberwithin{equation}{section}

\newtheorem {lemma}[theorem]{Lemma}
\newtheorem {proposition}[theorem]{Proposition}

\theoremstyle{definition}
\newtheorem {remark}[theorem]{Remark}
\newtheorem {definition}[theorem]{Definition}
\newtheorem {notation}[theorem]{Notation}

\newcommand{\Ja}{\grg}  % questa era una a e poi una alfa
\newcommand{\Jb}{\grd}  % questa era una b e poi una beta
\newcommand{\Jc}{\gre}  % questa era una c e poi una gamma
\newcommand{\Jp}{\grg} %questa era una n
\newcommand{\Jq}{\grd} %questa era una m

\DeclareMathOperator{\Ad}   {Ad}

\DeclareMathOperator{\End}  {End}

\DeclareMathOperator{\Hom}  {Hom}

\DeclareMathOperator{\Spec} {Spec}
\DeclareMathOperator{\Tr}   {Tr}

\DeclareMathOperator{\SL}   {SL}
\DeclareMathOperator{\PSL}  {PSL}

\DeclareMathOperator{\Mat}  {Mat}
\DeclareMathOperator{\Res}  {Res}
\DeclareMathOperator{\res}  {Res}
\DeclareMathOperator{\Der}  {Drv}
\DeclareMathOperator{\Funct}  {Funct}

\DeclareMathOperator{\Jdeg}  {Jdeg}
\DeclareMathOperator{\Indici}  {\Gamma}

\newcommand{\Id}  {{\mathrm{ Id }}}
\newcommand{\FF} {\mathcal F}

\newcommand{\Op}       {\operatorname{Op}}
\newcommand{\Gr}       {\operatorname{Gr}}
\newcommand{\V}        {\mathbb{V}}
\newcommand{\nop}[1]   {: #1 :}
\newcommand{\Sug}   {S}
\newcommand{\Sstorto}   {T^{(2)}}
\newcommand{\zz}   {c}
\newcommand{\zzt}   {c^{(t)}}
\newcommand{\zzs}   {c^{(s)}}

\newcommand{\calSp}{\mathcal{S} \mathit{p}}

\newcommand{\Rtest}    {{\mathnormal{R}}}
\newcommand{\bRtest}   {{\mathnormal{\overline R}}}
\newcommand{\bI}       {{\mathnormal{\bar I}}}
\newcommand{\bJ}       {{\mathnormal{\bar J}}}
\newcommand{\hJ}       {{\mathnormal{\hat J}}}
\newcommand{\calBg}[1]    {{\calB(\hat{\gog}_{#1})}}
\newcommand{\calBU}[1]    {{\calB(U_{#1})}}
\newcommand{\calBhU}[1]   {{\calB(\hU_{#1})}}
\newcommand{\hU}       {{\hat U}}
\newcommand{\hZ}       {Z}
\newcommand{\bU}       {{\mathnormal{\overline U}}}
\newcommand{\lt}       {\operatorname{lt}}

\newcommand{\Sp}       {\operatorname{Sp}}

 \newcommand{\mC}{\mathbb C}

\newcommand{\mN}{\mathbb N}  
  
  \newcommand{\mV}{\mathbb V}
  
\newcommand{\mZ}{\mathbb Z}

\newcommand{\calB}{\mathcal B}  
\newcommand{\calE}{\mathcal E} \newcommand{\calF}{\mathcal F} \newcommand{\calG}{\mathcal G}
 \newcommand{\calI}{\mathcal I}

\newcommand{\gob}{\mathfrak b}  
  \newcommand{\gog}{\mathfrak g}
  
 \newcommand{\gol}{\mathfrak l} 
  
  \newcommand{\gos}{\mathfrak s}
\newcommand{\got}{\mathfrak t}  
  
\newcommand{\goz}{\mathfrak z}

\newcommand{\gra}{\alpha} \newcommand{\grb}{\beta}       \newcommand{\grg}{\gamma}
\newcommand{\grd}{\delta} \newcommand{\gre}{\varepsilon} 
 \newcommand{\grl}{\lambda}     
\newcommand{\grf}{\varphi}

\newcommand{\lra}      {\longrightarrow}
\newcommand{\coinc}    {\equiv}
\newcommand{\isocan}   {\simeq}
\renewcommand{\geq}    {\geqslant}%vuole amssymb
\renewcommand{\leq}    {\leqslant}%vuole amssymb
\newcommand{\hgog}   {\hat\gog}
\newcommand{\limind} {\varinjlim}
\newcommand{\limpro} {\varprojlim}
\newcommand{\st}     {:}

\allowdisplaybreaks

\begin{document}

\author{Giorgia Fortuna, Davide Lombardo, Andrea Maffei, Valerio Melani}

\title{Local opers with two singularities: the case of $\gos\gol(2)$}

\newcommand{\Addresses}{{
		\bigskip
		\footnotesize
		
		\textit{E-mail addresses: }\texttt{giorgiafortuna@gmail.com}, \texttt{davide.lombardo@unipi.it}, \texttt{andrea.maffei@unipi.it}, \texttt{valerio.melani@unipi.it}
		\par\nopagebreak
		\textit{Mail address:} 
		\textsc{Dipartimento di Matematica, Universit\`a di Pisa, Largo B. Pontecorvo 5, 56127 Pisa, Italy}
	}}

\date{}

\maketitle

\begin{abstract}
We study local opers with two singularities for the case of the Lie algebra $\mathfrak{sl}(2)$, and discuss their connection with a two-variables extension of the affine Lie algebra. We prove an analogue of the Feigin-Frenkel theorem describing the centre at the critical level, and an analogue of a result by Frenkel and Gaitsgory that characterises the endomorphism rings of Weyl modules in terms of functions on the space of opers.
\end{abstract}

%We describe the centre of the universal enveloping algebra at the critical level and the endomorphism rings of Weyl modules, thus establishing analogues in our setting of both the Feigin-Frenkel theorem and of a result by Frenkel and Gaitsgory.}

\section*{Introduction}
Let $\gog$ be a complex simple Lie algebra, $G$ be a complex algebraic group with Lie algebra equal to $\gog$, and $G^L$ be the Langlands dual of $G$. Frenkel and Gaitsgory have put forward a relationship between, on one side, the geometry of the ``space'' of $G^L$-local systems
on the formal disc with a possible singularity in the origin, 
and on the other, certain categories of representations of the affine Lie algebra $\hgog$ equipped with an action of the loop group $G(\mC((t)))$ \cite{FG1,FG2,FG3,FG4,FG5,FG6,FG7}. This connection, known as the \textit{local geometric Langlands correspondence}, is still largely conjectural, but some of the results it suggests have been completely proved, for example in the ``spherical" case \cite{FG6,FG7}.
More recently, Dennis Gaitsgory pointed Giorgia Fortuna in the direction of investigating the situation in which the relevant $G^L$-local systems are allowed to have more than one singular point. In this paper we begin to develop this suggestion for the case of the Lie algebra $\mathfrak{sl}(2)$.

Let us fix a coordinate $t$ around the origin of the (formal) disc. We wish to consider functions in the variable $t$, parametrised by a second variable $a$, having poles only at $t=0$ and $t=a$. Formally, we set $A= \mC[[a]]$ and introduce the $A$-algebra
$$
K_2=\mC[[t,a]][\frac 1{t(t-a)}],
$$
whose elements are the functions we are interested in.
The variables $t$ and $a$ are of very different nature here: $t$ is a local coordinate for the geometric object of interest, namely the formal disc, while $a$ should be considered as a parameter (a more general definition of the 
algebraic object we work with is given in Section \ref{sect:FormalNeighbourhoodGraphs}).

The ring $K_2$ can be equipped with a `residue' map, defined as the sum of the residues around $t=0$ and $t=a$ (see Section \ref{sec:residui}). We can use this map to 
define a structure of $A$-Lie algebra on the space
$$
\gog\otimes_\mC K_2 \oplus A\,C_2
$$
in a way that closely mimics the construction of the usual affine Lie algebra (see Sections \ref{ssec:Liealgebras} and \ref{sec:examples}). Note that, for technical reasons, we will consider a polynomial version of this algebra, but this does not affect our conclusions, see Remark \ref{oss:2completamenti}. 
We do not know who first introduced this 2-singularities version of the affine algebra, but its construction was certainly known to Gaitsgory and Raskin (see for example the notes \cite{GaitsgoryKM} and \cite{Raskin}).
We can then proceed exactly as in the case of one singularity: we first construct a suitable completion of the enveloping algebra, and then specialise $C_2$ to $-1/2$ to obtain a certain `critical level' enveloping algebra $\hU_2$. Just as in the case of the usual affine Lie algebra, the center of $\hU_2$ turns out to be nontrivial, and we show that it is generated by certain 2-variables analogues of the classical Sugawara operators.

To introduce these generalised Sugawara elements, let $J^\gra$ be a basis of $\gog$ and let $J_\gra$ be the dual basis with respect to the Killing form of $\gog$. For every integer $k$ we can then define
\begin{align*}
\Sug^{(2)}_k&=
\sum_{n\in \mZ,\gra}
:(J^\gra t^n s^n) (J_{\gra}t^{k-n}s^{k-n-1}):+
:(J^\gra t^n s^{n+1}) (J_{\gra}t^{k-n}s^{k-n}): \\
\Sug^{(2)}_{k+\frac 12}& =
\sum_{n\in \mZ,\gra}:(J^\gra t^n s^n)(J_{\gra}t^{k-n} s^{k-n}):+
:(J^\gra t^n s^{n+1})(J_{\gra}t^{k-n}s^{k+1-n}):
\end{align*}
where $s=t-a$ and the colons denote a suitable (two-variables) normal ordered product, see Definition \ref{def:nop2}.
These operators are our main object of study in Section \ref{sect:Centre}, where we prove that they are topological generators of the center of $\hU_2$ (see Theorem \ref{thm:CentreU2}).
 
In order to describe the geometric side of the correspondence, one should consider $G^L$-connections in the case of $G^L$ being an adjoint group. Hence, in our case, we take $G=\SL(2)$ and $G^L=\PSL(2)$.
We consider $G^L$-connections on the formal disc, parametrised by $a$, with 
possible singularities at $t=0$ and $t=a$. We define
$2$-opers in this context in complete analogy to local opers with one singularity (see Section \ref{ssec:opers}), namely, as particular equivalence classes of $G^L$-connections.
With this definition,  the space of $2$-opers is represented by a smooth ind-scheme $\Op_2^*$ over $\Spec A$. 
We prove the following result, which is 
a two-singularities analogue of the Feigin-Frenkel theorem \cite{FF92} for $\gog=\gos\gol(2)$. 
\begin{theorem*}[Theorem \ref{thm:CentreU2} and Theorem \ref{thm:isomorfismo}] 
	The operators $S^{(2)}_k$ are algebraically independent and topologically generate the centre of 
	the algebra $\hU_2$.  Moreover, there is an isomorphism 
	$$
	 \FF_2:\Funct(\Op_2^*)\xrightarrow{\sim} Z(\hU_2).
	$$
\end{theorem*}

One of the main ingredients in the proof of this theorem, as well as of other results in the paper, 
is the fact that 
the algebra $\hU_2$ and the space of opers satisfy the so-called factorisation property \cite{BD:chiralalgebras}. 
Indeed, both these objects are defined over $A$, 
and we can consider their restrictions to $a=0$ and to $a\neq 0$. In the first case, we get back the corresponding object in the case of one singularity.
In the second, we obtain the product of two independent copies of the one-singularity object, one ``centred around zero" and the other ``centred around $a$". For the space of $2$-opers this takes the following concrete form (see Proposition \ref{prop:FactorisationProperties}):
\begin{equation}\label{eq:fattint}
\Op_2^*|_{a=0}\simeq \Op^*_1 \quad \text{ and } \quad \Op_2^*|_{a\neq 0} \simeq \Op_t^*\times_{\Spec Q}
\Op^*_s,
\end{equation}
where $\Op_1^*$ is the usual space of opers, $\Op_t^*=\Spec Q\times_{\Spec \mC}\Op_1^*$ 
is the space of opers over the field $Q=\mC((a))$, and $\Op_s^*$ is another copy of the same space, 
whose elements are however described in terms of the coordinate 
$s=t-a$. The analogous statement for the algebra $\hU_2$ is the content of 
Lemmas \ref{lem:hS} and \ref{lem:hE2}, whose precise formulation is a little more sophisticated since it involves taking suitable completions.

\medskip

In their study of the spherical case \cite{FG6}, 
Frenkel and Gaitsgory describe the endomorphism ring of what they call a Weyl module
of $\hgog$. Let $V^\grl$ be the irreducible finite-dimensional representation of highest weight $\grl$ and consider 
it as a representation of the positive part  $\hgog^+=\gog\otimes_\mC\mC[[t]]\oplus \mC\,C$ of the affine Lie algebra by letting $C$ act by $-1/2$ and $\gog\otimes_\mC t\mC[[t]]$ by zero. The induced module of $\hgog$, denoted by $\mV_1^\grl$, is called the Weyl module of weight $\lambda$. The Weyl modules have been recognised as the fundamental objects in the category of \textit{spherical modules}, that is, those continuous representations of $\hU_1$ on which the action of 
$\gog\otimes_\mC\mC[[t]]$ integrates to an action of $G(\mC[[t]])$. An important step in understanding the
category of spherical modules is the determination of the endomorphism rings of the Weyl modules. These rings have been shown \cite{FG6} to admit a very nice description in terms of the geometry of $\Op_1^*$. We now briefly recall the precise statement.

The unramified opers, that is, those that are trivial as $G^L$-connections, are parametrised by an ind-subscheme $\Op_1^{\operatorname{ur}}$ of $\Op_1^*$. The reduced subscheme of $\Op_1^{\operatorname{ur}}$ can be shown to decompose as the disjoint union of certain schemes $\Op_1^\grl$, indexed by the dominant integral weights $\lambda$. The opers parametrised by $\Op_1^\lambda$ can be characterised in terms of the behaviour of the corresponding connection around the singularity $t=0$, see Lemma \ref{lem:opur1}. The connection between the Weyl modules $\mathbb{V}_1^\lambda$ and the schemes $\Op_1^\lambda$ is provided by \cite{FG6}, 
where the authors show that, for every $\lambda$, the Feigin-Frenkel isomorphism induces an isomorphism $\Funct(\Op_1^\grl)\simeq \End(\mV_1^\grl)$. 

This result generalises to our setting in the following way. Given two irreducible, finite-dimensional representations $V^\grl$ and $V^\mu$ of 
$\gog$, 
we construct, by analogy to the 1-singularity case, a corresponding
Weyl module $\mV_2^{\grl,\mu}$ by inducing 
the representation
$V^\grl\otimes V^\mu$ (see Definition \ref{def:weyl2} for details). We also define a corresponding space of opers 
$\Op_2^{\grl,\mu}$ as follows. By base change from $\mC$ to $Q$, we obtain subschemes $\Op_t^\grl=\Spec Q\times_{\Spec\mC}\Op_1^\grl$ of $\Op^*_t$ and $\Op_s^\mu=\Spec Q\times_{\Spec\mC}\Op_1^\mu$ of $\Op^*_s$ respectively, where as before we consider the elements of $\Op^*_s$ as functions of the variable $s=t-a$. Thanks to the second isomorphism in Equation \eqref{eq:fattint}, we may then consider the product $\Op_t^\grl\times \Op_s^\mu$ as a subscheme of $\Op^*_2|_{a\neq 0}$: the schematic closure of $\Op_t^\grl\times \Op_s^\mu$ inside $\Op_2^*$ is by definition the scheme $\Op_2^{\grl,\mu}$.
By construction, $\Op_2^{\grl,\mu}$ parametrises connections that behave like elements of $\Op_1^\grl$ around $t=0$, and like elements of $\Op_1^\mu$ around $t=a$ (at least when $a \neq 0$).
In the case of $\gos\gol(2)$, we extend the main result of \cite{FG6} to this context as follows:

\begin{theorem*}[see Theorem \ref{teo:endomorfismi}]
The action of $Z(\hU_2)$ on $\mV_2^{\grl,\mu}$ and the isomorphism $\FF_2$ of the previous theorem induce an
isomorphism 
$$
\calG_2:\Funct(\Op^{\grl,\mu}_2)\xrightarrow{\sim} \End(\mV_2^{\grl,\mu}) .
$$
\end{theorem*}
To prove this result we use the factorisation properties of these objects, i.e., we again study what happens by specialising to $a=0$ and localising to $a\neq 0$.
The factorisation properties of $\mV_2^{\grl,\mu}$ are discussed in Lemma \ref{lemma:VlambdamuOutsideAndAlongTheDiagonal}. The restriction $\Op_2^{\grl,\mu}|_{a\neq 0}$ is isomorphic to $\Op_t^{\grl}\times _{\Spec Q}\Op_s^{\mu}$ by construction. 
Finally, the restriction $\Op_2^{\lambda,\mu}|_{a=0}$ is both more interesting and more delicate to understand: in Theorem \ref{thm:restrizionediagonale} we show that 
$$
\Op^{\grl,\mu}_2|_{a=0}\simeq \coprod_{\substack{|\mu-\lambda| \leq \nu \leq \lambda+\mu \\ \nu \equiv \lambda + \mu \bmod{2}}} \Op_1^\nu
$$
where the union on the right hand side is parametrised by the set of weights, now natural numbers, appearing in the decomposition of the $\gog$-module $V^\grl\otimes V^\mu$. While some of the ingredients in the proof of this isomorphism apply to general Lie algebras $\gog$, we make use of the hypergeometric series to construct some specific elements of $\Op^{\grl,\mu}_2|_{a=0}(\mC)$ (see Section \ref{sec:ipergeo}), which restricts some of our arguments to the case $\gog = \mathfrak{sl}_2$. 

\bigskip

In \cite{FG7}, Frenkel and Gaitsgory use their results on the endomorphism rings of the Weyl modules to prove that the category of spherical representations is equivalent to the category 
of sheaves over the space of unramified opers. 
A similar result is expected to hold in our generality, and indeed some of the ingredients of \cite{FG7}, such as the notion of semi-infinite cohomology, generalise well to our context. We plan to further investigate this problem in the future.

It is also expected that these results should hold for arbitrary quasi-simple groups $G$. 
Most of our constructions are given at this level of generality, and many of our results are either easy to extend to this broader context, or are already proved in the correct generality. To go forward in this direction, we believe the most difficult problem would be to extend Theorem \ref{thm:restrizionediagonale}. Indeed, we think one should obtain this result
as a consequence of a generalisation of Theorem \ref{teo:endomorfismi}, while in our approach we used Theorem \ref{thm:restrizionediagonale} as an intermediate step in the proof of Theorem \ref{teo:endomorfismi}.
We have also been informed by D.~Gaitsgory that a completely different approach to these questions is currently being developed by S.~Raskin.

\bigskip

The paper is organised as follows. In Section \ref{sez:algpre} we discuss some basic 
algebraic constructions, including in particular a basis  
$\{u_n, v_n\}$ for the Laurent series in two variables (see equation \eqref{eq:defuv}) which is technically important in the whole paper. 
In Section \ref{sect:Opers} we study the space of 2-opers: we introduce the scheme 
$\Op_2^{\grl,\mu}$ in a somewhat
utilitarian way, construct several explicit elements of $\Op_2^{\grl,\mu}(\mC)$ by means of the classical hypergeometric series, and prove Theorem \ref{thm:restrizionediagonale}. In Section \ref{sec:Liealgebras}
we describe the affine Lie algebra $\hgog_2$ and its completion. We also study the factorisation properties of $\hU_2$ (Lemmas \ref{lem:hS} and \ref{lem:hE2}) and introduce 
some filtrations of this algebra. In Section \ref{sect:Centre} we introduce 
our version of the Sugawara operators, we prove that they are central (Proposition \ref{prop:LkAreCentralPartOne}), and we show that they topologically generate the center of $\hU_2$ (Theorem \ref{thm:CentreU2}). In Section \ref{sez:FF2}
we identify the centre of $\hU_2$ as the ring of functions of the space of $2$-opers, 
and in Section \ref{sez:Weylmodules} we study the Weyl modules $\V_2^{\lambda, \mu}$, proving Theorem \ref{thm:EndoVlambda}.

\section{Algebraic preliminaries}\label{sez:algpre}

In our study of functions on the punctured disc we will need to consider several closely related rings, which we will also need to equip with suitable topologies.
 
In the case of one singularity we consider the ring $\mC[[t]]$ and its quotient field $\mC((t))$. We equip $\mC[[t]]$ and $\mC((t))$ with the topology induced by the ideal $(t)$ (namely, the ideals $(t^n)$ of $\mC[[t]]$ form a fundamental system of neighbourhoods of the origin). 

In order to deal with the case of two singularities, we introduce the ring $A=\mC[[a]]$ and its fraction field $Q=\mC((a))$. 
We think the variable $a$ as being a parameter, and use a separate variable $t$ as a coordinate around the origin of the (formal) disc. We consider functions that 
have poles in $0$ and in $a$, and we also set $s=t-a$. Here a ``function'' is considered to be a formal Laurent series. The relevant rings for the case of two singularities are slightly more complicated, and will be introduced in Section \ref{sect:FormalNeighbourhoodGraphs}. In the next few sections we focus on the the ring of formal Laurent series in two variables.

\subsection{Formal Laurent series}\label{sez:base} 

We consider the rings 
$Q[[t]], Q[[s]]$ endowed with the topologies induced by the ideals $(t)$ and $(s)$, respectively. We consider similar topologies on $Q((t))$ and $Q((s))$.
We denote by $K_2$ the ring of formal Laurent series in two variables, that is, $K_2=\mC[[t,s]][\frac{1}{ts}]$. We equip the rings $\mC[[t,s]]$ and $K_2$ with the topology induced by the ideal $(ts)$, and we consider them as $A$-algebras, where $a=t-s$. 
We now introduce two $A$-bases of $\mC[[t,s]][\frac1{ts}]$ that turn out to be computationally handy:
for every integer $n$ we set
\begin{equation}\label{eq:defuv}
u_n = t^ns^n, \quad v_n=t^ns^{n+1}
\end{equation}
and
\begin{equation}\label{eq:defxy}
x_n=t^ns^n, \quad y_n=t^{n+1}s^n.
\end{equation}
We also consider the subring $R_2$ of $K_2$ given by the $A$-span of the elements $u_n,v_n$ for $n\in \mZ$, or equivalently the $A$-span of the elements $x_n,y_n$. The following proposition is easy to prove.

\begin{proposition}\label{prp:base}
The elements $u_n$, $v_n$ for $n\leq -1$ form an $A$-basis of $K_2/\mC[[t,s]]$. More generally, the elements $u_n, v_n$ for $n\leq m-1$ form an $A$-basis of 
$K_2/u_m\mC[[t,s]]$,
and for $m \geq 1$ the elements $u_n, v_n$ with $0 \leq n\leq m-1$ are an $A$-basis of $\mC[[t,s]]/(u_m)$.
In particular, the ring $R_2$ is dense in $K_2$. 
\end{proposition}

Our main tool to deal with these rings will be the expansion and specialisation maps that we now define. We have natural inclusions
\[
\mC[[t,s]]\subset Q[[t]] \subset Q((t)),
\]
and similar ones with respect to the variable $s$.
Since $a$ is invertible in $Q((t))$, the inclusion $\mC[[t,s]]\subset Q((t))$ extends to an inclusion $\mC[[t,s]][a^{-1}]\subset Q((t))$. Moreover, both $t$ and $s$ are invertible in $Q((t))$,
thus the inclusion $\mC[[t,s]][a^{-1}]\subset Q((t))$ further extends to an injective map
\[
E_t:K_2[a^{-1}] \lra Q((t)),
\]
which we call the \textit{expansion map} with respect to the variable $t$. The same construction yields an injective ring homomorphism $E_s : K_2[a^{-1}]\lra  Q((s))$, and both $E_t$ and $E_s$ are $Q$-linear. Finally, we define
$$
E:K_2[a^{-1}]\lra Q((t))\times Q((s))
$$
as the product of $E_t$ and $E_s$. The ring $K_2[a^{-1}]$ is an integral domain, so $E$ cannot be an isomorphism; however, it is not far from being one, as the following lemma shows.
\begin{lemma}\label{lem:exp}
The morphism $E$ is continuous and injective, and its image is dense in $Q((t))\times Q((s))$. Moreover, the restriction of $E$ to $R_2[a^{-1}]$ has dense image as well.
\end{lemma}
\begin{proof} Notice that $E(t^ns^n)=(t^n \alpha,s^n \beta)$, where $\alpha,\beta$ are invertible elements of $Q[[t]]$ and $Q[[s]]$ respectively. This easily implies that $E$ is continuous. We now prove that it is injective with dense image.
As $1=(t-s)/a$, the ideals $(s)$ and $(t)$ are relatively prime in $\mC[[t,s]][a^{-1}]$, hence we obtain
$$
\frac{\mC[[t,s]][a^{-1}]}{\left( (ts)^n \right)} 
\cong \frac{\mC[[t,s]][a^{-1}]}{(t^n)} \times \frac{\mC[[t,s]][a^{-1}]}{(s^n)}
\cong \frac{Q[[t]]}{(t^n)}\times \frac{Q[[s]]}{(s^n)}\ .
$$
We thus get a dense embedding
\[
\mC[[t,s]][a^{-1}] \hookrightarrow \varprojlim_{n} \frac{\mC[[t,s]][a^{-1}]}{(ts)^n} \cong \varprojlim_n \left( \frac{Q[[t]]}{(t^n)} \times \frac{Q[[s]]}{(s^n)} \right) \cong Q[[t]] \times Q[[s]],
\]
and upon inverting $ts$ on both sides we obtain a dense embedding $K_2[a^{-1}] \to Q((t)) \times Q((s))$. Notice that inverting $ts=t(t-a)$ in $Q[[t]]$ is equivalent to inverting $t$, and similarly for $Q[[s]]$. The first claim follows since this embedding agrees with $E$.
To prove the claim about $R_2[a^{-1}]$ notice that, by Proposition \ref{prp:base}, this ring is dense in $K_2[a^{-1}]$.  
\end{proof}

Throughout the paper we will also need to study the special situation in which $t$ and $s$ coincide, that is, when $a=0$. We denote by $\Sp$ any map that specialises $a$ to zero: for example, there are natural maps
\[
\begin{array}{cccc}\Sp : & A &  \to &  \mathbb{C} \\
& f(a) & \mapsto & f(0)
\end{array}
\]
and 
\[
\begin{array}{cccc}\Sp : & K_2 &  \to & \mC((t)) \\
& f(t,s) & \mapsto & f(t,t);
\end{array}
\]
the latter also restricts to $\Sp : \mC[[t,s]] \to \mC[[t]]$.
The following is easy to check:
\begin{lemma}\label{lem:specializ}
The map $\Sp$ induces isomorphisms between $K_2/(t-s)$ and $\mC((t))$, and between $\mC[[t,s]] /(t-s)$ and $\mC[[t]]$.
\end{lemma}

\subsection{Residues}\label{sec:residui}

By analogy to the case of the classical residue $\res:\mC((t))\lra \mC$, that takes $f(t)$ to the coefficient of $t^{-1}$ in $f(t)$, we introduce the following maps:
\begin{align*}
\res_t&:Q((t))\lra Q & &\text{coefficient of } t^{-1}\\  
\res_s&:Q((s))\lra Q & &\text{coefficient of } s^{-1}\\  
\res_{t,s}&:Q((t))\times Q((s)) \lra Q & &\res_{t,s}(f,g)=\res_t(f)+\res_s(g)\\
\res_2&:K_2\lra Q & & \res_{2}(f)=\res_{t,s}(E(f))
\end{align*}
The basic properties of these maps are discussed in the following lemma, whose proof is straightforward.
\begin{lemma}\label{lem:residues}
The following hold:
\begin{enumerate}[a)]
 \item $\res_t,\res_s,\res_{t,s},\res_2$ are continuous maps with respect to the natural topologies on the source and the discrete topology on $Q$;
 \item $\res_t,\res_s,\res_{t,s}$ are $Q$-linear and $\res_2$ is $A$-linear;
 \item the image of $\res_2$ is contained in $A$;
 \item $\Sp(\res_2 f)=\res \Sp(f)$; 
 \item for all $i,j$ we have
\[
\res_2{t^is^j}=\bigg(\binom{i}{-j-1}+(-1)^{i+j+1}\binom{j}{-i-1}\bigg)\:a^{i+j+1}.
 \]
In particular, for all $i\geq 0$ and $j<0$ we have 
$\res_2{t^is^j}=\binom{i}{-j-1}\,a^{i+j+1}$, 
and for all $j\geq 0$ and $i<0$ we have
 $\res_2{t^is^j}= \binom{j}{-i-1}\,(-a)^{i+j+1}$. Furthermore, $\res_2{t^is^j}=0$ if $i+j\leq-2$. \end{enumerate}
\end{lemma}

Moreover, the bases $u_n, v_n$ and $x_n, y_n$ are dual to each other with respect to $\Res_2$, in the sense that we have
\begin{align*}\label{eq:dualitabase}
\Res_2(u_n\cdot x_{m})&=0&
\Res_2(u_n\cdot y_{-m-1})&=\grd_{n,m}\\
\Res_2(v_n\cdot x_{-m-1})&=\grd_{n,m}&
\Res_2(v_n\cdot y_{m})&=0.
\end{align*}
As a consequence we obtain that, for all $f\in K_2$, we have
\begin{equation}\label{eq:resduale}
f(t,s) = \sum_{n \in  \mZ} \Res_2(fu_n)y_{-n-1} + \Res_2(fv_n)x_{-n-1}.
\end{equation}
A similar formula is obtained by exchanging the role of $\{u_n, v_n\}$ with that of $\{x_n, y_n\}$.

\begin{remark}\label{rem:Res2Derivatives}
For all $n$ we have $u_n'=n(y_{n-1}+v_{n-1})$ and $v_n'= (2n+1)u_n-nav_{n-1}$. Using the formulas in Lemma \ref{lem:residues}, this implies in particular
\begin{align*}
\Res_2(u_m u_n')   & = 2n\delta_{m,-n}, & \Res_2(v_m u_n') &=-an \delta_{m,-n},\\
\Res_2( y_m u_n' ) & = an \delta_{m,-n},& \Res_2(v_m v_n') &=(2n+1)\delta_{m,-1-n} + a^2n\delta_{m,-n}.
\end{align*}
\end{remark}

\subsection{Formal neighbourhood of two graphs in the formal disc}\label{sect:FormalNeighbourhoodGraphs} 

In this section we give an explicit description of the formal neighbourhood of two graphs in the formal disc. We will use this in Section \ref{ssec:opers} to get an algebraic representation of the space of 2-opers.
Let $\mathbb{D} = \Spec \mC[[t]]$ be the formal disc. 
Let $\Rtest$ be an $A$-algebra and define $X=\Spec \Rtest \xrightarrow{\psi} \mathbb{D}$ by $\psi^\sharp(f(t))=f(a)\in \Rtest$ for all $f(t)\in \mC[[t]]$.
Inside the product $X \times \mathbb{D}$ we consider the union $\Gamma$ of the two graphs $\Gamma_\psi, \Gamma_0$ of 
the morphisms $\psi$ and $x \mapsto 0$. The formal neighbourhood $\mathbb{D}_\psi$ of $\Gamma$ in $X \times \mathbb{D}$ is then 
given by the spectrum of the ring
\[
\Rtest[2] := \varprojlim_{n} \frac{\Rtest \otimes_{\mC} \mathbb{C}[[t]]}{(t(t-a))^n} = \varprojlim_{n} \frac{\Rtest[[t]]}{(t(t-a))^n}.
\]
\begin{remark}\label{rmk:SeriesVsPolynomials}
 Notice that replacing $\mC[[t]]$ with $\mC[t]$ in the above construction gives exactly the same ring $\Rtest[2]$.
\end{remark}

We now give an alternative description of this ring.
We define $R/\!/t,s/\!/$ to be the subring of $R[[t,s]]$ given by those elements $f = \sum_{i,j \geq 0} f_{ij} t^i s^j$ with the following property: for every $i$ there are only finitely many $j$ for which $f_{ij} \neq 0$, and symmetrically for every $j$ there are only finitely many $i$ for which $f_{ij} \neq 0$.
Similarly to the notation introduced in Section \ref{sez:base} for the ring $K_2$, we define elements $u_n, v_n$ of $R/\!/t,s/\!/ [\frac{1}{ts}]$ as in formula \eqref{eq:defuv}.

Consider the quotient $T := \displaystyle \frac{R/\!/t,s/\!/}{(t-s-a)}$. For every $n \geq 0$ there is a natural map $T \to \frac{R[[t]]}{(t(t-a))^n}$ sending $t \mapsto t, s \mapsto t-a$, so we get a homomorphism $T \to R[2]$. We now prove that this is an isomorphism. 
We begin with the following lemma:
\begin{lemma}Let $\Rtest$ and $T$ be as in the discussion above.
For every $n \geq 0$, the elements $\{u_k, v_k \bigm\vert 0 \leq k \leq n-1\}$ form an $\Rtest$-basis of the quotient $T/(u_n)$. Their images in $\frac{R[[t]]}{((t-a)t)^n}$ also form of an $\Rtest$-basis of the latter.
\end{lemma}
\begin{proof}
 The map $T \to \frac{R[[t]]}{((t-a)t)^n}$ is clearly surjective and factors via $T/(u_n)$. Hence it suffices to show that the images in $T/(u_n)$ of the elements $u_k, v_k$ are $\Rtest$-linearly independent and generate $T/(u_n)$ over $\Rtest$.
To check linear independence it suffices to notice that $u_k, v_k$ map to monic polynomials in $t$ whose degrees are all different and less than $2n$. The fact that they generate follows immediately by induction on $n$.
\end{proof}

The fact that $T \to \Rtest[2]$ is an isomorphism now follows easily: its kernel is trivial (because it is given by the intersection of the ideals $(u_n)$, which is clearly $(0)$), and it is surjective because $T$ is complete with respect to the topology generated by the ideals $(u_n)$ (hence $T$ contains any series of the form $\sum_{k \geq 0} s_ku_k + s'_k v_k$ with $s_k, s'_k \in \Rtest$).

\begin{remark}
 This argument shows in particular that the elements of $T$ can be represented (in a unique way) as series $\sum_{k \geq 0} s_k u_k +s'_k v_k$ with $s_k, s'_k \in \Rtest$. 
\end{remark}

In Section \ref{sect:Opers} we will also consider the ring $R(2) := R[2][\frac{1}{t}, \frac{1}{s}]$ and its quotients by its $R[2]$-submodules $(ts)^n R[2]$. A similar proof shows that, for a fixed integer $n$, the set $\{u_k, v_k \bigm\vert k \in \mathbb{Z}, k \leq n-1\}$ forms an $\Rtest$-basis of this quotient.
Notice moreover that this description (together with Proposition \ref{prp:base}) implies that  in the case $R=A$ we have isomorphisms
\[  A[2] \simeq \mC[[t,s]], \quad A(2) \simeq K_2. \]

Similarly to what we did in Section \ref{sez:base} we now study the two special cases $a=0$ and $a \in \Rtest^\times$. As in that case, we have expansion and specialisation maps
$$
E:R(2)\lra R_a((t))\times R_a((s)) 
\qquad
\Sp:R(2)\lra \frac R{(a)} ((t)).
$$
The following proposition is easy to show by arguments similar to those used in Section \ref{sez:base}.
\begin{proposition}\label{prop:FactorisationPropertiesR} The following hold:
\begin{itemize}
 \item If $a=0$ in $R$, then the specialisation map $\Sp$ induces isomorphisms $R[2]\simeq\Rtest[[t]]$ and $R(2)\simeq R((t))$. 
 \item If $a$ is invertible in $R$, then the expansion map induces isomorphisms $R[2] \simeq R[[t]]\times R[[s]]$ and $R(2)\simeq R((t))\times R((s))$.
\end{itemize}
\end{proposition}

The derivative with respect to $t$ naturally induces canonical $R$-linear derivations on the rings $R[2]$ and $R(2)$. We will denote both these derivations by $f \mapsto \partial f$ or by $f \mapsto \dot f$. Similar derivations can also be defined on the rings $R_a((t)), R/(a), R_a((s))$. Notice that the expansion and specialisation maps commute with these derivations.

\subsection{Checking isomorphisms generically and along the diagonal}

Our strategy to prove many of the results in this paper is to study the situation `outside of the diagonal' (i.e.~when $a$ is invertible) and along it ($a=0$). The next lemma allows us to obtain global information from these two special cases:

\begin{lemma}\label{lem:isoMN}
Let $M, N$ be two $A$-modules and $\grf:M\lra N$ be a morphism of $A$-modules. Then
\begin{enumerate}[a)]
 \item if $M$ is flat and $\grf_a:M[a^{-1}]\lra N[a^{-1}]$ is injective, then $\grf$ is injective.
 \item if $N$ is flat, $\grf_a:M[a^{-1}]\lra N[a^{-1}]$ is surjective, and $\overline{\grf} :M/aM\lra N/aN$ is injective, then $\grf$ is surjective.
\end{enumerate}
In particular, if $M$ and $N$ are flat, $\grf_a:M[a^{-1}]\lra N[a^{-1}]$ is an isomorphism, and $\overline{\grf}:M/aM\lra N/aN$ is injective, then $\grf$ is an isomorphism.
\end{lemma}

\begin{proof}
\begin{enumerate}[a)]
\item As $M$ is flat, we can consider it as a submodule of $M[a^{-1}]$. The injectivity of
$\grf:M\lra N$ then follows immediately from the assumption on $\grf_a$.

\item By flatness we can again consider $N$ as a submodule of $N[a^{-1}]$. Let $y\in N$. 
As $\grf_a$ is surjective, there exists $n\geq 0$ and $x\in M$ such that
$y=\grf(x/a^n)$, or equivalently $a^n y =\grf(x)$. We prove that $y$ is in the image of $\grf$ by induction on $n$. 
For $n=0$ there is nothing to prove.
If $n>0$ then $\overline{\grf}(\overline{x})=0$, where $\overline{x}$ is the image of $x$ in $N/aN$. Since $\overline{\grf}$ is injective, we have $\overline{x}=0$, or equivalently $x=ax'$ for some $x' \in N$. As $N$ is flat we deduce $a^{n-1}y=\grf(x')$, and we conclude by induction.
\end{enumerate}
\end{proof}

\section{Opers}\label{sect:Opers}

We start by recalling the definition of an oper, due to Beilinson and Drinfeld \cite{BD:quantization},  
in the more usual case of one singularity. We follow \cite[Part I]{FG2}.
Let $G$ be the simple, simply connected complex algebraic group with Lie algebra $\gog$, so that $G^L$ is a group of adjoint type (with Lie algebra $\mathfrak{g}^L$).
The opers we are interested in are particular equivalence classes of $G^L$-connections on the formal 
punctured disc. We fix a local coordinate $t$ for the formal disc and let $\Rtest$ be a $\mC$-algebra. 
We consider $\Rtest$-families of connections on the punctured disc of the form 
$$ \nabla = \frac d {dt} + M, $$
where $M \in  \gog^L \otimes \Rtest((t)) $. There is a natural action of the loop group of $G^L$ on the set of connections: 
identifying a connection with the corresponding matrix $M$, for $H\in G^L(\Rtest((t)))$ we set
\begin{equation}\label{eq:HMazione}
 H \cdot M := \Ad_H (M) - \dot{H}\,H^{-1}. 
\end{equation}
To define the class of opers we fix a maximal torus $T$ of $G^L$ 
and a Borel subgroup $B$ containing $T$, and let $\got\subset \gob \subset \gog^L$ be the corresponding Lie algebras. We also fix root vectors 
$e_\gra\in \gog^L$. 
The set of connections for which $M$ is of the form
\begin{equation}\label{eq:Oper}
p + C, \quad \text{ with } C \in \mathfrak{b} \otimes \Rtest((t))   \text{ and } p\in \sum_{\gra\text{ simple}} e_{-\gra} \, R((t))^*,
\end{equation}
is stable under the action of the loop group $B(\Rtest((t)))$.

\begin{definition}\label{def:Oper} 
An \textbf{oper} over a $\mC$-algebra $\Rtest$ is a $B(\Rtest((t)))$-equivalence class of connections with matrix $M$ of the form \eqref{eq:Oper}. We denote by $\Op_1^*$ the functor from the category of $\mC$-algebras to the category of sets 
which assigns to a $\mC$-algebra $\Rtest$ the set of opers over $\Rtest$.
\end{definition}

When $\mathfrak{g}=\mathfrak{g}^L=\mathfrak{sl}_2$, so that $G^L=\PSL_2$, the space of opers has a simple description. 
Namely, up to $B(\Rtest((t)))$-equivalence, every oper can be represented as
\begin{equation}\label{eq:RepresentingOpers}
\frac{d}{dt} + \begin{pmatrix}
0 & f \\
1 & 0
\end{pmatrix}
\end{equation}
for a certain (unique) $f \in \Rtest((t))$. We may therefore think of $\PSL_2$-opers as elements of $\Rtest((t))$. 
In particular, the functor $\Op_1^*$ is representable by the ind-scheme
\[
\Op_1^* := \varinjlim_{N \to \infty} \operatorname{Spec} \mathbb{C}[\zz_i : i \geq -N],
\]
where $\zz_i$ corresponds to the coefficient of $t^i$. 
We denote the scheme $\operatorname{Spec} \mathbb{C}[\zz_i : i \geq -N]$ by $(\Op_1^*)_{\geq -N}$. We will also need the 
following variants, where it will be important to keep track of the name of the variable: as functors,
\[
\begin{array}{cccc}
 \Op^*_t : & Q-\text{alg} & \to & \mathbf{Set} \\
 & \Rtest & \mapsto & \Rtest((t))
 \end{array}, \quad \begin{array}{cccc}
 \Op^*_s : & Q-\text{alg} & \to & \mathbf{Set} \\
 & \Rtest & \mapsto & \Rtest((s)).
 \end{array}
\]
These functors are clearly representable by the ind-schemes
\[
\Op_t^* := \Op_1^* \times_{\Spec \mC} \Spec Q \quad \text{ and }\quad\Op_s^* := \Op_1^* \times_{\Spec \mC} \Spec Q.
\]
The coordinates on $\Op_t^*$, respectively $\Op_s^*$, will be denoted by $\zzt_i$, respectively $\zzs_i$.
We will also use the notations $(\Op_t^*)_{\geq -N}$, $(\Op_s^*)_{\geq -N}$ with the obvious meaning.

\subsection{The functor $\Op^*_2$}\label{ssec:opers}
We now consider a generalisation of the construction in the previous paragraph, based on the idea of building a factorisation space of opers $\Op^*_n$ (see for example the PhD thesis of the first author of this paper \cite[Section 3.5]{GiorgiaPhD}).
In the special case $G^L=\PSL_2$ and $n=2$ we are interested in, we give 
an explicit description of this space and of the associated (ind-)scheme.

As in Section \ref{sect:FormalNeighbourhoodGraphs}, fix an affine test scheme $X=\Spec \Rtest$ and a map 
$X \xrightarrow{\psi} \mathbb{D}$; equivalently, $\Rtest$ is a $\mC[[a]]$-algebra.
We consider implicitly that we have a second map $X \to \mathbb{D}$, 
that we take to be the trivial map carrying all of $X$ to $0 \in \mathbb{D}$. 
We now consider meromorphic connections on $\mathbb{D}$, parametrised by $X$, allowing for singularities in $0$ (considered as the graph of the constant map) and along the graph of $\psi$. 
To give an algebraic description of these connections, we write them in the form $\frac{d}{dt}+M$, where $M\in \Mat_{2\times 2}(\Rtest(2))$ and $\Rtest(2)$ is as in Section \ref{sect:FormalNeighbourhoodGraphs}. 
The action of the group $\PSL(2,R(2))$ on the set of connections written in this form is still given 
	by formula \eqref{eq:HMazione}, where $\dot H$ denotes the canonical derivation on the ring $R(2)$ introduced at the end of Section \ref{sect:FormalNeighbourhoodGraphs}.
\begin{definition}\label{def:2Oper} A $2$-oper is a $B(\Rtest(2))$-conjugacy class of connections $\frac{d}{dt}+M$ where $M$ is of the form
\[
p + C, \quad \text{ with } C \in \mathfrak{b} \otimes \Rtest(2) \text{ and } p\in \sum_{\alpha \text{ simple}} e_{-\gra} \, R(2)^* .
\]
\end{definition}
As in the case of one singularity, one shows without difficulty that every 2-oper has a unique representative of the form \eqref{eq:RepresentingOpers}, where $f$ is now an element of $\Rtest(2)$. 

\begin{remark}
 In the fundamental case $X=\mathbb{D} \xrightarrow{\operatorname{id}} \mathbb{D}$, the graph of $\psi$ is a subscheme of 
 $\mathbb{D} \times \mathbb{D}$, where the two discs have two very different roles: the latter is the space on which the 
 connection is defined (and in particular the derivation $\frac{d}{dt}$ in the formula above is with respect to the 
 variable of this disc), while the former is a space of parameters. To avoid confusion, we identify the two discs with 
 $\Spec \mathbb{C}[[a]]$ and $\Spec \mathbb{C}[[t]]$ respectively. Notice that this is consistent with the notation 
 employed in Section \ref{sect:FormalNeighbourhoodGraphs}, where $a$ denotes the pullback of the coordinate on $\mathbb{D}$ to the scheme $X$. 
 For this reason, we consider $X$ as a scheme over $\Spec \mathbb{C}[[a]]$.
\end{remark}
Based on the above discussion, we define the functor
\[
\Op_2^*: \, A-\text{alg}\lra \mathbf{Set} \quad \text{ by }\quad  \Op^*_2(\Rtest)=\Rtest(2).
\]
By what we have explained in Section \ref{sect:FormalNeighbourhoodGraphs}, this functor is represented by an ind-scheme: indeed, it is the limit of the functors
$$
\left(\Op_2^*\right)_{\geq -N}(\Rtest)= \frac{1}{(ts)^N} \Rtest[2], 
$$
each of which is represented by
$$
\Spec A[a_i, b_i:i\geq -N],
$$
where $a_i$ corresponds to the coefficient of $u_i$ and $b_i$ to the coefficient of $v_i$ (see Section \ref{sect:FormalNeighbourhoodGraphs}).
Notice that we have $\Op^*_2(A)\simeq \mC[[t,s]][\frac{1}{ts}]$ and $\Op^*_1(\mC)=\Op^*_2(\mC)\simeq \mC((t))$, 
where $\mC$ is given the structure of an $A$-algebra by sending $a$ to $0$.
Finally, we remark that the geometric analogue of Proposition \ref{prop:FactorisationPropertiesR} 
immediately yields the following:
\begin{proposition}\label{prop:FactorisationProperties}
The following hold:
\begin{enumerate}
 \item  Let $\Op^*_2|_{a=0}$ be the subfunctor of $\Op^*_2$ given by $\Op^*_2 \times_{\Spec A} \Spec \mC$.
The map 
\[
 \Op^*_2|_{a=0}(\Rtest) \ni f \mapsto \Sp(f) \in \Op^*_1(\Rtest)
\]
is an isomorphism between $\Op^*_2|_{a=0}$ and $\Op^*_1$. For all integers $N$, this isomorphism restricts to 
an isomorphism between the two schemes $(\Op^*_2)_{\geq-N}|_{a=0}$ and $(\Op^*_1)_{\geq-2N}$.

\item Let $\Op^*_2|_{a \neq 0}$ be the subfunctor of $\Op^*_2$ given by $\Op^*_2 \times_{\Spec A} \Spec Q$.
The map 
\[
 \Op^*_2|_{a \neq 0}(\Rtest) \ni f \mapsto (E_t(f), E_s(f)) \in  \Op^*_t(\Rtest) \times \Op^*_s (\Rtest)\ 
\]
is an isomorphism between 
$\Op^*_2|_{a \neq 0}$ and $\Op^*_t \times_{\Spec Q} \Op^*_s$.
For all integers $N$, this isomorphism restricts to 
an isomorphism between the two schemes $(\Op^*_2)_{\geq-N}|_{a\neq 0}$ and $(\Op^*_t)_{\geq-N} \times_{\Spec Q} (\Op^*_s)_{\geq-N}$.
\end{enumerate}
\end{proposition}

We denote by $\Sp:\Op^*_2|_{a=0}\lra \Op_1^*$ and $E:\Op^*_2|_{a\neq 0}\lra \Op_t^*\times_{\Spec Q} \Op^*_s$ the two isomorphisms described in the previous proposition.
We denote also by $\Sp$, respectively $E$, the restriction of these isomorphisms to 
isomorphisms between $(\Op^*_2)_{\geq-N}|_{a=0}$ and 
$(\Op^*_1)_{\geq-2N}$, respectively 
between $(\Op^*_2)_{\geq-N}|_{a=0}$ $(\Op^*_t)_{\geq-N} \times_{\Spec Q} (\Op^*_s)_{\geq-N}$.
We will need some information on these isomorphisms that we will obtain using the coordinates $a_i, b_i, \zz_i, \zzt_i, \zzt_s$ introduced above. 
In the formulas below we write $a_i,b_i$ also for the restrictions of these functions to the closed subscheme $a=0$ or to its open complement. Denoting by $E^\sharp$ and $\Sp^\sharp$ the pullback maps on 
coordinate rings, we have 
\begin{equation}\label{eq:azbz}
 \Sp^{\sharp}(\zz_{2i})=a_i\ ,\qquad \Sp^{\sharp}(\zz_{2i+1})=b_i
\end{equation}
and
\begin{equation}\label{eq:xyuv}
\begin{array}{c}
\displaystyle E_t^{\sharp}(\zzt_n)= \sum_{i=-\infty}^n  \binom{i}{n-i}  (-a)^{2i-n}a_i  + \sum_{i=-\infty}^n  \binom{i+1}{n-i}  (-a)^{2 i-n+1} b_i \ ,\\
\displaystyle E_s^{\sharp}(\zzs_n)= \sum_{i=-\infty}^n  \binom{i}{n-i}   a^{2 i-n}a_i   + \sum _{i=-\infty}^{n-1} \binom{i}{n-i-1} a^{2 i-n+1} b_i \ .
\end{array}
\end{equation}
Setting $a_i=b_i=0$ for $i<-N$, the same formulas also describe
an isomorphism between the coordinate rings of 
 $(\Op^*_t)_{\geq-N} \times_{\Spec Q} (\Op^*_s)_{\geq-N}$ and
 $(\Op^*_2)_{\geq-N}|_{a\neq 0}$. 
Finally, these isomorphisms are homogeneous if the variables are graded as follows.

\begin{definition}\label{def:Degrees}
We assign degree $i+2$ to the variables $\zzt_i$, $\zzs_i$ and $\zz_i$, degree 
$-1$ to $a$, degree $2i+2$ to $a_i$, and degree $2i+3$ to $b_i$. With this choice, the above equations are homogeneous.
\end{definition}

\subsection{Functions and derivations on opers}\label{sez:funzionioper}

In this section we discuss some technical aspects of the definition of the space of regular functions on an ind-scheme, with a focus on our cases of interest $\Op_1^*$ and $\Op_2^*$. We also spell out the topological structure of these spaces of functions, and describe how the vector fields on the punctured disc act on them.

Let $S$ be a ring and $X=\limind X_n$ be the inductive limit of the affine $S$-schemes $X_n$. We denote by $\Funct_S(X)=\Funct(X)$ the $S$-algebra of regular functions on $X$, that is, $\Hom_S(X,\Spec S[t])$. This is the projective limit of the coordinate rings  of the schemes $X_n$. 
If $S\lra S'$ is a ring homomorphism, it is not true in general that $\Funct(\Spec S'\times_{\Spec S}X )=S'\otimes_S\Funct(S)$ (this may fail even when $S \to S'$ is a localisation). A similar problem arises for the product of two ind-schemes. In our situation 
this will not cause any real issues, but it forces us to include some 
topological clarifications. The construction of $\Funct(X)$ as a projective limit endows it with a natural topology, where we consider all the coordinate rings of the schemes $X_n$ to be equipped with the discrete topology. In order to discuss products we will also need the tensor product topology. Recall that the tensor product $X \otimes Y$ of two topological vector spaces $X, Y$ is naturally equipped with the topology for which the neighbourhoods of zero are of the form $U \otimes Y + X \otimes V$, where $U$ and $V$ are neighbourhoods of zero in $X$ and $Y$ respectively.

\smallskip

In our setting, the relevant spaces of regular functions have very simple descriptions: for example, one has
\begin{align*}
\Funct(\Op_2^*) = & \{f\in A[[\dots,b_{-1},a_0,b_0,a_1,\dots]]\st \text{for all $N \in \mathbb{Z}$, the series $f(\dots,0,a_N,b_N,\dots)$}\\
& \quad \text{obtained by
	specialising $a_i$ and $b_i$ to $0$ for $i<N$ is a polynomial} \},
\end{align*}
and similarly for the other spaces of opers. In particular $\Funct(\Op_2^*)$ is torsion-free, hence flat over $\mC[[a]]$.

Let now $Y=\limind Y_n$ and let $\grf:Y\lra X$ be a morphism of ind-schemes, obtained as a limit of morphisms of schemes $\grf_n:Y_n\lra X_n$. The map $\grf$ induces a natural continuous map $\grf^\sharp:\Funct(X)\lra\Funct(Y)$, obtained by precomposition. 
For example, the isomorphism of ind-schemes $\Sp : \Op_2^*|_{a=0} \to \Op_1^*$ of Proposition \ref{prop:FactorisationProperties} induces by pullback an isomorphism  
	$\Sp^\sharp:\Funct(\Op_1^*)\lra \Funct(\Op^*_2|_{a=0})$. The map
	$\calSp:\Funct(\Op_2^*)\lra \Funct(\Op^*_1)$ is defined as the composition of the restriction map from $\Funct(\Op_2^*)$ to $\Funct(\Op_2^*|_{a=0})$ with $(\Sp^\sharp)^{-1}$.
We similarly define an expansion map $\calE$ as the composition of the inclusion of $\Funct(\Op_2^*)[a^{-1}]$ in 
$\Funct(\Op_2^*|_{a\neq 0})$ with the isomorphism $(E^{-1})^{\sharp}:\Funct(\Op_2^*|_{a\neq 0}) \lra \Funct(\Op_t^*
\times_{\Spec Q} \Op_s^*)$ induced by the second part of Proposition \ref{prop:FactorisationProperties}.

It is easy to check that the specialisation map $\calSp$ is surjective, that it induces an isomorphism between $\Funct(\Op_2^*)/a\Funct(\Op_2^*)$ and $\Funct(\Op_1^*)$, and that the natural topology on $\Funct(\Op_1^*)$ agrees with quotient topology.

The expansion map $\calE:\Funct(\Op_2^*)[a^{-1}]\lra \Funct(\Op_t^*
\times_{\Spec Q} \Op_s^*)$ is injective and has dense image. 
There are thus two natural topologies on $\Funct(\Op_2^*)[a^{-1}]$: the topology induced by $\Funct(\Op_t^* \times_{\Spec Q} \Op_s^*)$ through the map $\calE$, and the topology for which the neighbourhoods of zero are given by $U_n[a^{-1}]$, where $U_n$ is a neighbourhood of zero in $\Funct(\Op_2^*)$. One may check that these two topologies coincide. 
Finally, the multiplication map 
$\Funct(\Op_t^*)\otimes_Q\Funct(\Op_s^*)\lra 
\Funct( \Op_t^*\times_{\Spec Q}\Op_s^* )$ is injective and has dense image. The topology induced on $\Funct(\Op_t^*)\otimes_Q\Funct(\Op_s^*)$ by this embedding can be shown to agree with the tensor product topology. 

\subsubsection{Action of the vector fields of the punctured disc}\label{ssez:campivettorifunzioni}
The action of the vector fields of the formal disc on $\Funct(\Op_1^*)$ extends
to the case of two singularities without changes.

\begin{definition}\label{dfn:Derivazioni}
We denote by $\Der_1$ the Lie algebra  $\mC((t))\partial$ of continuous vector fields on the punctured disc with the usual bracket $[u\partial,v\partial]=(uv'-v'u)\partial$. 
The action of the group of automorphisms of the formal disc determines an action of the algebra $t\mC[[t]]\partial$ on $\Funct(\Op_1^*)$. This action extends to all of $\Der_1$, and is given explicitly by the formula (see for example formula (3.5-11) in \cite{Frenkel_Langlands_loop_group})
\begin{equation}\label{eq:azioneDer2}
\big(u\,\partial \cdot F\big) (f) = -dF_f [2\dot u \, f + u \, \dot f-\frac 12\dddot u],
\end{equation}
where $f \in \Op_1^*(\mC)$, $F \in \Funct(\Op_1^*)$, and  $u \in \mC((t))$.
We similarly denote by $\Der_2$ the Lie algebra $K_2\partial$, where $\partial$ is the only continuous $A$-linear derivation
of $K_2$ such that $\partial t=\partial s =1$, introduced at the end of Section \ref{sect:FormalNeighbourhoodGraphs}. We define an action of $\Der_2$ on $\Funct(\Op^*_2)$ by the same formula.
\end{definition}

These definitions have obvious analogues in the case of $\Op^*_t$, $\Op^*_s$, and their product. 

\bigskip

We have an expansion map $E:\Der_2\lra \Der_t\times \Der_s$  
given by $E(u \partial)=(E_t(u) \partial, E_s(u) \partial)$ and a specialisation map $\Sp:\Der_2\lra \Der_1$ given 
by $\Sp(u \partial)=\Sp(u) \partial$.
These are homomorphisms of Lie algebras since $E$ and $\Sp$ commute with the derivations. Moreover, for all $F$ in 
$\Funct(\Op_2^*)$ and $u\partial \in \Der_2$ we have
\[
\calE(u\partial\cdot F)=E(u\partial)\cdot \calE(F)
\qquad
\calSp(u\partial\cdot F)=\Sp(u\partial)\cdot \calSp(F).
\]

We will use some concrete information about this action. Recall from Section \ref{sect:FormalNeighbourhoodGraphs}
that $\{u_i,y_i\}$ (defined as in Equation \eqref{eq:defxy}) is a topological $R$-basis of $R(2)$, for any $A$-algebra $R$. We denote by $\gra_i$, $\grb_i$ the coordinates with respect to this basis, so that
$\alpha_i = a_i-ab_i$ and $\beta_i=b_i$. The following lemma is then easy to obtain by a direct computation.

\begin{lemma}\label{lemma:DerivativesOfAlphaBeta}
The action of $\Der_2$ on the variables $\alpha_i, \beta_i$ is given by the following formulas, where $m$ is any integer:
\begin{enumerate}
\item $\begin{aligned}[t]
u_{m}\partial \alpha_i 
& = -(2i+2m+1)\beta_{i-m} + a(m+i+1)\alpha_{i-m+1} \\
& \quad -m(m-1)(2m-1)a\delta_{i,m-2}-\frac12 m(m-1)(m-2)a^3\delta_{i,m-3}.\\
\end{aligned}$

\item 
$\begin{aligned}[t]u_{m}\partial \beta_i
& =  -2(m+i+1)\alpha_{i-m+1}-(m+i+1)\,a\,\beta_{i-m+1} \\
& \quad +2m(m-1)(2m-1)\delta_{i,m-2}+m(m-1)(m-2)a^2\delta_{i,m-3}.\\
\end{aligned}
$

\item $\begin{aligned}[t]
v_{m} \partial \alpha_i & = -2(m + i + 1) \alpha_{i-m}  + (m + i  + 1) a \beta_{i-m}  -(m + i + 1)a^2 \alpha_{i-m+1}\\
& \quad + \frac{1}{2}m(m-1)(m-2) a^4 \delta_{i, m-3} + 
\frac{3}{2} m(m-1)(2m-1) a^2 \delta_{i,m-2} \\
& \quad + m(2m-1)(2m+1) \delta_{i,m-1}\\
\end{aligned}$

\item $\begin{aligned}[t]
v_{m} \partial \beta_i 
& = -(2m+2i+3)\beta_{i-m}+(m+i+1)\,a\, \alpha_{i-m+1}   \\
&\quad -m(m-1)(2m-1)a\delta_{i,m-2}-\frac12 m(m-1)(m-2)a^3\delta_{i,m-3}.
\end{aligned}$
\end{enumerate}
\end{lemma}

\subsection{Unramified opers and the scheme $\Op_1^\lambda$}
Following Frenkel and Gaitsgory \cite[Section 6.12]{FG1}
we recall, focusing on the special case of $G^L=\PSL_2$, the definition and some elementary properties of the scheme of unramified opers 
and of the schemes $\Op_1^\lambda$. 
Let $\Op^{\operatorname{ur}}_1$ be the subfunctor of $\Op_1$ which corresponds to unramified opers: for every $\mC$-algebra
$\Rtest$ the set $\Op^{\operatorname{ur}}_1(\Rtest)$ consists of those $f\in \Rtest((t))$ for which there exists $H \in \PSL_2(\Rtest((t)))$ such that 
\begin{equation}\label{eqn:urop1} 
\begin{pmatrix} 0 & f \\ 1 & 0 \end{pmatrix}= -\dot{H}H^{-1}.
\end{equation}
One can prove \cite[Sect.~6.12]{FG1} that $\Op_1^{\operatorname{ur}}$ is an ind-subscheme of $\Op_1^*$ which is not reduced. In this paper we are interested
in its reduced version, that we denote by $\Op_1^{\operatorname{int}, \operatorname{reg}}$. This is a countable union of subschemes
of $(\Op_1^*)_{\geq -2}$ which can be described functorially as follows: if $\Rtest$ is a $\mC$ algebra such that $\Spec \Rtest$ is connected, then
\begin{align*}
\Op_1^{\text{int,reg}}(\Rtest)=\{f=\sum_{i\geq -2} \zz_i\, t^i \in \Op_1^*(R) \st\; &\zz_{-2}\in \mC \text{ and there exists } H(t) \in \PSL_2(\Rtest((t))) \\
&\text{ such that equation \ref{eqn:urop1}
is satisfied}\}.
\end{align*}
If $f$ is an unramified oper as in the definition above, the coefficient $\zz_{-2}$ turns out to be of the form $A_\lambda=\tfrac {\lambda^2}4 +\tfrac\lambda 2$ for 
some $\lambda\in \mN$. For such a $\lambda$ we define
$$
\Op_1^{\lambda}(\Rtest)=\{f=\sum_{i\geq -2} \zz_i\, t^i \in \Op^{\text{int,reg}}_1(\Rtest) \st \zz_{-2}=A_\lambda\}\ .
$$
Notice that this scheme is denoted by $\Op^{\lambda,\operatorname{reg}}$ in \cite{FG6}.
With this definition we have $\Op^{\text{int,reg}}_1=\bigsqcup_{\lambda\in \mN} \Op_1^{\lambda}$. Notice that $\Op_1^\lambda$ is a smooth subscheme of 
the scheme $(\Op_1^*)_{\geq -2}$. The following lemma is a particular case of \cite[Lemma 1]{FG6}, formulated 
in terms of coordinates; its proof for $\PSL_2$ is elementary. 

\begin{lemma}\label{lem:opur1}
Let $\lambda \in \mN$ and let $f=\sum_{i\geq -2}\zz_it^i$ be an element of $R((t))$.
\begin{enumerate}[\indent a)]
\item If $\zz_{-2}=A_\lambda$, then $f\in \Op_1^\lambda(\Rtest)$ if and only if there exists a solution of the equation $\ddot y=f y$ of the form $t^{-\tfrac {\lambda} 2}\grf(t)$, with 
$\grf(t) \in \Rtest[[t]]$ and $\grf(0)=1$.
\item If $\zz_{-2}=A_\lambda$, then $f\in \Op_1^\lambda(\Rtest)$ if and only if there exists $\psi\in \Rtest[[t]]$ such that 
$$f=\frac{A_\lambda}{t^2}-\frac{\lambda}t \psi+\psi^2+\dot \psi \ .$$
\item If $\zz_{-2}\in \mC$, then $f\in \Op_1^{\operatorname{int}, \operatorname{reg}}(\Rtest)$ if and only if there exists 
$H \in \SL_2(R((t^{1/2})))$ such that equation \eqref{eqn:urop1} is satisfied;
\item If $\Rtest=\mC$ and $f$ defines a meromorphic function in a neighbourhood
of $0$, then $f\in \Op_1^{\operatorname{int}, \operatorname{reg}}(\mC)$ if and only if the monodromy of the equation $\ddot y=f y$ around $0$ is equal to $\pm \operatorname{Id}$.
\end{enumerate}
\end{lemma}

We can now obtain an equation for $\Op^\lambda_1$ as a subscheme of $(\Op^*_1)_{\geq 2}$.
Writing $f=\sum _{i\geq -2}\zz_i\,t^i$ and $\psi=\sum_{k \geq 0} \psi_k t^k$ for a generic element of $\Rtest[[t]]$, by part b) of Lemma \ref{lem:opur1} we see that $f$ belongs to $\Op_1^{\lambda}(\Rtest)$ if and only if $\zz_{-2}=A_\lambda$ and
the following system in the variables $\psi_0,\dots,\psi_{\lambda-1}$ admits a solution in $\Rtest$:
\begin{equation}\label{eqn:Plambda}
\begin{cases}
\begin{aligned}
\zz_{-1} & = -\grl \psi_0  \\
\zz_{0} & = \psi_0^2-(\grl-1) \psi_1  \\
\zz_{1} & = 2\psi_0\psi_1-(\grl-2) \psi_2  \\
&\vdots\\
\zz_{k}& = \sum_{i+j=k} \psi_i\psi_j - (\grl-1-k) \psi_{k+1}\\
&\vdots\\
\zz_{\grl-2}& = \sum_{i+j=\grl-2} \psi_i\psi_j -  \psi_{\grl-1}\\
\zz_{\grl-1}& = \sum_{i+j=\grl-1} \psi_i\psi_j \\
\end{aligned}
\end{cases}
\end{equation}
Expressing the $\psi_i$ with $i<\lambda$ recursively from the first $\lambda$ equations we can rewrite the last equation in the form
$$
\zz_{\lambda-1}=P_{\lambda}(\zz_{-1},\dots,\zz_{\lambda-2}) \ .
$$
It is easy to see that, if we assign degree $i+2$ to the variable $\zz_i$ as in Definition \ref{def:Degrees}, the polynomial $P_\lambda$ is homogeneous of degree $\lambda+1$ indeed, all the equations in \eqref{eqn:Plambda} become homogeneous if we give weight $i+1$ to the variable $\psi_i$.
Hence $\Op_1^{\lambda}\subset (\Op_1^* )_{\geq -2}$ is defined by the equations
\begin{equation}\label{eq:EquationsForOpLambda}
\begin{cases}
\zz_{-2} = A_\lambda \\ 
\zz_{\lambda-1} = P_\lambda(\zz_{-1}, \ldots, \zz_{\lambda-2})
\end{cases}
\end{equation}
We now make a remark on the structure of these equations. 
\begin{lemma}\label{lemma:monomiPlambda}
Each monomial of total degree $\lambda+1$ in the variables $\zz_{-1},\dots,\zz_{\lambda-2}$  appears in $P_\lambda$ with nonzero coefficient.
\end{lemma}
\begin{proof}
If we set $\gamma_i =-\zz_i$ we see that the system \eqref{eqn:Plambda} is equivalent to 
\begin{align*}
\psi_0 & = * \, \gamma_{-1}  \\
\psi_1 & = * \, \psi_0^2     + *\,  \gamma_0  \\
\psi_2 & = * \, \psi_0\psi_1 + *\,  \gamma_1  \\
&\vdots\\
\psi_{k+1}& = *\, \sum_{i+j=k} \psi_i\psi_j + *\,  \gamma_k \\
&\vdots\\
\psi_{\grl-1}& = *\, \sum_{i+j=\grl-2} \psi_i\psi_j + *\,  \gamma_{\grl-2} \\
0& = *\, \sum_{i+j=\grl} \psi_i\psi_j +   \gamma_{\grl-1} \\
\end{align*}
where each $*$ denotes a strictly positive coefficient. The statement follows by an easy induction.
\end{proof}

The equation describing $\Op^\grl_1$ shows in particular that this scheme is smooth. 
We now describe its tangent space without using these equations. 

\begin{lemma}\label{lem:tanspnr}
Let $\Rtest$ be a $\mC$-algebra
and let $R[\varepsilon]$ be the ring of dual numbers over $R$. 
Let $f$ be an element of $\Op^{\lambda}_1(\Rtest)$ 
and let $H \in \SL_2(R((t^{1/2})))$ be such that  $\dot H = -F H$, where $F=\left(\begin{smallmatrix} 0 & f \\ 1 & 0 \end{smallmatrix}\right)$.
Given $g \in R((t))$, the element $f+\gre g$ belongs to $\Op^\grl_1(\Rtest[\gre])$ if and only if  
$g\in t^{-1}R[[t]]$ and 
$$
\Res \big(H^{-1}\cdot \begin{pmatrix} 0 & g \\ 0 & 0 \end{pmatrix}\cdot H\big)=0\ ,
$$
where the residue of a matrix is defined as the matrix of the residues of its entries.
\end{lemma}
\begin{proof}
The necessity of the condition $g\in t^{-1}R[[t]]$ is clear from the inclusion $\Op_1^{\grl}(R)\subset A_\grl t^{-2}+t^{-1}R[[t]]$. For the second condition 
we consider the system 
$$
\dot W=-\begin{pmatrix} 0 & f+\gre g \\ 1 & 0 \end{pmatrix} W \ .
$$
We know that this system has a solution $W\in \SL_2(\Rtest[\gre]((t^{1/2})))$ if and only if $f+\gre g $ is in $\Op^\grl_1(\Rtest[\gre])$.
Writing $W=U+\gre V$, the system is equivalent to 
\begin{equation}\label{eq:SystemEpsilon}
\dot U = -F U \quad \text{ and }\quad  \dot V= -FV-\begin{pmatrix} 0 & g \\ 0 & 0 \end{pmatrix}U \ ,
\end{equation}
while the condition $\det(W)=1$ translates to $\det U=1$ and $\Tr(U^{-1}V)=0$.
By changing coordinates we can assume $U=H$. Set $V=HX$. The second equation in \eqref{eq:SystemEpsilon} is then equivalent to 
$$
\dot X=-H^{-1}\begin{pmatrix} 0 & g \\ 0 & 0 \end{pmatrix}H,
$$
while $\Tr(U^{-1}V)=0$ becomes $\Tr(X)=0$. The solubility of these equations is equivalent to the condition in the statement, since a formal power series can be formally integrated if and only if its residue vanishes.
\end{proof}

\subsection{The scheme $\Op_2^{\lambda, \mu}$}\label{sec:Oplm}

Let $\lambda, \mu$ be natural numbers. We now define a subscheme $\Op_2^{\lambda, \mu}$ for the case of two singular points, which is the analogue of the subscheme $\Op_1^\lambda$ of the previous section. By Proposition \ref{prop:FactorisationProperties}, 
we have an isomorphism
\[
 \Op_2^*|_{a \neq 0} \cong \Op_t^* \times_{\Spec Q} \Op_s^*,
\]
and the right hand side contains the subscheme $\Op_t^\lambda \times_{\Spec Q} \Op_s^\mu$,
where 
\[
\Op_t^\lambda=\Spec Q \times _{\Spec \mC} \Op_1^\lambda\subset \Op^*_t
\]
and $\Op_s^\mu$ is defined similarly. 
We define $\Op_2^{\lambda, \mu}$ as the schematic closure of the image of $\Op_t^\lambda \times_{\Spec Q} \Op_s^\mu$ in $\Op_2^*$. 
Notice that, since $\Op_1^\lambda, \Op_1^\mu$ are contained in $(\Op_1^*)_{\geq -2}$, 
we also have that $\Op_2^{\lambda, \mu}$ is a closed subscheme of $(\Op_2^*)_{\geq -2}$.
By definition, if $\Rtest$ is a $Q$-algebra, then
$$
\Op_2^{\lambda,\mu}|_{a\neq 0}(\Rtest)=\{f\in (ts)^{-2} \Rtest[2]\st E_t(f) \in \Op^{\lambda}_t(\Rtest) \text{ and } E_s(f) \in \Op^{\mu}_s(\Rtest)\}\ .
$$
This implies that if $\Rtest$ is an $A$-algebra and $a$ is not a zero divisor in $\Rtest$ then 
\begin{equation}\label{eq:DescriptionOp2LambdaMu}
\Op_2^{\lambda,\mu}(\Rtest)=\{f\in (ts)^{-2} \Rtest[2]\st E_t(f) \in \Op^{\lambda}_t(\Rtest_a) \text{ and } E_s(f) \in \Op^{\mu}_s(\Rtest_a)\}\ .
\end{equation}
By definition, the subscheme 
$\Op_t^\lambda \times_{\Spec Q} \Op_s^\mu$ of $(\Op_t^*)_{\geq -2} \times_{\Spec Q} (\Op_s^*)_{\geq -2}$ 
is smooth and irreducible, and is given by the equations
\begin{align}
   \label{eq:lmxy1} \zzt_{-2} &= A_\lambda, &  \zzs_{-2} &= A_\mu, \\
   \label{eq:lmxy2} \zzt_{\lambda-1} &= P_\lambda(\zzt_{-1}, \ldots, \zzt_{\lambda-2}), &  \zzs_{\mu-1} &= P_\mu(\zzs_{-1}, \ldots, \zzs_{\mu-2}),
\end{align}
where $(\Op_t^*)_{\geq -2} = \Spec Q[\zzt_i, i \geq -2]$ and $(\Op_s^*)_{\geq -2} = \Spec Q[\zzs_i, i \geq -2]$, and $\zzt_i$ (resp.\ $\zzs_i$)
is the coefficient of $t^i$ (resp.\ $s^i$). These equations are written with respect to the natural coordinates on $\Op^*_t$ and $\Op^*_s$. However, the coordinates we have chosen for $\Op_2^*$ are the coefficients $a_i,b_i$ with respect to the basis $u_i,v_i$. 
To obtain the equations defining $\Op_2^{\lambda,\mu}|_{a\neq 0}$ we thus need to use the change of coordinates afforded by \eqref{eq:xyuv}.
In particular, equations \eqref{eq:lmxy1} translate into 
\begin{equation}\label{eq:lmab1}
 a_{-2} = a^2 A_\mu\quad \text{ and } \quad b_{-2} = a(A_\mu- A_\lambda) \ .  
\end{equation}
Let now 
\[
\tilde{P}_\lambda = E_t^{\sharp} \left( \zzt_{\lambda-1}-P_{\lambda} (\zzt_{-1}, \dots, \zzt_{\lambda-2}) \right), \quad \tilde{P}_\mu = E_s^{\sharp} \left( \zzs_{\mu-1}-P_{\mu}(\zzs_{-1}, \dots, \zzs_{\mu-2}) \right) \ .
\]
These are polynomials in the variables $a_i,b_i$, and by equation \eqref{eq:xyuv} we see 
that they only involve the variables $a_{-2},b_{-2},\dots,a_{\nu-1},b_{\nu-1}$, where $\nu=\max\{\lambda,\mu\}$. Furthermore, these polynomials have coefficients in $\mC[a^{\pm1}]$, and are homogeneous of degrees respectively  $\lambda+1$ and $\mu+1$.

Recall that the coordinate ring of $(\Op_2^*)_{\geq-2}|_{a\neq 0}$ is $Q[a_i,b_i:i\geq -2]$. The ideal $I^{\lambda,\mu}_a\subset Q[a_i,b_i:i\geq -2]$ generated by $\tilde P_\lambda$, $\tilde P_\mu$, and $a_{-2}-a^2A_\mu$, $b_{-2}-a(A_\mu-A_\lambda)$
(corresponding to the equations \eqref{eq:lmab1}) defines $\Op^{\lambda,\mu}_2|_{a\neq 0}$ inside $(\Op^*_2)_{\geq -2}|_{a \neq 0}$. The ideal
$I^{\lambda,\mu}_a$ is prime (since $\Op_t^\lambda$ and $\Op_s^\mu$ are irreducible and smooth), and 
its intersection $I^{\lambda,\mu}$ with $A[a_i,b_i:i\geq -2]$
is the ideal defining $\Op_2^{\lambda,\mu}$ in $(\Op^*_2)_{\geq -2}$.

The scheme $(\Op^*_2)_{\geq -2}$ is defined not only over $\mC[[a]]$, but also over $\mathbb{C}[a]$.
We claim that $\Op_2^{\lambda,\mu}$ is also defined over $\mC[a]$. As noticed above, the generators of $I_a^{\lambda,\mu}$ have coefficients
in $\mC[a^{\pm 1}]$. The claim then follows from the following lemma, where we take $S$ to be the coordinate ring of $(\Op_2^*)_{\geq -2}$ (seen as a scheme over $\mathbb{C}[a]$), and $J$ to be the ideal defining 
$\Op_2^{\lambda,\mu}$ as a subscheme of $(\Op_2^*)_{\geq -2}|_{a\neq 0}$ (itself considered as a scheme over $\mC[a^{\pm1}]$). 

\begin{lemma}\label{lem:definitosuCa}
Let $S$ be a flat $\mC[a]$-algebra and let $\hat S = S\otimes _{\mC[a]}\mC[[a]]$.
By flatness we have $S\subset \hat S \subset \hat S_a$ and $S\subset S_a \subset \hat S_a$.
Let $J$ be an ideal of $S_a$ and let $\hat J$ be its extension to $\hat S_a$. Let 
$\hat I=\hat J\cap \hat S $ and $I=S\cap \hat J$.
Then $\hat I=I\otimes_{\mC[a]}\mC[[a]]$.
\end{lemma}

\begin{proof}
We notice first that $I=J\cap S$, hence if $f\in S$ and $af\in I$ then $f\in I$. 
Let $H=I\otimes_{\mC[a]}\mC[[a]]$, that is the extension of $I$ in $\hat S$. 
We have $H\subset \hat I$, and we only need to prove the other inclusion. 

We first claim that if $f$ is in $\hat S$ and $af$ is in $H$, then $f\in H$. Indeed, let $af=\sum f_i\ell_i$ with $f_1,\dots,f_n\in I$ and $\ell_i\in\mC[[a]]\otimes_{\mC[a]}S$.
Write $\ell_i=g_i+ah_i$ with $g_i\in S$ and $h_i\in \hat S$. Hence
$$
af=g+ah
$$
where $g=\sum g_i f_i\in I$ and $h=\sum h_i f_i\in H$. 
Hence $g \in I \subset S$ is divisible by $a$ (in $\hat S$, hence also in $S$), and by the remark at the beginning of this proof we have $g'=g/a\in I$. Hence 
$f=g'+h$ is in $H$ as claimed.

We can now prove the inclusion $\hat{I} \subset H$. Let $f\in \hat I\subset \hat J$: by definition, there exists $m\geq 0$ such that $a^mf\in H$, and by what just noticed we get $f\in H$.
\end{proof}

\subsection{The hypergeometric case}\label{sec:ipergeo}
In this section we study a particular class of elements of $\Op_2^{\lambda,\mu}(A)$. We consider 
$f$ of the form
$f= b t^{-2} + c s^{-2} + d t^{-1}s^{-1}$ with $b,c,d\in \mC$. From equation \eqref{eq:lmab1}
we immediately get that, if $f$ is an element of $\Op_2^{\lambda,\mu}(A)$ of this form, then 
\begin{equation}\label{eq:HypergeometricElements}
f=a^{2}A_\mu u_{-2}+a(A_\mu-A_\grl)v_{-2}+ B\,u_{-1}=\frac{A_\grl}{t^2}+\frac{A_\mu}{s^2}+\frac{B-A_\grl-A_\mu}{ts}
\end{equation}
for some $B \in \mathbb{C}$.
If we set $y=t^{-\grl/2}s^{-\mu/2}\grf$, then the equation $\ddot y=f y$ is equivalent to 
\begin{equation}\label{eq:ipergeo}
t (t-a) \ddot \grf - ((\grl+\mu)t-a\grl)\dot\grf +(A_{\grl+\mu}-B) \grf=0,
\end{equation}
which is the hypergeometric equation with singularities in $0,a,\infty$. 
Assume that $\grl\leq \mu$ and that $B=A_\nu$ with $\nu=\grl+\mu-2j$ for some $j=0,1,\dots,\grl$: we 
will use equation \eqref{eq:DescriptionOp2LambdaMu} and Lemma \ref{lem:opur1} a) to show that for these values of $B$ the element $f$ introduced above belongs to $\Op_2^{\grl,\mu}(A)$. In fact, these are the only values of $B$ for which $f$ is an element of $\Op_2^{\grl,\mu}(A)$:
this can be checked directly, but will also become clear from the results of the next sections. To rewrite the equation in a more traditional form, we make the change of variables $t=aw$ and set 
$$
\gra = \frac{\nu-\mu-\grl}2=-j \ , \qquad \grb =\frac{-\grl-\mu-\nu-2}2=j-\grl-\mu-1  \ , \qquad\grg=-\grl\ .
$$
Equation \eqref{eq:ipergeo} then becomes
\begin{equation}\label{eq:ipergeo2}
w (w-1) \ddot \psi + \big((\gra+\grb+1)w-\grg)\dot\psi + \gra\grb \psi=0 \ ,
\end{equation}
with Riemann scheme given by 
$$
\begin{array}{ccccc}
 0 & & 1 & & \infty \\
\hline
 0 & & 0 & & \gra \\
 1+\grl & & 1+\mu  & & \grb \\
\end{array}\ .
$$
A solution of the hypergeometric equation with exponent $0$ near $0$ can be written explicitly in terms of the hypergeometric series \cite[§2.1.1]{HigherTranscendentalFunctions}. Usually 
one assumes that $\grg$ is not in $\mathbb{Z}_{\leq 0}$, but the hypergeometric series also makes sense if $\grg$ is in $\mathbb{Z}_{\leq 0}$, provided that $\gra$ is an integer with $\grg \leq \gra \leq 0$, which is exactly our situation. In particular, the polynomial 
$$
F(\gra,\grb,\grg,w)=\sum_{n=0}^{-\gra}\frac{(\gra)_n(\grb)_n}{(\grg)_n}\, \frac{w^n}{n!}\ ,
$$
where $(x)_n=x(x+1)\cdots(x+n-1)$, is a solution of \eqref{eq:ipergeo2} of exponent $0$ near $0$.
Set
\[
\grf(t)= (-a)^{\mu/2}
\sum_{n=0}^{j}\frac{(\gra)_n(\grb)_n}{(\grg)_n}\, \frac{t^n \, a^{-n}}{n!}
\]
and $y=t^{-\grl/2}s^{-\mu/2}\grf(t)$.
Notice that $s^{-\mu/2}\grf(t)$ tends to $1$ as $t \to 0$, so $y \sim t^{-\grl/2}$ as $t \to 0$. In addition, $y$ is a solution of $\ddot y=f y$ in the ring $Q[[t]]$, so by Lemma \ref{lem:opur1} a) we have that 
$E_t(f)\in \Op_t^{\grl}(Q)$.
A similar argument shows that $E_s(f)\in \Op_s^{\mu}(Q)$, hence $f\in \Op_2^{\grl,\mu}(A)$.
We summarise the discussion of the previous paragraphs in the following lemma:
\begin{lemma}\label{lemma:HypergeometricElements}
Let $\lambda \leq \mu$ be natural numbers and let $\nu=\lambda+\mu-2j$ for some $j=0,\ldots,\lambda$. Then the set $\Op_2^{\lambda,\mu}(A)$ is nonempty, since it contains in particular the element $f$ given in Equation \eqref{eq:HypergeometricElements} (with $B=A_\nu$). This element satisfies $\Sp(f)=A_{\nu} t^{-2}$.
\end{lemma}

\subsection{The restriction to the diagonal of $\Op_2^{\lambda, \mu}$}
We now study the scheme $\Op_2^{\lambda, \mu}|_{a=0}$. Recall from Proposition \ref{prop:FactorisationProperties}
that we have an isomorphism $\Sp$ between $\Op^*_2|_{a=0}$ and $\Op^*_1$. In this section we identify $\Op_2^{\lambda, \mu}|_{a=0}$ with a subscheme of $\Op^*_1$ using this isomorphism, and prove the following result.

\begin{theorem}\label{thm:restrizionediagonale} Let $\lambda, \mu$ be natural numbers. Then we have an equality of schemes
 \[
  \Op_2^{\lambda, \mu}|_{a=0} = \coprod_{\substack{|\mu-\lambda| \leq \nu \leq \lambda+\mu \\ \nu \equiv \lambda + \mu \bmod{2}}} \Op_1^\nu\ .
 \]
\end{theorem}
To show this theorem we will first prove that the set of $\mC$-points of the two schemes are 
equal, and then prove that the scheme on the left hand side is smooth and reduced. From this we will deduce that the scheme structures coincide as well. Without loss of generality, we will assume $\lambda \leq\mu$. 

Before delving into the details of the proof we recall that the equations which define $
\Op^{\lambda,\mu}_2$ in $(\Op^*_{2})_{\geq -2}$ involve only the variables $a_{-2},b_{-2},
\dots, b_{\mu-1}$, hence, by \eqref{eq:azbz}, the equations which define 
$\Op_2^{\lambda, \mu}|_{a=0}$ in $(\Op_1)_{\geq -4}$ involve only the variables 
$\zz_{-4},\dots, \zz_{2\mu-1}$. Moreover, \eqref{eq:lmab1} shows that
$\zz_{-3}=\zz_{-4}=0$, so we may consider $\Op_2^{\lambda,\mu}|_{a=0}$ as a subscheme of $(\Op_1^*)_{\geq -2}$. 

\subsubsection{The equation for $\zz_{-2}$}
In this section we describe a particular equation in the ideal defining the scheme $\Op_2^{\grl,\mu}|_{a=0}$. This equation will be a polynomial in the variable $\zz_{-2}$.
Recall that from Equation \eqref{eq:lmab1} in $\Op_2^{\grl,\mu}$ we have
$a_{-2} = a^2 A_\mu$ and $b_{-2} = a(A_\mu- A_\lambda)$: we will consider these values as fixed, and regard
the other equations as polynomials in the variables $a_i,b_i$ for $i\geq -1$. 

\begin{definition}\label{def:flambda}
By equation \eqref{eq:xyuv}, for every $n \geq -2$ the polynomial $E_t^{\sharp}(a^{n+2}\zzt_n)$ in the variables $a_i, b_i$ 
has coefficients in $A$ (recall that we consider $a_{-2}=a^2A_\mu$ and $b_{-2}=a(A_\mu-A_\lambda)$ as fixed). By the homogeneity of $P_\lambda$ we get 
\[
Q_\grl :=a^{\grl+1}\tilde P_{\grl}\in A[a_i, b_i:i\geq -1] \qquad \text{ and } \qquad f_\lambda:=\calSp(Q_\lambda)\in \mC[\zz_i:i\geq -2]\ .
\]
In particular, $Q_\grl$ is an element of $I^{\grl,\mu}$, and $f_\grl$ is an element in the ideal defining the scheme $\Op_2^{\grl,\mu}|_{a=0}$.
\end{definition}

\begin{lemma}\label{lemma:PropertiesOfFlambda}Let $\grl,\mu$ be arbitrary natural numbers. Then
the only variable involved in the polynomial $f_\lambda$ is $\zz_{-2}$.  
The degree of $f_\lambda$ is $\lambda+1$, and the coefficient of $\zz_{-2}^{\lambda+1}$ does not depend on $\mu$.
\end{lemma}
\begin{proof}
We want to compute the part of the polynomial $Q_\grl$ which has degree zero in the variable $a$.
By equations \eqref{eq:xyuv}, for every $n \geq -1$ we have
\[
\calSp( E_t^{\sharp}(a^{n+2} \zzt_n) ) = {-2 \choose n+2}  (-1)^{n} A_\mu + {-1 \choose n+2} (-1)^{n+1} (A_\mu-A_\lambda) + {-1 \choose n+1} (-1)^{n} \zz_{-2},
\]
and the first statement follows since $\zz_{-2}$ is the only variable involved in the above equality. By homogeneity, it is clear that the degree of $f_\lambda$ is at most $\lambda+1$. 
To show that the degree is precisely $\grl+1$ we need to prove that the coefficient of $a_{-1}^{\lambda+1}$ is nonzero in the polynomial 
$Q_\grl$. Notice that this monomial appears only in the expansion of $(\zzt_{-1})^{\lambda+1}$, 
whose coefficient in $P_\lambda$ is nonzero by Lemma \ref{lemma:monomiPlambda}. By the formula above (with $n=-1$),
the coefficient of $\zz_{-2}^{\lambda+1}$ in $f_\lambda$ is $(-1)^{\lambda+1}$, 
multiplied by the coefficient of $(\zzt_{-1})^{\lambda+1}$ in $-P_\lambda$ (which is clearly independent of $\mu$). 
\end{proof}

From now on, since $f_\lambda$ depends only on $\zz_{-2}$, we will denote it by $f_\lambda(\zz_{-2})$.

\begin{lemma}\label{lemma:RootsOffLambda1}
Assume $\lambda \leq \mu$. 
The $\lambda+1$ roots of $f_\lambda(\zz_{-2})$ are $\zz_{-2} = A_{\mu+\lambda-2i}$ for $i=0,\ldots,\lambda$, and these numbers are all distinct.
\end{lemma}

\begin{proof}
For each $i=0,\ldots,\lambda$, Lemma \ref{lemma:HypergeometricElements} provides an element of $\Op_2^{\grl,\mu}(A)$ that is sent to $A_{\mu+\lambda-2i} t^{-2}$ under the specialisation map
$A\lra\mC$. As $f_\lambda$ must vanish on each of these specialisations, this shows that $A_{\mu+\lambda-2i}$ is a root of $f_\lambda$ for each $i=0,\ldots,\lambda$.
Moreover, these $\grl+1$ values are all distinct, since $A_\nu = A_{\nu'}$ if and only if $\nu + \nu' = -2$. It follows that this is a complete list of the roots of $f_\grl$.
\end{proof}

\subsubsection{The $\mC$-points and $\mC[\gre]$-points of $\Op_2^{\lambda, \mu}|_{a=0}$}
We now turn to the determination of the complex points and of the tangent spaces of $\Op_2^{\lambda, \mu}|_{a=0}$. By definition, the set  $\Op_2^{\lambda, \mu}|_{a=0}(\mC)$ is equal to 
$\Op_2^{\lambda, \mu}(\mC)$, where $\mC$ is considered as a $\mC[[a]]$-algebra with the trivial action of $a$.
It will be useful to recall that $\Op_2^{\lambda,\mu}$ is defined over $\mC[a]$, and we may also
consider $\mC$ as a $\mC[a]$-algebra where $a$ acts trivially. The projection 
$\mC[a]\to \mC$ which sends $a$ to $0$  induces a map 
\[
\Op_2^{\lambda, \mu}(\mC[a])\lra \Op_2^{\lambda, \mu}(\mC)\subset (\Op^*_1)_{\geq -2}(\mC),
\]
which is simply the restriction of $\Sp$ to $\Op_2^{\lambda, \mu}(\mC[a])$. This map has a right inverse, given by the natural injection of $\mC$ in $\mC[a]$, hence in particular it is surjective. 

\begin{lemma}\label{lem:Cpunti1} Let $\lambda\leq \mu$ be natural numbers. Then
 \[
  \Op_2^{\lambda, \mu}|_{a=0} (\mC) \subset \Op_1^{\operatorname{int},\operatorname{reg}}(\mC)\ .
 \]
\end{lemma}
\begin{proof}
By surjectivity of the map $\Op_2^{\lambda, \mu}(\mC[a])\lra \Op_2^{\lambda, \mu}(\mC)=\Op_2^{\lambda, \mu}|_{a=0}(\mC)$, it is enough to prove that if $f\in \Op_2^{\lambda, \mu} (\mC[a])$  then $\Sp(f)\in \Op_1^{\operatorname{int}, \operatorname{reg}}(\mC)$.
Recall that the equations defining $\Op_2^{\lambda,\mu}$ involve only the variables $a_{-2},\dots, b_{\mu-1}$, while the equations defining $\Op_2^{\lambda,\mu}|_{a=0}$ involve only the variables $\zz_{-2},\dots, \zz_{2\mu-1}$.  

Let $f=\sum_{i \geq -2} (a_iu_i+b_iv_i)$ be an element of $\Op_2^{\grl,\mu}(\mC[a])$ and let $f_{\operatorname{tr}}$ be the truncation $\sum_{-2\leq i \leq \mu-1 } (a_iu_i+b_iv_i)$. Then $f_{\operatorname{tr}}$ is still an element of $\Op_2^{\grl,\mu}(\mC[a])$ (because the coefficients $a_i, b_i$ with $i \geq \mu$ do not intervene in the equations defining $\Op_2^{\grl,\mu}$), and similarly $\Sp(f)$ is in $\Op_2^{\grl,\mu}|_{a=0}(\mC)$ if and only 
$\Sp(f_{\operatorname{tr}})$ is.
Hence we can assume that $f$ is a finite $\mC[a]$-linear combination of 
the elements $u_{-2},\dots,v_{\mu-1}$, hence in particular $f\in \tfrac1{t^2s^2}\mC[t^{},s^{}]$. 

We consider $f$ as a meromorphic function in the variables $t$ and $a=t-s$. For every fixed value of $a \in \mathbb{C}$, the function $f$ has poles only in 
$t= 0$ and $t=a$, and these poles are of order at most two (this also holds for $a=0$ by equations
\eqref{eq:lmab1}). Moreover, by Lemma \ref{lem:opur1}, we see that for $a\neq 0$ the monodromy around $0$ of the system 
$\dot v= \left(\begin{smallmatrix} 0 & f \\ 1 & 0 \end{smallmatrix}\right) v$ is 
equal to $\pm\Id$, and the same holds for the
monodromy around $a$. In particular, if we fix $a \in \mathbb{C}$ with $|a|<1/2$ and consider the monodromy
around the circle $|t|=1$, this is also equal to $\pm \Id$. Since the monodromy varies continuously, also the monodromy for $a=0$ around this circle is $\pm \operatorname{Id}$. Finally, since also for $a=0$ the function $f$ has a pole of order at most two, we deduce that $\Sp(f)\in \Op_1^{\operatorname{int}, \operatorname{reg}}(\mC)$.
\end{proof}

We now make some comments on the proof of the previous lemma. Assume $f$ is a Laurent polynomial in $t,s$ as in the lemma
and set $F=\left(\begin{smallmatrix} 0 & f \\ 1 & 0 \end{smallmatrix}\right)$. Then the discussion above proves that any local solution of $\dot v+ Fv=0$ extends to a meromorphic function on the plane, with possible poles along $t= 0$ and $t=a$.
In particular, there exists a $\PSL_2$-valued holomorphic function $H(t,a)$, defined for $t\neq 0,a$, such that for each fixed value of $a$ the function $H(\cdot,a)$ satisfies $\dot H \, H^{-1}=-F$. We use this remark to prove the following lemma.
\begin{lemma}\label{lem:lmCepspunti}
Let $\lambda\leq \mu$ be natural numbers and let $\mC[\varepsilon]$ be the ring of dual numbers. We have $$\Op_2^{\lambda,\mu}|_{a=0}(\mC[\gre])\subset \Op_1^{\operatorname{int}, \operatorname{reg}}(\mC[\gre])\ .$$
\end{lemma}
\begin{proof}
We will use Lemma \ref{lem:tanspnr}. 
Let $f_0+\gre g_0\in \Op_2^{\lambda,\mu}(\mC[\gre])\subset \Op_1^*(\mC[\gre])=\mC[\gre]((t))$.
First we show that $g_0\in t^{-1}\mC[[t]]$. Indeed from the inclusion $\Op_2^{\grl,\mu}|_{a=0}\subset  (\Op_1^*)_{\geq -2}$ it follows
that $g_0\in t^{-2}\mC[[t]]$. By Definition \ref{def:flambda} the polynomial $f_\lambda(\zz_2)$ is in the ideal defining $\Op_2^{\lambda,\mu}|_{a=0}$, so we know that $f_\lambda(\zz_{-2}(f_0+\varepsilon g_0))=0$. By Lemma \ref{lemma:RootsOffLambda1}, the polynomial $f_\lambda(\zz_{-2})$ has complex coefficients and no repeated roots. This implies that $\zz_{-2}(f_0+\varepsilon g_0)=\zz_{-2}(f_0)+\varepsilon \zz_{-2}(g_0)$ lies in $\mathbb{C}$, hence that $\zz_{-2}(g_0)=0$. Since $\zz_{-2}(g_0)$ is precisely the coefficient of $t^{-2}$ in $g_0$, we have $g_0\in t^{-1}\mC[[t]]$ as claimed.

We now check the second condition of Lemma \ref{lem:tanspnr}.
We will argue similarly to the proof of Lemma \ref{lem:Cpunti1}: as in the proof of that lemma, to show the desired inclusion it is enough to prove that
the image of the specialisation map
$$
\Sp:\Op_2^{\lambda,\mu}(\mC[a][\gre])\lra \Op_1^*(\mC[\gre])
$$
is contained in $\Op_1^{\operatorname{int},\operatorname{reg}}(\mC[\gre])$,
and again we may assume that all coordinates $a_i,b_i$ are zero for $i>\mu$. 
Let $h\in \Op_2^{\lambda,\mu}(\mC[a][\gre])$, hence $h=f+\gre g$ with $f\in\Op_2^{\lambda,\mu}(\mC[a])$, and let $F=\left(
\begin{smallmatrix}
0 & f \\ 1 & 0
\end{smallmatrix}\right)$ 
and $H$ be as in the discussion that precedes the lemma (in particular, we have $\dot{H}H^{-1}=-F$). Since $a$ is not a zero divisor in $\mC[a][\gre] \subset \mC[[a]][\gre]$, by Equation \ref{eq:DescriptionOp2LambdaMu} we have that 
$E_t(h)$ is in $\Op_t^{\lambda}(Q[\gre])$ and $E_s(h)$ is in $\Op_s^{\mu}(Q[\gre])$. By Lemma \ref{lem:tanspnr}, this implies
\[
\Res_t \bigg(E_t\big(H^{-1}\begin{pmatrix} 0 & g \\ 0 & 0 \end{pmatrix}H\big)\bigg)=0
\quad\text{ and }\quad
\Res_s \bigg(E_s\big(H^{-1}\begin{pmatrix} 0 & g \\ 0 & 0 \end{pmatrix}H\big)\bigg)=0\ .
\]
Since all our functions are meromorphic, we may rewrite these conditions as 
$$
\Res_0 \big(H^{-1}\begin{pmatrix} 0 & g \\ 0 & 0 \end{pmatrix}H\big)=\Res_a \big(H^{-1}\begin{pmatrix} 0 & g \\ 0 & 0 \end{pmatrix}H\big)=0
$$
for all $a\neq 0$. Since $0$ and $a$ are the only singularities of this function, for sufficiently small values of $a \in \mC$ the sum of the above residues is computed by the Cauchy integral 
around a fixed circle centred at $0$. By continuity, this integral is zero also for $a=0$. By Lemma \ref{lem:Cpunti1}, $\Sp(f)$ belongs to $\Op_1^{\operatorname{int}, \operatorname{reg}}(\mC) = \bigsqcup_{\nu \in \mN} \Op_1^{\nu}(\mC)$, hence it belongs to $\Op_1^{\nu}(\mC)$ for some $\nu$. By Lemma \ref{lem:tanspnr}, the conditions checked above imply $\Sp(f+\gre g)\in \Op_1^{\nu}(\mC[\gre]) \subset \Op_1^{\operatorname{int},\operatorname{reg}}(\mC[\gre])$, which is what we wanted to show.
\end{proof}

\subsubsection{Proof of Theorem \ref{thm:restrizionediagonale}}
Combining Lemmas \ref{lemma:RootsOffLambda1} and \ref{lem:Cpunti1} we obtain
\begin{equation}\label{eq:InclusionPreEquality}
  \Op_2^{\lambda, \mu}|_{a=0} (\mC) \subset 
  \coprod_{\substack{|\mu-\lambda| \leq \nu \leq \lambda+\mu \\ \nu \equiv \lambda + \mu \bmod{2}}} \Op_1^\nu(\mC) \ .
\end{equation}
To prove that equality holds in this inclusion we check that the ``dimension'' of the two varieties are the same. To do this, we recall the following lemma in commutative algebra (see \cite[\href{https://stacks.math.columbia.edu/tag/00QK}{Lemma 00QK}]{stacks-project}):
if 
$S$ is an integral $\mC[[a]]$-algebra of finite type, and $S/aS$ is not the zero ring, then 
$$
\dim S/aS = \dim S_a,
$$
and moreover every irreducible component of $S/aS$ has the same dimension.
We cannot apply this lemma directly, because the scheme $\Op_2^{\grl,\mu}$ is not of finite type over $\mC[[a]]$. However, as already pointed out several times, 
the equations defining this variety involve only a finite number of variables, so we will be able to easily reduce to the finitely-generated situation. 

Assume $\grl\leq\mu$ and let $\nu=\grl+\mu-2j$ with $j=0,\dots,\grl$ so that, as already noticed, 
$\Op_1^\nu$ and $\Op_2^{\lambda,\mu}|_{a=0}$ intersect nontrivially by the discussion in Section \ref{sec:ipergeo}.

Choose $N$ larger than $\grl+\mu$. 
Let $X$ be the intersection of $\Op_2^{\lambda,\mu}$ with the subscheme of $(\Op^*_2)_{\geq -2}=\Spec(A[a_i,b_i:i\geq -2])$ given by 
$a_i=b_i=0$ for $i\geq N$. 
Similarly define $Y$ and $Z^\nu$ as the subschemes of $\Op_2^{\lambda,\mu}|_{a=0}$ and $\Op_1^{\nu}$ defined by $\zz_i=0$ for all $i\geq 2N$. Notice in particular that
$Y= X\times_{\Spec A} \Spec \mC$. We already noticed that 
\[
\Op_2^{\grl,\mu} = X\times_{\Spec A} \Spec A[a_i,b_i\st i\geq N]
\]
and similarly 
\begin{equation}\label{eq:YandZnu}
\Op_2^{\lambda,\mu}|_{a=0} = Y\times_{\Spec \mC} \Spec\mC[\zz_i\st i\geq 2N] \; \text{ and } \, 
\Op_1^{\nu} = Z^{\nu}\times_{\Spec \mC} \Spec \mC[\zz_i\st i\geq 2N]\ .
\end{equation}
To show that equality holds in \eqref{eq:InclusionPreEquality} it suffices to prove that for every $\nu$ the (irreducible) scheme $Z^\nu$ is a connected component of $Y$.
Notice that $Y$ and $Z^\nu$ are schemes of finite type over $A$. 
Let $Y'$ be a connected component of $Y$ which intersects $Z^\nu$. 
By Lemma \ref{lem:Cpunti1} and Equation \eqref{eq:YandZnu} we know that 
$Y'(\mC)\subset Z^\nu(\mC)$. Now notice that 
the dimension of $Z^\nu$ is $2N=2(N+2)-4$ (four equations in $2(N+2)$ variables), 
and similarly also
$X\times_{\Spec A} \Spec Q$ has dimension $2N$. By \cite[\href{https://stacks.math.columbia.edu/tag/00QK}{Lemma 00QK}]{stacks-project}, recalled above, we have that $Y'$ (which is a union of irreducible components of $X \times_{\Spec A} \Spec \mC$) also has dimension $2N$. As $Z^\nu$ is irreducible we deduce $Y'(\mC)=Z^\nu(\mC)$, hence in particular the reduced subscheme of $Y'$ agrees with the (smooth) scheme $Z^\nu$.
By Lemma \ref{lem:lmCepspunti} we have that the tangent space to $Y$ at every point $y \in Y(\mC)$ is contained in the tangent space to $Z^\nu$ at $y$: as $Z^\nu$ is smooth, we deduce $Y'=Z^\nu$, proving our claim. \qed

\section{Lie algebras}\label{sec:Liealgebras}

In this section we introduce some variants in our context of the classical affine algebra and of
its enveloping algebra. We give our definitions starting from a general simple $\mathbb{C}$-Lie algebra $\mathfrak{g}$, even though we will mostly be interested in the case 
$\mathfrak{g}=\mathfrak{sl}_2$.

\subsection{Lie algebras and enveloping algebras}\label{ssec:Liealgebras}

Let $S$ be an integral commutative $\mC$-algebra, and let $\Rtest$ be a commutative $S$-algebra equipped with an $S$-linear derivation $f\mapsto \partial f=f'$ and 
a residue map $\Res_R:\Rtest\lra S$, i.e., an $S$-linear map such that $\Res_R(f')=0$ for all $f\in \Rtest$. The affine Lie algebra corresponding to $\mathfrak{g}$ and $\Rtest$ is a central extension of $\Rtest \otimes_\mC \gog$ by a $1$-dimensional factor $S\cdot C_\Rtest$,
\[
\hat \gog_\Rtest   =   \gog \otimes_\mC \Rtest \oplus S\cdot C_\Rtest,
\]
with non-trivial bracket defined by
\[
[x \, f , y \,g ]=[x,y] \, fg+\res_R (f'\,g)\, \kappa(x,y) \, C_\Rtest,
\]
where $x,y \in \gog$, $f,g \in \Rtest$, and $\kappa(x,y)$ is the Killing form of $\gog$. By definition, $\hat{\gog}_R$ is a	 Lie algebra over $S$.
The enveloping algebra of $\hgog_R$ will be taken over the ring $S$ and will be denoted by
$U_R$. 
The derivative $\partial$ gives an action of the $S$-Lie algebra 
$\Der_R:=R\partial$ on $\hgog_R$ by $S$-derivations as follows
$$
(f\partial)\cdot (xg)=xfg'\qquad (f\partial)\cdot hC_R=0 \quad \text{ for }  f,g\in R,\, h\in S, x\in\gog.
$$
This induces a corresponding action by derivations on the enveloping algebra $U_R$.

We will be interested in particular cases of this construction for which the ring has some further properties that it is perhaps convenient to introduce in a uniform way. 
We assume that $\Rtest$ is a free $S$-algebra,  with a countable $S$-basis $\calB_R=\{r_\Jp\}_{\Jp \in \Gamma}$ indexed by $\Indici=\tfrac 1{k_R} \mZ$ (in our applications we will have $k_R=1$ or $2$). 
For $\Jp \in \Gamma$, let $I_R(\Jp)$ be the $S$-span of the set $\{r_i:i\geq \Jp\}$ and assume that the following conditions hold (notice the asymmetry between $\mZ$ and $\Gamma$ in condition I2):
\begin{itemize}
 \item[I\,1:] $I_R(0)$ is an $S$-subalgebra of $\Rtest$   
 \item[I\,2:] $I_R(n)\cdot I_R(\Jp)\subset I_R(n+\Jp)$ for all $n\in \mZ$ and $\Jp\in\Indici$
	 \item[I\,3:] $\Res_R I_R(0)=0$ and $\partial I_R(0) \subset I_R(0)$, and for all $\Jp \in \Gamma$, $\partial I_R(\gamma)\subset I_R(\gamma-1)$.
\end{itemize}
Assumptions I1 and I3 imply that the $S$-Lie submodule of $\hgog_R$ given by 
\begin{equation}\label{eq:hgp}
\hat{\gog}_R^+ = \gog  \otimes_\mC I_\Rtest(0) \oplus S\,C_\Rtest 
\end{equation}
is a Lie sub-algebra of $\hat{\gog}_R$ called the \emph{positive part} of $\hat{\gog}_\Rtest$.

\medskip

We denote by $\bRtest$ the completion of $\Rtest$ with respect to the topology defined by the $S$-submodules $I_R(\Jp)$,  and let $\bI_R(\Jp)$ be the completion of $I_R(\Jp)$. Every element of $\bRtest$ can 
be written as a series $\sum_{\Jp \geq N, \Jp \in \Gamma}  s_\Jp r_\Jp$ for some $N \in \Gamma$.

Condition I3 also implies that $\Res_\Rtest$ is continuous with respect to the topology on $\Rtest$ defined by the submodules $I_\Rtest(\Jp)$ and the discrete topology on $S$, and that $\partial$ is continuous with respect to the topology defined by the submodules $I_\Rtest(\Jp)$. We also denote by $\Res_\Rtest$ and $\partial$ their unique continuous extensions to $\bRtest$.

The topology on $\Rtest$ induces a topology on $U_\Rtest$ as follows. For $n\geq 0$ let $J_R(n)$ be the left $U_\Rtest$-ideal generated by $\gog\otimes_\mC I_\Rtest(n)$. We consider the following completion of the enveloping algebra $U_R$:
$$
\bU_\Rtest=\limpro \frac{U_\Rtest}{J_R(n)},
$$
where the limit is taken in the category of $U_\Rtest$-modules, 
and we denote by $\bJ_R(n)$ the closure of $J_R(n)$ in $\bU_R$. We now recall that $\bU_\Rtest$ has a natural structure of associative algebra.

\begin{lemma}\label{lem:defhatU}
For every $x\in U_\Rtest$ and every $n \in \mZ$ there exists $N \in \mZ$ such that $y\cdot x\in J_\Rtest(n)$ for all $y\in J_\Rtest(N)$.
\end{lemma}
\begin{proof}
The statement is clearly linear, hence it suffices to check it for a set of $S$-generators 
of $U_\Rtest$. We can therefore assume that $x=C_\Rtest^j x_1\cdots x_m$, where 
$x_i= a_i f_i$ with $a_i\in \gog$ and $f_i\in \Rtest$. As $C_\Rtest$ is central and 
$J_\Rtest(n)$ is a left ideal, we can further reduce to the case $j=0$.
Similarly, we can also assume that $y$ is of the form $y=b g$, with $g\in I_\Rtest(N)$ and 
$b\in \gog$. We need to prove that for $N$ large enough we then have $y \cdot x\in J_\Rtest(n)$.

We prove the claim by induction on $m$. If $m=0$ the claim is trivial. 
Let $m>0$ and let $d$ be such that $f_i\in  I_\Rtest(d)$ for 
for $i=1,\ldots,m$. By the induction hypothesis, there exists $M\geq n$ such that for all $z= c h$ with 
$h\in I_\Rtest(M)$ and $c\in \gog$ we have 
$$
z\cdot x_{i}\cdots x_m\in J_\Rtest(n)
$$
for all $i=2,\cdots,m$. Choose $N = \max\{0, n, M+|d|, |d|\}+k_R$ . Then for $i=1,\ldots,n$ we have
$$
z_i:=[y, x_i]= [bg, a_if_i] = [b,a_i]gf_i + \res (g'f_i)  \kappa(b,a_i) C_\Rtest= c_i h_i 
$$
with $h_i\in I_\Rtest(M)$ and $c_i\in\gog$ (notice that $gf_i' \in I_\Rtest(N+|d|-k_R) \subseteq I(0)$, so it has trivial residue). Hence 
$$
y\cdot x= x\cdot y + z_1\cdot x_2  \cdots x_m + x_1\cdot z_2 \cdot x_3 \cdots x_m +\cdots +x_1\cdots x_{m-1}\cdot z_m
$$
belongs to $J_\Rtest(n)$ since this is true for every summand. 
To see that this last statement holds, recall that $J_\Rtest(n)$ is a left ideal and that we have:
\begin{itemize}
\item  $y \in J_\Rtest(n)$, because $y \in J_\Rtest(N)$ and $N \geq n$;
\item  $z_ix_{i+1}\cdots x_m \in J_\Rtest(n)$, because $z_i=c_ih_i$ with $h_i \in I_\Rtest(M)$, and by definition $M$ is such that $h_i \in I_\Rtest(M) \Rightarrow (c_ih_i)x_{i+1}\cdots x_m \in J_\Rtest(n)$;
\item $x_1\cdots x_{m-1} z_m \in J_\Rtest(n)$, because $z_m$ itself is in $J_R(M)$, hence in $J_\Rtest(n)$.
\end{itemize}
\end{proof}

This allows us to endow $\bU_\Rtest$ with the structure of an associative algebra as follows. Let $p_n, q_n$ be two sequences of elements of $U_\Rtest$ that converge to elements $p, q \in \bU_\Rtest$. For every $n$, denote by $\overline{p_n}, \overline{q_n}$ the classes of $p_n, q_n$ in $U_\Rtest/J_\Rtest(n)$. By the previous lemma, there exists $N$ such that for all $h,k\geq N$
$$
p_N\cdot q_N \coinc p_h \cdot q_k \mod J_\Rtest(n);
$$
furthermore, the class of $p_h \cdot q_k \bmod J_\Rtest(n)$ is independent of the choice of the elements $p_h$, $q_k$ lifting $\overline{p_h}$, $\overline{ q_k}$, and $p_Nq_N$ is a convergent sequence in $\bU_\Rtest$. We may therefore define the product of $p$ and $q$ as the limit of the sequence $p_N q_N$.

Our main goal is to study the center of the algebra $\hU_R$ at the \emph{critical level}, that is, the algebra
\[ \hU_R : = \frac{\bU_R}{(C_R+\frac 12)}.\]
We denote by $\hJ_R(n)$ the image of $\bJ_R(n)$ in $\hU_R(n)$, and by $\hZ_R$ the center of $\hU_R$. 

\begin{remark}\label{oss:2completamenti}
In some of our constructions it would be more natural, or more usual, to consider the Lie algebra 
$\hgog_\bRtest$ and then construct $\bU_\bRtest$ as the completion of $U_\bRtest$ with respect to the left ideals generated by $\gog \otimes_\mC \bI_R(n)$. However, our real objects of interest are the algebra $\bU_R$ and its quotient $\hU_R$, and we have isomorphisms $\bU_\bRtest \cong \bU_\Rtest$ and $\hU_\bRtest \cong \hU_\Rtest$.
We prefer to work with the Lie algebra $\hgog_\Rtest$ because we can then use the classical form of the Poincar\'e-Birkhoff-Witt theorem, which applies to Lie algebras that are free as modules over their base ring.
\end{remark}

\subsection{Filtrations on the completed enveloping algebra}
\label{ssec:PBW}
In this section we define two different filtrations on the algebra $\hU_R$, one coming from the usual Poincaré-Birkhoff-Witt grading, the other induced by a suitable total ordering of the basis elements of $\hgog_R$.
Let $\{J^\gra\}_{\alpha}$
be a $\mC$-basis of $\gog$. We denote by $\calBg{R}=\{J^\gra r_\Jp\}_{\alpha, \Jp}\cup\{C_\Rtest\}$ the induced $S$-basis of the affine Lie algebra $\hat{\gog}_R$.
Suppose now that the basis $J^\alpha$ of $\gog$ is totally ordered. We order the elements of $\calBg{R}$ as follows: we let $C_\Rtest<J^\gra r_\Jp$ for all $\Jp$ and $\gra$, and
$$ J^\gra r_\Jp<J^\grb r_\Jq $$
if $\Jp<\Jq$ or if $\Jp=\Jq$ and $\gra<\grb$. Given elements $x_1,\ldots,x_m \in \calBg{R}$, we say that $x_1\cdots x_m$ is an \emph{ordered monomial} if $x_i\leq x_{i+1}$ for $i=1,\ldots,m-1$. We denote by $\calBU R$ the set of ordered monomials. 
As $\hgog_R$ is free over $S$, the Poincaré-Birkhoff-Witt theorem implies that the set $\calBU R$ is an $S$-basis of $U_R$. Notice that $U_R$ is a free algebra over $S[C_R]$, with basis given by the set $\calBhU R$ of ordered monomials $x_1\cdots x_m$ where each $x_i$ is different from $C_R$. We denote by $x_1\cdots x_n$ and $\calBhU R$ also their images in $\hU_R$.
We now introduce some filtrations on $U_R$ and $\hU_R$. 

For an ordered monomial $C_R^c  x_1\cdots x_n$ in $\calBU{R}$ (or in $\calBhU R$), where $x_j=J^{\gra_j} r_{\Ja_j}$, we define
\[
\deg(C_R^c x_1\cdots x_n)=n \qquad\text{ and } \qquad \Jdeg(C_R^c x_1\cdots x_n)=(\Ja_1,\ldots,\Ja_n) \in \Gamma^n.
\]
We define a filtration of $U_R$ by setting $U^{\leq n}_R$ to be the $S[C_R]$-span of 
the ordered monomials $x\in \calBhU R$ with $\deg(x)\leq n$. We call this filtration the \textit{PBW filtration},
although it is slightly different from the standard one (note that we assign degree $0$ to the central element $C_R$).  Similarly we define $\hU_R^{\leq n}$ as the closure of the $S$-span of the elements $x\in \calBhU R$ with $\deg(x)\leq n$, or equivalently, as the closure of the image of $U^{\leq n}_R$ in $\hU_R$.

In order to define the second filtration and discuss some of its basic properties we need to introduce some notation for sequences of elements of $\Gamma$. We begin by ordering the non-decreasing sequences of elements of $\Gamma$ with respect to the reverse lexicographic order. 
More precisely,
if $\Ja=\Ja_m\leq \dots \leq \Ja_1$ 
and $\Jb=\Jb_n\leq \dots \leq \Jb_1$ are two non-decreasing sequences in $\Gamma$, we say that
$\Ja> \Jb$ if there exists $0\leq h\leq m,n$ such that the following conditions both hold:
\begin{itemize}
 \item $
\Ja_1=\Jb_1, \; \Ja_2=\Jb_2,\dots, \Ja_{h}=\Jb_{h}$;
\item either $h=n<m$,  or $h< \min\{n,m\}$ and $\Ja_{h+1} > \Jb_{h+1}$.
\end{itemize}

We also define the degree of a non-decreasing sequence $\Ja=\Ja_m\leq \dots \leq \Ja_1$ as $\deg(\Ja)=m$, and the concatenation  $\Ja\cdot \Jb$ of $\Ja$ and $\Jb$ as the sequence obtained from $\Ja_m,\ldots,\Ja_1, \Jb_n,\ldots,\Jb_1$ by reordering the elements in non-decreasing order. 
We record the following trivial remark about properties of this ordering:
\begin{lemma}\label{lem:ordineGamma}
Let $\Ja,\Jb,\Jc$ be three non-decreasing sequences of elements of $\Gamma$. We have $\deg(\Ja \cdot \Jb)=\deg(\Ja)+\deg(\Jb)$, and, if $\Jb \geq \gamma$, then $\Ja \cdot \Jb \geq \Ja \cdot \Jc \geq \Ja$. Moreover, given a non-decreasing sequence $\Ja$ of degree $m$, any non-empty subset of $\{ \Jb = \Jb_m \leq \ldots \leq \Jb_1 \st \Jb \leq \Ja \}$ has a unique maximal element.
\end{lemma}

We now construct a second filtration by generalising the definition of $J_R(n)$.
\begin{definition}\label{def:JRJR}
Given a non-decreasing sequence $\Ja=\Ja_m\leq \dots \leq \Ja_1$, we set 
\begin{equation}\label{eq:J}
\begin{aligned}
J_R^{\leq m}(\Ja) & = \langle x\in \calBU R\st \deg(x)\leq m \text{ and } \Jdeg(x)\geq \Ja \rangle_S +U^{\leq m-1}_R\\
J_R^{\leq m}[\Ja] & = \langle x\in \calBU R\st \deg(x)\leq m \text{ and } \Jdeg(x)  >  \Ja \rangle_S +U^{\leq m-1}_R,
\end{aligned}
\end{equation}
and we define $\hJ_R^{\leq m}(\Ja)$ (respectively $\hJ_R^{\leq m}[\Ja]$) as the closure of the image of $J_R^{\leq m}(\Ja)$ (respectively $J_R^{\leq m}[\Ja]$) in $\hU_R$.
To denote these spaces we will use also the notation $J_R^{\leq m}(r_{\Ja_m},\dots,r_{\Ja_1})$, 
$J_R^{\leq m}[r_{\Ja_m},\dots,r_{\Ja_1}]$, and similarly for their completions. 
\end{definition}

Let $\calB(\hJ_R^{\leq m}(\Ja))$ be the set of ordered monomials $x\in \calBhU R$ such that  $\Jdeg(x)\geq \Ja$ and $\deg(x)\leq m$, and define $\calB(\hJ_R^{\leq m}[\Ja])$ similarly.
These sets are bases 
respectively of $J_R^{\leq m}(\Ja)$ and  $J_R^{\leq m}[\Ja]$ as $S[C_R]$-modules, and topological bases respectively of $\hJ_R^{\leq m}(\Ja)$ and $\hJ_R^{\leq m}[\Ja]$ as $S$-modules. More explicitly, 
any element of $\hU_R$ can be written uniquely as a series 
\begin{equation}\label{eq:serie}
\sum_{b\in \calBhU{R}}s_b \,b,
\end{equation}
such that each $s_b$ is in $S$, and for every $n\in\mN$ the set $\{b : \Jdeg(b)<n, s_b\neq 0\}$ is finite. Similarly, any element in $\hJ^{\leq m}(\Ja)_R$ (resp.\ $\hJ^{\leq m}[\Ja]_R$) can be written uniquely as a series as above, with $b\in\calB(\hJ_R^{\leq m}(\Ja))$ (resp.\ $\calB(\hJ_R^{\leq m}(\Ja))$). We now use this description of the elements of $\hU_R$ as series to deduce some properties of this algebra.

\begin{lemma}\label{lem:piatto}
The following hold:
\begin{enumerate}[\indent a)]
\item If $S$ is an integral domain then $\hU_R$ is a torsion-free $S$-module. In particular, if $S$ is a principal ideal domain, then $\hU_R$ is a flat $S$-module;
\item The submodule $\hJ_R(n)$ is a direct summand of $\hU_R$ as an $S$-module. 
In particular, for every $f\in S$ we have $\hJ_R(n)\cap f\hU_R(n)=f\hJ_R(n)$. Similar statements hold for 
the submodules $\hJ_R^{\leq m}(\Ja)$ and $\hJ_R^{\leq m}[\Ja]$;
\item If $S$ is an integral domain and $f\in S$ is nonzero, then $\hJ_R(n)[f^{-1}]\cap \hU_R=\hJ_R(n)$. Similar statements hold for the submodules $\hJ_R^{\leq m}(\Ja)$ and $\hJ_R^{\leq m}[\Ja]$;
\item The classes of ordered monomials $x\in \calBhU{R}$ with $\Jdeg(x)=a$ form an $S$-basis of the quotient
$\hJ_R^{\leq m}(\Ja)/\hJ_R^{\leq m}[\Ja]$. 
\end{enumerate}
\end{lemma}
\begin{proof}
Parts a) and d) are clear from the description of the elements of $\hU_R$ as series.
To prove b), notice that a complement of $\hJ_R^{\leq m}(\Ja)$ as an $S$-module is given by the elements as in equation
\eqref{eq:serie} such that $s_b=0$ for all $b\in \calB(\hJ_R^{\leq m}(\Ja))$. A complement in the other two cases can be described in a similar way. 
Part c) follows from b).
\end{proof}

For an element $y\in U_R$ we define the \emph{J-degree} and the \emph{leading term} of $y$ as follows. 
Choose $m$ minimal such that $y\in U^{\leq m}_R$. Then 
$$y=\sum_{\substack{x\in \calBhU{R} \\ \deg(x)=m}} f_x(C_R) \, x +u,$$
with $u\in U^{\leq m-1}_R$ (or $u=0$ if $m=0$) and $f_x(C_R) \in S[C_R]$.
Let $\Ja\in \Gamma^m$ be the minimum of the finite set $\{\Jdeg(x) \st \deg(x)=m, \, f_x \neq 0\}$.
We set 
$$
\deg(y)=m, \qquad \Jdeg(y)=\Ja, \qquad \lt(y)=\sum_{x\in \calBhU{R} \, \st \, \Jdeg(x)=\Ja}f_x(C_R)\, x\in \frac{U^{\leq m}_R}{U^{\leq m-1}_R}. 
$$
The same definition also makes sense for $y \in \hat{U}_\Rtest$, provided that $y\in \hat{U}_\Rtest^{\leq m}$ 
for some $m$, with the only difference that the coefficients $f_x$ are in the ring $S$ (notice that, thanks to the description of elements of $\hU_R$ as series, we know that the set $\{\Jdeg (x) : f_x\neq 0\}$ has a minimum). We consider the leading term of $y\in U_R$ (resp.~of $y\in \bigcup_m \hU^{\leq m}_R$) as an element of the commutative graded ring
$$
\Gr(U_R)=\bigoplus \frac{U^{\leq m}_R}{U^{\leq m-1}_R} \qquad \text{(resp. } 
\Gr(\hU_R)=\frac{\Gr(U_R)}{(C_R+\frac 12)}\text{ )}
$$
Continuing with the notation above we also remark that, for every $y\in \hU_R^{\leq m}$, the images of $y$ and $\lt(y)$ in $\hJ_R^{\leq m}(\gamma)/\hJ_R^{\leq m}[\gamma]$ coincide.

The following lemma contains some elementary observations about the interaction between the product in 
$\hU_R$ and the submodules $\hJ_R^{\leq m}(\Ja)$. It will be used repeatedly in Section \ref{sect:Centre}.

\begin{lemma}\label{lem:filtrazioneJmn}
Assume $S$ is an integral domain. 
\begin{enumerate}[\indent a)]
\item if $x,y\in \bigcup_m \hU_R^{\leq m}$ then $\Jdeg(x\cdot y)=\Jdeg(x)\cdot \Jdeg(y)$ and $\lt(x\cdot y)=\lt(x)\cdot \lt(y)$. In particular, if $x,y\neq 0$ then $\Jdeg(x\cdot y) \geq \Jdeg(x)$, with equality only if $y \in S$;
\item $\hJ_\Rtest^{\leq m}(\Ja)\cdot \hJ_\Rtest^{\leq n}(\Jb)\subset \hJ_R^{\leq m+n}(\Ja\cdot \Jb)$ and
      $\hJ_\Rtest^{\leq m}(\Ja)\cdot \hJ_\Rtest^{\leq n}[\Jb]\subset \hJ_R^{\leq m+n}[\Ja\cdot \Jb]$;
\item let $\Ja_2 \leq \Ja_1$ be elements in $\Gamma$. If $y \in \hJ_R^{\leq 1}[\Ja_1]$, then for every $x \in \hU^{\leq 1}_R$ we have $x\cdot y \in \hJ_R^{\leq 2}[\Ja_2,\Ja_1]$.
\end{enumerate}
\end{lemma}
\begin{proof}
The ring $\Gr(\hU_R)$ is a commutative polynomial algebra with coefficients in the integral domain $S$ and generators the variables $J^\gra r_\Ja$. Part a) is then trivially true when $x,y$ are monomials; it also holds for general $x, y$ by Lemma \ref{lem:ordineGamma}. 
Parts b) and c) follow immediately from part a).
\end{proof}

\subsection{Examples: $\hat{\mathfrak{g}}_1$, $\hat{\mathfrak{g}}_2$, $\hat{\mathfrak{g}}_s$, $\hat{\mathfrak{g}}_t$, $\hat{\mathfrak{g}}_{t,s}$}\label{sec:examples}
We now specialise the previous construction to the concrete cases we are interested in. 

\begin{itemize}
\item $\hU_1$. We choose $S=\mC$, $\Rtest=\mC[t^{\pm 1}]$ with the usual derivation, and $\Res_\Rtest=\Res$ to be the coefficient of $t^{-1}$. The topology is given by choosing $\calB_1=\{t^n\}$, so that $\bRtest=\mC((t))$. 
We denote by $\hat{\gog}_1$ the corresponding affine Lie algebra. Similarly, $U_1$ and $\hat{U}_1$ will stand for the corresponding universal enveloping algebra and its completion with respect to the submodules $J_1(n)$  at the critical level.
\item $\hU_t$. This algebra is constructed in a similar way to $\hU_1$, with the field of complex numbers replaced by $Q$. In this case we use the notations $\hat{\gog}_t, U_t, \hat{U}_t, C_t$ with their obvious meaning.

\item $\hU_s$. The construction is identical to $\hU_t$, but we denote the variable by $s$ instead of $t$.
\item $\hU_{t,s}$. We choose $S=Q$, $\Rtest=\Rtest_{t,s} = Q[t^{\pm 1}]\times Q[s^{\pm 1}]$
with the derivation given by $\partial(f(t),g(s))=(f'(t),g'(s))$, and let $\Res_R=\Res_t+\Res_s$ be the sum of the one-variable residue maps. The basis $\calB=\{r_n\}$ is indexed by the half-integers, so that 
$k_R=2$ and 
\[
r_n=s^n\quad \text{ for $n \in \mZ$ }\quad \text{and} \quad r_n=t^{\lfloor n \rfloor} \text{ for $n \notin \mZ$.}
\]
 Notice that elements of $R$ should more precisely be represented as pairs: for example, by $s^n$ we mean the pair $(0,s^n)$, and by $t^m$ we mean $(t^m,0)$. In particular, $s^0=(0,1)$ and $t^0=(1,0)$ are two distinct elements of $R$.
The completed ring in this case is $\bRtest=Q((t))\times Q((s))$. 
We notice that we have an isomorphism 
$\displaystyle \hgog_{t,s}\simeq\frac{\hgog_t\oplus\hgog_s}{Q(C_t-C_s)}$,
and that for all $n \in \frac 12 \mN$ we have
$$
J_{t,s}(n) = \frac{U_t\otimes_Q J_s(\lfloor n+\frac 12 \rfloor) + J_t(\lfloor n \rfloor)\otimes_Q U_s}{(C_t\otimes 1 -1\otimes C_s)}.
$$
In particular, for all $n\in \tfrac 12 \mN$ we have an isomorphism
\begin{equation}\label{eq:UtsAndUtUs}
\frac{\hU_{t,s}} {\hJ_{t,s}(n)} \simeq \frac{\hU_{t}} {\hJ_{t}(\lfloor n+\frac 12 \rfloor)} \otimes_Q \frac{\hU_s} {\hJ_s(\lfloor n \rfloor) }.
\end{equation}
This implies that we have a natural injective homomorphism of $Q$-algebras with dense image
$$
\hat{U}_t\otimes _Q \hat U_s \lra \hat U_{t,s}\ .
$$
We will implicitly use this map to identify elements in $\hat{U}_t\otimes _Q \hat U_s$ with their images in  $\hat U_{t,s}$.
\item $\hU_2$. We choose $S=A$, and we take $\Rtest$ to be the subalgebra $R_2$ of $K_2$ spanned (over $A$) by the elements $u_n$,
$v_n$ (see Section \ref{sez:base}). The derivation $\partial$ is the unique $A$-linear derivation such that $\partial t=\partial s=1$, and 
$\Res=\Res_2$. As $S$-basis of $R$ we take the set $\calB=\{u_n,v_n\}$, indexed by the half-integers as follows:
$$
r_n=u_n\quad \text{ for $n \in \mZ$ } \quad \text{and} \quad r_n=v_{\lfloor n \rfloor} \text{ for $n \notin \mZ$.}
$$
The completed ring in this case is $\bRtest=K_2$. 
\end{itemize}

Each of these choices determines a collection of objects that will be denoted by the corresponding
subscript. For example, for the first example in the list we have the ideals $I_1(n)$
and their completions $\bI_1(n)$, the central element $C_1$, the enveloping algebra $U_1$ with filtration $J_1(n)$, the quotient at the critical level $\hU_1$, the bases $\calBg{1}$ and $\calBhU 1$, and so on. We use similar notations for each of the other choices above. 

\subsection{Specialisation and expansion}\label{sec:SpecialisationAndExpansion}
Let $(S,R)$ and $(S',R')$ be two pairs as in the beginning of Section \ref{ssec:Liealgebras}.
Let $\phi:S\lra S'$ be a ring homomorphism and $\phi:\Rtest\lra \Rtest'$ be a morphism of $S$-algebras, where $R'$ is considered as an $S$-algebra via $\phi$. If $\phi$ commutes with the derivation and residue maps of $R, R'$.
In this situation, we get an induced $S$-linear map of Lie algebras between 
$\hgog_R$ and $\hgog_{R'}$, and also a morphism at the level of universal enveloping algebras. All these morphisms will be denoted by $\phi$. 
If, moreover, there exists $\ell>0$ such that $\phi(I(n))\subset I'({\ell n})$  for all $n$, then $\phi$ also induces a map between the associated completions.

An example of this construction is the specialisation map $\Sp$ discussed in Section
\ref{sez:base}. In particular, from the results of that section we immediately deduce that the specialisation map induces an isomorphism
$$
\hgog_2/a\hgog_2\isocan \hgog_1.
$$
We prove a similar result for the the enveloping algebra.

\begin{lemma}\label{lem:hS}For all $n\in \mZ$ we have
\begin{enumerate}[\indent a)]
 \item $\Sp(u_n)=t^{2n}$ and $\Sp(v_n)=t^{2n+1}$;
 \item for all $x,y\in \calBg{2}$ with $x<y$ we have $\Sp(x)<\Sp(y)$. In particular, the image of an ordered monomial is an ordered monomial, and $\Sp(\calBhU 2)= \calBhU 1$;
 \item $\Sp(\hJ_2(n))=\hJ_1(2n)$ and $\Sp^{-1}(\hJ_1(2n))=a \hU_2 +\hJ_2(n)$. 
\end{enumerate}
In particular, $\Sp:\hU_2\lra \hU_1$ is well-defined and induces an isomorphism $\hU_2/a\hU_2 \simeq \hU_1$. The natural topology on $\hU_1$ coincides with the quotient topology induced by $\Sp$.
\end{lemma}
\begin{proof}
Parts a) and b) are trivial. Part c) and the last statement follow from a), b) and the description of elements of $\hU_2$, $\hU_1$ as series.
\end{proof}

We now consider the case of the expansion map $E:K_2\lra K_t\times K_s$. The map $E$ commutes with the residue and derivation maps. However, since
$E(R_2)$ is not contained in $R_{t,s}$, we cannot directly apply the observation at the beginning of this paragraph
to obtain a map from $U_2$ to $U_{t,s}$ as we did for the specialisation map. On the other hand, we can use Remark 
\ref{oss:2completamenti} to directly construct  a map $E:\hU_2\lra \hU_{t,s}$. Indeed, one can easily check that $E(I_2(n))$ is contained in $\bI_{t,s}(n)$; since $\bRtest_{t,s}$ and $\hU_{t,s}$ are complete, this implies that $E$ induces maps
$$
E:\hgog_2\lra \hgog_{\bRtest_{t,s}}\quad\text{ and }\quad E:\hU_2\lra \hU_{t,s}.
$$
By Lemma \ref{lem:exp} the map at the level of Lie algebras is injective and has dense image. We now study the induced map between the enveloping algebras at the critical level.

\begin{lemma}\label{lem:hE1}For all $n\in \mZ$ we have
 
\begin{enumerate} [\indent a)]
\item $E(u_n)\coinc  (-a)^n t^n+a^n s^n        \bmod \bI_{t,s}(n+1)$, hence in particular for all $e\in \gog$ we have
      $\Jdeg(E(e\,u_n))=\Jdeg(e\,u_n)=n$ and $\lt(E(e\,u_n))= e\, a^n s^n$;
\item $E(v_n)\coinc  (-a)^{n+1}t^n+a^ns^{n+1} \bmod \bI_{t,s}(n+3/2)$, hence in particular for all $e\in \gog$ we have
      $\Jdeg(E(e\,v_n))=\Jdeg(e\,v_n)=n+\tfrac 12$ and $\lt(E(e\,v_n))= e\, (-a)^{n+1}t^n$;
\item $E(y_n)\coinc a^{n+1}s^n \bmod \bI_{t,s}(n+1)$;
\item If $x\in \hU_2^{\leq m}$ for some $m$, then $\deg (E(x))=\deg (x)$ and $\Jdeg (E(x))=\Jdeg (x)$.
\end{enumerate}
\end{lemma}
\begin{proof}
The first three statements follow from a straightforward computation. We prove d). 
Denote by $r^{(2)}_n$, respectively $r_n^{(t,s)}$, the bases of $R_2$, respectively $R_{t,s}$, 
introduced in the previous section.
Combining parts a) and b) with Lemma \ref{lem:filtrazioneJmn} a) we obtain
\begin{equation}\label{eq:Emonomi}
\lt\big(E(J^{\gra_1}r^{(2)}_{\Jp_1}\cdots J^{\gra_m}r^{(2)}_{\Jp_m})\big)=
\pm a^{N} J^{\gra_1}r^{(t,s)}_{\Jp_1}\cdots J^{\gra_m}r^{(t,s)}_{\Jp_m}
\end{equation}
for an appropriate choice of sign and of exponent $N$. 
Writing elements of 
$\hU_2^{\leq m}$ as series, this formula implies the statement.

\end{proof}

Notice that the map $E$ extends to a map from $\hU_2[a^{-1}]$ to $\hU_{t,s}$ that we still denote by $E$. By Lemma \ref{lem:piatto} c) we see that the subspace topology induced on $\hU_2$ by the inclusion $\hU_2 \subset \hU_2[a^{-1}] $ coincides with the natural topology of $\hU_2$.

\begin{lemma}\label{lem:hE2}The map $E:\hU_2[a^{-1}]\lra \hU_{t,s}$ is injective with dense image. More precisely:
 
\begin{enumerate} [\indent a)]
\item for all $n\in \Gamma$ we have $E^{-1}(\hJ_{t,s}(n))=\hJ_2(n)[a^{-1}]$ ;
\item for every non-decreasing sequence $\Ja$ of elements of $\Gamma$ of degree $m$, the map $E$ induces an isomorphism between $\hJ^{\leq m}_2(\Ja)[a^{-1}]/\hJ^{\leq m}_2[\Ja][a^{-1}]$ and $\hJ^{\leq m}_{t,s}(\Ja)/\hJ^{\leq m}_{t,s}[\Ja]$;
\item for every element $n\in \Gamma$ the map $E$ induces an isomorphism between $\hU_2[a^{-1}]/\hJ_2(n)[a^{-1}]$ and $\hU_{t,s}/\hJ_{t,s}(n)$;
\item $\hU_2[a^{-1}]$ has the topology induced from $\hU_{t,s}$ through the map $E$.
\end{enumerate}
\end{lemma}
\begin{proof}
Parts a) and b) of the previous lemma give the inclusion $E(\hJ_2(n)[a^{-1}])\subset \hJ_{t,s}(n)$.
Let now $x\in \hU_2[a^{-1}]$ and assume that $E(x)\in \hJ_{t,s}(n)$. 
Write $x=y+j$, where $j\in \hJ_2(n)[a^{-1}]$ and $y$ is a finite $Q$-linear combination of elements $z$ in 
$\calBhU 2$ with $\Jdeg(z)<n$. If $y\neq 0$ then $\Jdeg(y)<n$, hence we also have $\Jdeg(E(y))<n$ by Lemma \ref{lem:hE1} d), and therefore $E(y) \not \in \hJ_{t,s}(n)$. Since $E(x)$ and $E(j)$ are in $\hJ_{t,s}(n)$ we get a contradiction. 
This proves a) and implies that $E$ is injective. It also shows that $\hU_2[a^{-1}]$ has the induced topology, that is, part d).

To prove b), notice that the monomials $x\in \calB(\hU_2)$ 
such that $\deg (x)=m$ and $\Jdeg(x)=\grg$ form a $Q$-basis of $\hJ_2^{\leq m}(\grg)[a^{-1}]/\hJ^{\leq m}_2[\Ja][a^{-1}]$, and an analogous statement holds for the quotient $\hJ^{\leq m}_{t,s}(\grg)/\hJ^{\leq m}_{t,s}[\Ja]$.
By equation \eqref{eq:Emonomi}, if
$$
x=J^{\gra_1}r^{(2)}_{\grg_1}\cdots J^{\gra_m}r^{(2)}_{\grg_m}\in\calB(\hU_2)
$$
is such a monomial, then
$$
E(x)\equiv \pm a^N J^{\gra_1}r^{(t,s)}_{\grg_1}\cdots J^{\gra_m}r^{(t,s)}_{\grg_m} \, \bmod J_{t,s}^{\leq m}[\grg],
$$
proving that $E$ induces an isomorphism $\hJ^{\leq m}_2(\Ja)[a^{-1}]/\hJ^{\leq m}_2[\Ja][a^{-1}] \xrightarrow{\sim} \hJ^{\leq m}_{t,s}(\Ja)/\hJ^{\leq m}_{t,s}[\Ja]$ since it sends a basis of the former to a basis of the latter.

We now prove c), which implies in particular that the image of $E$ is dense. From a) we have a natural injective map $\hU_2[a^{-1}]/\hJ_2(n)[a^{-1}] \to \hU_{t,s}/\hJ_{t,s}(n)$, so it suffices to show that this map is surjective. Notice that any element of $\hU_{t,s}/\hJ_{t,s}(n)$ is the image of an element in $\hU_2^{\leq m}$ for some $m$, so, arguing by induction on $m$, it is enough to prove that $E$ induces a surjection between  $\hU_2^{\leq m}[a^{-1}]/\hJ^{\leq m}_2(n)[a^{-1}]$ 
and $\hU_{t,s}^{\leq m}/\hJ^{\leq m}_{t,s}(n)$
for all $m$. Assume by contradiction that this does not hold. By Lemma \ref{lem:ordineGamma} there exists a unique maximal non-decreasing sequence $\Ja$ of elements of $\Gamma$
such that:
\begin{enumerate}
 \item $\deg (\Ja) =m$;
 \item $\Ja$ is less than or equal to $n$ with respect to the reverse lexicographic order;
 \item there exists $y\in \hU_{t,s}^{\leq m}$ with $\Jdeg(y)=\Ja$ such that $\bar y\in \hU_{t,s}^{\leq m}/\hJ^{\leq m}_{t,s}(n)$ is not in the image of the map $\bar E: \hU_2^{\leq m}[a^{-1}]/\hJ^{\leq m}_2(n)[a^{-1}] \lra \hU_{t,s}^{\leq m}/\hJ^{\leq m}_{t,s}(n)$ induced by $E$.
\end{enumerate}
However, by b) there exists $x\in \hJ_2^{\leq m}(\Jp)\subset \hU_2^{\leq m}[a^{-1}]$ such that $E(x)\in y+\hJ_{t,s}^{\leq m}[\Jp]$, which contradicts the maximality of $\Ja$. 
\end{proof}

\section{The centre of the enveloping algebra $\hat{U}_2$}\label{sect:Centre}

Our purpose in this section is to describe the centre of the enveloping algebra $\hat{U}_2$. We denote by $Z_2$ the center of $\hU_2$, and by $Z_2^{\leq n}$ its intersection with $\hU_2^{\leq n}$.
We use a similar notation for the  various enveloping algebras constructed in Section \ref{sec:examples}
We assume from now on that $\gog = \gos\gol_2$, even though some of our constructions make sense for a general semisimple Lie algebra $\gog$. We fix a basis $J^\gra$ of $\gog$ and we denote by $J_\gra$ the dual basis with respect to the Killing form.

\subsection{The centre of $\hat{U}_1$}
We recall some results on the 1-variable case of our construction. Most of the following material is taken from \cite{Frenkel_Langlands_loop_group}, to which we refer for further details. The original reference is \cite{FF92}. 

\begin{definition}[Normal ordered product] For $m_1, m_2 \in \mathbb{Z}$ we define
\[
\nop{ (J^{\alpha} t^{m_1}) (J^{\beta} t^{m_2}) } = \begin{cases}
(J^{\alpha} t^{m_1}) (J^{\beta} t^{m_2}) \text{ if } m_1 \leq m_2 \\
(J^{\beta} t^{m_2}) (J^{\alpha} t^{m_1}) \text{ if } m_1 > m_2
\end{cases}
\]
\end{definition}

\begin{definition}[1-variable Sugawara operators]\label{def:1Sugawara}
For $k \in \mathbb{Z}$ we set
\[
\Sug_k^{(1)} := \sum_{n \in \mathbb{Z}, \alpha} \nop{ J^{\alpha} t^n \; J_{\alpha} t^{k-n-1}  }
\]
which is a well-defined element of $\hU_1$. 
The same formula may be used to define an element $S_k^{(t)}$ of $\hU_t$, and -- replacing $t$ by $s$ -- also an element $S_k^{(s)}$ of $\hU_s$.
\end{definition}

\begin{theorem}[{see \cite[Proposition 4.3.4]{Frenkel_Langlands_loop_group}}]\label{thm:CenterU1}
The operators $\Sug_k^{(1)}$ are algebraically independent, and the center $Z_1$ of $\hat{U}_1$ is the completed polynomial algebra $\widehat{\mathbb{C}[\Sug^{(1)}_k]}_{k \in \mathbb{Z}}$, where the completion is taken with respect to the topology generated by the ideals $(\Sug^{(1)}_k \mid k \geq n)$.
\end{theorem}

The following result is the main technical ingredient in the proof of Theorem \ref{thm:CenterU1}, and is the content of the proof of \cite[Proposition 4.3.4]{Frenkel_Langlands_loop_group}. 
Notice that the submodules $\hJ_1(n)$ are not invariant under the natural action of $\hat{\gog}_1$ on $\hU_1$. On the other hand, it is straightforward to see that $\hJ_1(n)$ is stable under the action of $\hat{\gog}_1^+$ for each $n \geq 0$. Therefore, the quotient $\hU_1/\hJ_1(n)$ inherits the structure of a $\hat{\gog}_1^+$-module.

\begin{lemma}\label{lemma:FreFei1}
The ring
$$
\frac{Z_1}{\hJ_1(N)\cap Z_1} 
$$
is a polynomial algebra over $\mC$ generated by $\Sug^{(1)}_k$ for $k<2N$. 
The natural inclusion induces isomorphisms
$$
\frac{Z_1}{\hJ_1(N)\cap Z_1} \simeq \left( \frac{\hU_1}{\hJ_1(N)} \right)^{\hgog_1^+}
$$
and 
$$
\Gr_{PBW} \left(\frac{Z_1}{\hJ_1(N)\cap Z_1} \right) \simeq \left( \Gr_{PBW}\frac{\hU_1}{\hJ_1(N)} \right)^{\hgog_1^+},
$$
where the associated graded rings are taken with respect to the PBW filtration of Section \ref{ssec:PBW}.
\end{lemma}
\begin{remark}\label{rmk:LnGenerateInDegreeOne}
In particular, it follows from \cite[Theorem 3.4.2]{Frenkel_Langlands_loop_group} that $Z_1^{\leq 1}$ is spanned by $1$ and $Z_1^{\leq 2}$ is spanned by $1$ and the operators $\Sug_k^{(1)}$ for $k \in \mathbb{Z}$.
\end{remark}

\begin{remark}\label{rmk:Zt}
Tensoring with $Q$ and observing that $\frac{\hat{U}_t}{\hJ_t(N)} \cong \frac{\hat{U}_1}{\hJ_1(N)} \otimes_{\mC} Q$ we obtain that a similar statement also holds for the centers $Z_t, Z_s$ (where $\mC$ is replaced by $Q$ everywhere).
\end{remark}

In the following lemma we use the notation introduced in Definition \ref{def:JRJR}.

\begin{lemma}\label{lemma:ExpansionOneVariableSugawara}
	The one-variable Sugawara operators $\Sug^{(1)}_{k}$ satisfy the following congruences:
	\begin{enumerate}
		\item $k=2j$ even:
		\[
		\Sug_k^{(1)} \equiv 2 \sum_{\alpha} J_{\alpha} t^{j-1} \, J^\alpha t^{j} \bmod \hJ_1^{\leq 2}[t^{j - 1}, t^{j}]
		\]
		\item $k=2j+1$ odd:
		\[
		\Sug_k^{(1)} \equiv \sum_{\alpha} J_{\alpha} t^{j} \, J^\alpha t^{j} \bmod \hJ_1^{\leq 2}[t^{j}, t^{j}]
		\]
	\end{enumerate}
Similar congruences hold for $\Sug_k^{(t)}$ and $\Sug_k^{(s)}$.
\end{lemma}

\subsection{Normal ordered product for the case of two singularities}

To simplify the notation we parametrise the bases $u_n$, $v_n$ and $x_n$, $y_n$ of Section \ref{sez:base} in a uniform way, as follows:
\begin{equation}\label{eq:defwz}
w_m = \begin{cases}
u_m, \text{ if } m \in \mathbb{Z} \\
v_{m-1/2}, \text{ if } m \in \frac{1}{2} + \mathbb{Z}
\end{cases}
\text{ and }\quad 
z_n = \begin{cases}
x_n, \text{ if } n \in \mathbb{Z} \\
y_{n-1/2}, \text{ if } n \in \frac{1}{2} + \mathbb{Z}
\end{cases}
\end{equation}
With this notation we may now extend the definition of the normal ordered product to the case of two singularities. 

\begin{definition}\label{def:nop2}
For $m_1, m_2 \in \frac{1}{2}\mathbb{Z}$ we define
\[
\nop{ (J^\alpha w_{m_1}) (J^\beta w_{m_2})} \, = \begin{cases}
(J^\alpha w_{m_1}) (J^\beta w_{m_2}) \text{ if } m_1 \leq m_2 \\
(J^\beta w_{m_2} ) (J^\alpha w_{m_1}) \text{ if } m_1 > m_2 \\
\end{cases}
\]
Since the elements $\{J^\alpha w_m\}_{m, \alpha}$ are $A$-free, this definition extends $A$-linearly to the $A$-span of $\{J^\alpha w_m\}_{m, \alpha}$. Every element of the form $ew_m $ for $e \in \gog$ lies in this span, and one checks that
\[
\nop{ (e_1w_{m_1}) (e_2w_{m_2})} \, = \begin{cases}
( e_1w_{m_1}) (e_2w_{m_2}) \text{ if } m_1 \leq m_2 \\
(e_2w_{m_2} ) (e_1w_{m_1}) \text{ if } m_1 > m_2 \\
\end{cases}
\]
for arbitrary $e_1, e_2 \in \gog$.
\end{definition}

\subsection{Properties of the Casimir element}\label{sez:Casimiro}
In this section we collect some identities related to the Casimir element $\sum_\alpha J^\alpha\otimes  J_\gra\in \gog \otimes \gog$ that we will use often in what follows. Since the Casimir element does not depend on the choice of the basis $J^\gra$,
for fixed $f,g \in K_2$ we have that also the sum $\sum_\gra (J^\gra f) (J_\gra g)\in \hU_2$ does not depend on this choice. Moreover, it also follows that $\sum_\gra[J^\gra, J_\gra]=0$, and hence also that
$\sum_{\gra}\kappa([J^\gra,e],J_\gra)=0$ for all $e\in \gog$. It is also easy to check that 
for all $e\in \gog$ one has 
	$$\sum_\gra [J^\gra,[J_\gra,e]]=e\quad \text{ and }\quad
	\sum_\alpha [J^\alpha, e] \otimes J_\alpha + \sum_\alpha J^\alpha \otimes [J_\alpha,e] = 0.$$

The identities in the next lemma are easily proved from the definitions by rewriting $y_n$ as $au_n +v_n$ and applying the properties of the Casimir element just recalled.

\begin{lemma}\label{lem:SplitTheSumsWithY}
	For all $h,k\in\mZ$ and $m,n \in \frac12 \mathbb{Z}$ we have:
	\begin{itemize}
		\item $\nop{(e_1 u_h )(e_2 y_k )} \, = \begin{cases}
		(e_1 u_h )(e_2 y_k ) \text{ if } h \leq k \\
		(e_2 y_k )(e_1 u_h ) \text{ if } h > k \\
		\end{cases} $ for all $e_1, e_2 \in \gog$.
		\item $
		\displaystyle \sum_{\alpha} \nop{(J^\alpha w_n)(J_\alpha z_m)} \, = \begin{cases}
		\sum_{\alpha} (J^\alpha w_n)(J_\alpha z_m) \text{ if } n \leq m \\
		\sum_\alpha (J_\alpha z_m)(J^\alpha w_n) \text{ if } n > m \\
		\end{cases} 
		$
	\end{itemize}
\end{lemma}

Similarly one can prove the following relations in $\hU_2$.

\begin{lemma}\label{lem:casimirouu}
	Let $f,g$ be elements of $K_2$. In $\hU_2$ we have
	$
	\sum_{\alpha} [J^\alpha f, J_\alpha g] =  -\frac{\dim \gog}{2} \cdot \operatorname{Res}_2(f'g).
	$
	In particular, by the formulas in Remark \ref{rem:Res2Derivatives} we have
\begin{enumerate}
	\item $\sum_{\alpha} [J^\alpha u_n, J_\alpha u_m] = -\dim \gog \cdot n\delta_{m,-n}$
	\item $\sum_{\alpha} [J^\alpha u_n, J_\alpha v_m] =  \frac{\dim \gog}{2} na \delta_{m,-n}$
	\item $\sum_{\alpha} [J^\alpha u_n, J_\alpha y_m] =  -\frac{\dim \gog}{2} na \delta_{m,-n}$
	\item $\sum_{\alpha} [J^\alpha v_n, J_\alpha v_m] =  -\frac{\dim \gog}{2} \left( (2n+1)\delta_{m,-1-n} + a^2n\delta_{m,-n} \right)$
\end{enumerate}
\end{lemma}

\subsection{Statement of the result and strategy of proof}

In analogy with the Sugawara operators of Definition \ref{def:1Sugawara} we introduce the following variant for the algebra $\hU_2$:
\begin{definition}\label{def:Sugawara}
For $k \in \frac{1}{2}\mathbb{Z}$ we  denote by $\Sug_{k}^{(2)}$ the element of $\hU_2$ given by
\[
\Sug_{k}^{(2)} := \sum_{n \in \frac{1}{2}\mathbb{Z},\alpha} \nop{J^\alpha w_n \; J_\alpha z_{-n-\frac{1}{2}}w_k }
\]
\end{definition}

We will show that these operators are central:
\begin{proposition}\label{prop:LkAreCentralPartOne}
For every $k \in \frac{1}{2}\mathbb{Z}$ the element $\Sug_k^{(2)}$ lies in the centre of $\hat{U}_2$.
\end{proposition}

We postpone the proof of this result to Section \ref{sec:2SugawaraAreCentral} below. 
Having produced an abundance of central elements, we will then show that the operators $\Sug^{(2)}_k$ do in fact (topologically) generate the center of $\hat{U}_2$. We will need the following definition:
\begin{definition}
Let $A[X_k]$ be the polynomial algebra in infinitely many variables indexed by $k \in \frac{1}{2}\mathbb{Z}$. We equip $A[X_k]$ with the topology for which a basis of neighbourhoods of $0$ is given by the ideals $\left( X_i \bigm\vert i \geq n \right)$, and we let $\widehat{A[X_k]}$ be the corresponding completion.
\end{definition}

\begin{remark}
The completion $\widehat{A[X_k]}$ can be described as follows: its elements are formal power series $f$ in $A[[X_k]]$ with the property that, for all $n \in \frac{1}{2}\mathbb{Z}$, evaluating $f$ at $X_m=0$ for all $m \geq n$ yields a polynomial.
\end{remark}

Assuming Proposition \ref{prop:LkAreCentralPartOne}, the elements $\Sug^{(2)}_k$ are central in $\hat{U}_2$, so there is a ring morphism $\phi : A[X_k] \to Z_2$ sending $X_k$ to $\Sug^{(2)}_k$. The map $\phi$ is continuous, hence it extends to a map (again denoted $\phi$) from $\widehat{A[X_k]}$ to $Z_2$: this follows from the fact that $\phi(X_k)=\Sug^{(2)}_k$ belongs to $\hJ_2( \lfloor \frac{k}{2} \rfloor )$. 
The following theorem is the main result of this section.

\begin{theorem}\label{thm:CentreU2}
The map $\phi : \widehat{A[X_k]} \to Z_2$ is an isomorphism. In particular, the operators $\Sug_k^{(2)}$ are algebraically independent and topologically generate $Z_2$.
\end{theorem}

Our strategy to prove Theorem \ref{thm:CentreU2} may be summarised as follows. We observe that both $\widehat{A[X_k]}$ and $Z_2$ have no $a$-torsion, hence are flat over $A$, so by Lemma \ref{lem:isoMN} the map $\phi$ is an isomorphism if and only if the induced maps 
\[
\overline{\phi} : \frac{\widehat{A[X_k]}}{ a\widehat{A[X_k]}} \to \frac{Z_2}{aZ_2}
\]
and
\[
\phi_a:\widehat{A[X_k]}[a^{-1}] \to Z_2[a^{-1}]
\]
are isomorphisms.
In Lemma \ref{lemma:IsomorphismAlongTheDiagonal} we show 
that $\overline{\phi}$ is an isomorphism; this follows easily from Theorem \ref{thm:CenterU1}.
To prove the second isomorphism we will use the expansion map $E:\hU_2[a^{-1} ]\lra \hU_{t,s}$ and the fact that $\hat{U}_{t,s}$ is a suitable completion of $\hat{U}_t \otimes \hat{U}_s$. We will then be reduced to checking a statement concerning the center of two copies of the algebra $\hU_1$, which we will again be able to reduce to the 1-variable case. This part of the proof will be carried out in Section \ref{sec:centrofuoridiagonale}.

\begin{lemma}\label{lemma:IsomorphismAlongTheDiagonal}
The map $\overline{\phi} : \widehat{A[X_k]}/a\widehat{A[X_k]} \to Z_2/aZ_2$ is an isomorphism.
\end{lemma}
\begin{proof}
Observe that upon specialising $a$ to $0$ the bases $w_n, z_n$ of Equation \eqref{eq:defwz} both become $\Sp(w_n)=\Sp(z_n)=t^{2n+1}$. It follows that for each $k \in \frac{1}{2}\mathbb{Z}$ the Sugawara operator $\Sug^{(2)}_{k}$ of Definition \ref{def:Sugawara} specialises to 
\[
\begin{aligned}
\Sp(\Sug^{(2)}_{k}) & = \sum_{n \in \frac{1}{2}\mathbb{Z}, \alpha} \nop{J^{\alpha} \Sp(w_n) J_\alpha \Sp(z_{-n-\frac{1}{2}} w_k) } \\
& = \sum_{m \in \mathbb{Z}, \alpha} \nop{J^{\alpha} t^{m+1} \; J_\alpha t^{-m-1}  t^{2k} }
\end{aligned}
\]
which is the $1$-variable Sugawara operator $\Sug^{(1)}_{2k+1}$ of Definition \ref{def:1Sugawara}. 
Now notice that we have an obvious injection ${Z_2}/{aZ_2} \hookrightarrow Z(\hat{U}_2/a\hat{U}_2)$ and an isomorphism $Z(\hat{U}_2/a\hat{U}_2) \cong Z(\hU_1) = Z_1$.
The composition
\begin{equation}\label{eq:whZ21}
\frac{\widehat{A[X_k]}} { a\widehat{A[X_k]}}  = \widehat{\mathbb{C}[X_k ]} \xrightarrow{\overline{\phi}} \frac{Z_2}{aZ_2} \hookrightarrow Z\bigg(\frac{\hat{U}_2}{a\hat{U_2}}\bigg) \cong Z_1
\end{equation}
is an isomorphism by Theorem \ref{thm:CenterU1} (notice that $X_k$ is sent to $\Sp(\Sug_k^{(2)})=\Sug_{2k+1}^{(1)}$), hence all intermediate maps are as well, and in particular $\overline{\phi}$ is an isomorphism as claimed.
\end{proof}

We denote by $\goz$ the image of $\phi$ and, for $N\in \mZ$, we define $\goz_{< N}$ to be 
the image of $A[X_i:i< 2N]$. With a similar argument we also prove:

\begin{lemma}\label{lem:goz1}
The following hold:
\begin{enumerate}[\indent a)]
\item $\goz\cap a\hU_2=\goz\cap aZ_2=a\goz$;
\item for $N\in \mN$ we have $\big(\goz_{< N} + (\hJ_2(N)\cap Z_2)\big)\cap aZ_2=a\big(\goz_{< N} + (\hJ_2(N)\cap Z_2)\big)$.
\end{enumerate}
\end{lemma}
\begin{proof}
To prove a) notice that, if $\phi(x)=au$ for some $u \in \hU_2$, then $u$ is central and $\Sp(\phi(x))=0$. Since by Theorem \ref{thm:CenterU1} the composition of the morphisms in \eqref{eq:whZ21} is an isomorphism, we deduce that $x=ay$, and since $\hU_2$ is torsion-free we deduce that $u=\phi(y)$.

For b), let $x=\phi(y)+j=au$ with $y\in A[X_i:i< 2N]$, $j\in \hJ_2(N)\cap Z_2$ and $u\in aZ_2$. Then $\Sp(x)=0$.  
Now notice that $\Sp(\goz_{<N})$ and $\Sp(\hJ_2(N))\subset \hJ_1(2N)$ have trivial intersection, hence $\Sp(\phi(y))=\Sp(j)=0$. As in the proof of a) this implies $y=az$ with $z\in A[X_i:i< 2N]$, and by part b) of Lemma \ref{lem:piatto} it also implies $j=ah$ with $h\in \hJ_2(N)$, proving the claim.
\end{proof}

\subsection{Action of $\Der_2$ on the Sugawara operators}
We now collect some information about the action of the Lie algebra $\Der_2$ on
Sugawara operators. The formulas we now derive will be necessary to prove Theorem
\ref{thm:isomorfismo}, and will also allow us to simplify some computations in the 
proof of the centrality of the operators $\Sug_k^{(2)}$. 
Recall from Section \ref{ssec:Liealgebras} that the Lie algebra $\Der_2$ acts on $U_2$ and $\hU_2$. In the next lemma we explicitly compute the action of certain elements of $\Der_2$ on Sugawara operators.

\begin{lemma}\label{lemma:DerivativesOfL}
The following hold for all $k \in \mathbb{Z}$:
\begin{enumerate}
\item for all integers $m$ we have
\[
(u_m \partial) \Sug^{(2)}_k = 2(k-m)\Sug^{(2)}_{k+m-\frac{1}{2}} + (k-m)a \Sug^{(2)}_{k+m-1} + f_1(k,m) \delta_{m+k,1} + f_2(k,m) \delta_{m+k,2}
\]
for some elements $f_1(k,m)$, $f_2(k,m)$ of $\mathbb{Q}[a]$. In addition, for $m=0$ we have $f_1(k,0)=f_2(k,0)=0$ for all $k$, while for $m=2$ we have $f_1(-1,2)=6$.
\item $\partial \Sug^{(2)}_{k-\frac{1}{2}} = (2k-1)\Sug^{(2)}_{k-1} - a(k-1)\Sug^{(2)}_{k-\frac{3}{2}}$.
\item $s\partial \Sug^{(2)}_k = (2k-1)\Sug^{(2)}_k - ka\Sug^{(2)}_{k-\frac{1}{2}}$.
\item $s\partial \Sug^{(2)}_{k-\frac{1}{2}} =(2k-2)\Sug^{(2)}_{k-\frac{1}{2}} - a(k-1)\Sug^{(2)}_{k-1} + a^2(k-1)\Sug^{(2)}_{k-\frac{3}{2}}$.
\item $ts \partial \Sug^{(2)}_{-1/2}=-3\Sug^{(2)}_{0}+2a\Sug^{(2)}_{-1/2}$.
\item $ts^2\partial \Sug_{-1}^{(2)} = -5 \Sug^{(2)}_0 + 2a \Sug_{-1/2}^{(2)}$.
\end{enumerate}
\end{lemma}

\begin{remark}\label{rmk:ActionOfDerivationsOn2Sugawara}
It is not too hard to show directly that $f_1(k,m)=m(m-1)(2m-1)$ and $f_2(k,m)=\frac{1}{2}a^2 m(m-1)(m-2)$ for all $k,m$ (including $m < 0$). However, this would require some additional computations, and the information provided by this lemma is enough for our purposes. Moreover, one could obtain a complete description of the action of $\Der_2$ on the operators $\Sug_k^{(2)}$ as a consequence of Theorem \ref{thm:isomorfismo} and Lemma \ref{lemma:DerivativesOfAlphaBeta}.
\end{remark}

\begin{proof}
The various cases require very similar computations, so we only give details for (1), assuming in addition $m \geq 1$ (we will point out below what the differences are in the case $m \leq 0$).
For the purposes of this proof we find it easier to let all sums run over integer values of $n$. We therefore begin by rewriting the operators $\Sug_k^{(2)}$ (for $k \in \mathbb{Z}$) so that each sum is indexed by $n \in \mathbb{Z}$ (subject to suitable inequalities): we have
\begin{equation}\label{eq:DecompositionOfSugawara}
\begin{aligned}
\Sug_k^{(2)} = & \underbrace{\sum_{2n \leq k-1, \alpha} J^\alpha u_n \; J_\alpha y_{k-n-1}}_I + \underbrace{\sum_{2n \leq k-2, \alpha} J^\alpha v_n \; J_\alpha u_{k-n-1}}_{II} \\
& + \underbrace{\sum_{2n > k-1, \alpha} J_\alpha y_{k-n-1} \; J^\alpha u_n}_{III} + \underbrace{\sum_{2n > k-2, \alpha} J_\alpha u_{k-n-1} \; J^\alpha v_n}_{IV}.
\end{aligned}
\end{equation}
We start by computing the action of $u_m\partial$ on the sum labelled $I$ in Equation \eqref{eq:DecompositionOfSugawara}: we have
\[
\begin{aligned}
(u_m\partial) & \sum_{2n \leq k-1, \alpha} J^\alpha u_n \; J_\alpha y_{k-n-1} =  \sum_{2n \leq k-1} J^\alpha (2nv_{n+m-1} + na u_{n+m-1}) \; J_\alpha y_{k-n-1} \\
& \quad+ \sum_{2n \leq k-1} J^\alpha u_n \; J_\alpha \left( (2(k-n-1)+1) u_{m+(k-n-1)} + (k-n-1)ay_{m+(k-n-1)-1} \right).
\end{aligned}
\]
Replacing $J^{\alpha} v_{n+m-1} \; J_\alpha y_{k-n-1}$ with $J^{\alpha} v_{n+m-1} \; \left( J_\alpha v_{k-n-1} + J_\alpha au_{k-n-1} \right)$ and relabelling the summation index, we rewrite this as 
\[
\begin{aligned}
\sum_{2(n-m+1) \leq k-1, \alpha} & 2(n-m+1) J^\alpha v_{n}\; J_\alpha v_{k-(n-m+1)-1}
\\ & + a\sum_{2(n-m+1) \leq k-1, \alpha} 2(n-m+1) J^\alpha v_{n}\; J_\alpha u_{k-(n-m+1)-1} \\
& + a\sum_{2(n-m+1) \leq k-1, \alpha} (n-m+1) J^\alpha u_{n} \; J_\alpha y_{k-(n-m+1)-1} \\
& + a \sum_{2n \leq k-1, \alpha} (k-n-1) J^\alpha u_n \; J_\alpha y_{m+k-n-2} \\
& + \sum_{2n \leq k-1, \alpha} (2k-2n-1) J^\alpha u_n \; J_\alpha u_{m+(k-n-1)}.
\end{aligned}
\]
We similarly compute the action of $u_m\partial$ on the sum labelled $II$ in \eqref{eq:DecompositionOfSugawara} and shift indices to find
\[
\begin{aligned}
(u_m\partial) \sum_{2n \leq k-2, \alpha}  J^\alpha v_n \; J_\alpha u_{k-n-1} 
= & \sum_{2(n-m) \leq k-2, \alpha} (2n-2m+1)J^\alpha u_{n} \; J_\alpha u_{k-(n-m)-1} \\
& - a\sum_{2(n-m+1) \leq k-2} (n-m+1) J^\alpha v_{n} \; J_\alpha u_{k-(n-m+1)-1} \\
& + \sum_{2n \leq k-2, \alpha} 2(k-n-1) J^\alpha v_n \; J_\alpha v_{(k-n-1)+m-1} \\
& + a \sum_{2n \leq k-2, \alpha} (k-n-1) J^\alpha v_n \; J_\alpha u_{(k-n-1)+m-1}.
\end{aligned}
\]
We now sum these two contributions, splitting each sum into a ``main" (infinite) sum and a finite ``error" sum. More precisely, we include in the main term \textit{all} products of a given shape that are in a normal order, and define the error term so as to obtain an equality. 
For example, we rewrite $\sum_{2n \leq k+2m-2, \alpha} (2n-2m+1) J^\alpha u_{n} \; J_\alpha u_{k-(n-m)-1}$ as  
\[
\sum_{2n \leq k+m-1, \alpha } (2n-2m+1)J^\alpha u_{n} \; J_\alpha u_{k+m-n-1} + \sum_{\substack{k+m-1 <2n \\ 2n \leq k+2m-2 \\ \alpha}} (2n-2m+1)J^\alpha u_{n} \; J_\alpha u_{k+m-n-1}.
\]
Here the infinite series includes all products of the form $(2n-2m+1) \, J^\alpha u_{n} \; J_\alpha u_{k+m-n-1}$ that are in normal order: these are precisely those for which $n \leq k+m-n-1$, that is, $2n \leq k+m-1$. The same decomposition also holds for $m \leq 0$, with the only difference involving the error term $\sum_{\substack{k+m-1 <2n \\ 2n \leq k+2m-2}} (2n-2m+1)J^\alpha u_{n} \; J_\alpha u_{k+m-n-1}$: when $m \leq 0$, this should be interpreted as the opposite of the sum $\sum_{\substack{k+m-1 \geq 2n \\ 2n > k+2m-2}} (2n-2m+1)J^\alpha u_{n} \; J_\alpha u_{k+m-n-1}$. Similar comments apply to all main and error terms. Taking into account these small differences, the rest of the proof goes through unchanged also for $m \leq 0$.

Notice that some error sums carry a negative sign: for example, we rewrite $\sum_{2n \leq k-1, \alpha} (2k-2n-1) J^\alpha u_n \; J_\alpha u_{m+k-n-1}$ as
\[
\sum_{2n \leq k+m-1, \alpha} (2k-2n-1) J^\alpha u_n \; J_\alpha u_{m+k-n-1}-\sum_{k-1 <2n \leq k+m-1, \alpha} (2k-2n-1) J^\alpha u_n \; J_\alpha u_{m+(k-n-1)},
\]
so that again the first sum includes \textit{all} products of the form $(2k-2n-1) J^\alpha u_n \; J_\alpha u_{m+k-n-1}$ that are in normal order.

Since $(u_m\partial)(I)$ consists of 5 sums, and $(u_m\partial)(II)$ consists of 4 sums, we thus obtain 9 main terms and 9 error terms. The main terms easily combine to yield
\[
\begin{aligned}
2& (k-m) \sum_{2n \leq k+m-2, \alpha} J^\alpha v_{n} \; J_\alpha v_{k-(n-m+1)-1} + 2(k-m) \sum_{2n \leq k+m-1, \alpha} J^\alpha u_n \; J_\alpha u_{m+(k-n-1)}  \\
& + (k-m)a \sum_{2n \leq k+m-3, \alpha} J^\alpha v_{n} \; J_\alpha u_{k-(n-m+1)-1} + (k-m)a \sum_{2n \leq k+m-2, \alpha} J^\alpha u_n \; J_\alpha y_{m+k-n-2}.
\end{aligned}
\]

We now consider $(u_m \partial)(III + IV)$, which involves sums analogous to those appearing in $(u_m \partial)(I + II)$, but with the index running over the complementary set, and with every product ordered in the opposite way. Thus $(u_m \partial)(III + IV)$ also contributes 9 main terms, similar to those above, but indexed by the complementary set of integers $n$, and with the two factors of each summand appearing in the opposite order.
It follows that the main terms in $(u_m \partial)\Sug_k^{(2)} = (u_m \partial)(I + II + III + IV)$ sum to
\[
\begin{aligned}
2& (k-m) \sum_{n, \alpha} \nop{J^\alpha v_{n} \; J_\alpha v_{m+k-n-2}} + 2(k-m) \sum_{n, \alpha} \nop{J^\alpha u_n \; J_\alpha u_{m+k-n-1}}  \\
& + (k-m)a \sum_{n, \alpha} \nop{J^\alpha v_{n} \; J_\alpha u_{m+k-n-2}} + (k-m)a \sum_{n, \alpha} \nop{J^\alpha u_n \; J_\alpha y_{m+k-n-2}} \\
& = 2(k-m)\Sug^{(2)}_{k+m-\frac{1}{2}} + (k-m)a \Sug^{(2)}_{k+m-1}.
\end{aligned}
\]
On the other hand, each error sum in $(u_m \partial)(III + IV)$ has the \textit{opposite} sign with respect to the corresponding error sum in $(u_m \partial)(I + II)$: indeed, the terms missing on one side are over-counted on the other, and vice-versa. Thus the 18 error terms come in pairs, with each pair comprising two sums (over the same finite set of indices) that differ only for a sign and for the ordering of the two factors in each summand. In other words, each pair of error sums combines to give a sum of commutators. Explicitly, one gets
\begin{equation}\label{eq:ExplicitErrorTerms}
\begin{aligned}
\sum_{k+m-2 < 2n \leq k+2m-3, \alpha} & 2(n-m+1) [J^\alpha v_{n}, J_\alpha v_{k-(n-m+1)-1}] \\
& \hspace{-2cm} - \sum_{k-2 < 2n \leq k+m-2, \alpha} 2(k-n-1) [J^\alpha v_n, J_\alpha v_{(k-n-1)+m-1}] \\
& \hspace{-2cm} + \sum_{k+m-1 <2n \leq k+2m-2, \alpha} (2n-2m+1)[J^\alpha u_{n}, J_\alpha u_{k-(n-m)-1}] \\
& \hspace{-2cm} -\sum_{k-1 <2n \leq k+m-1, \alpha} (2k-2n-1) [J^\alpha u_n, J_\alpha u_{m+(k-n-1)}] \\
& \hspace{-2cm} + a \sum_{k+m-3 < 2n \leq k+2m-3, \alpha} 2(n-m+1) [J^\alpha v_{n}, J_\alpha u_{k-(n-m+1)-1}] \\
& \hspace{-2cm} - a\sum_{k+m-3<2n\leq k+2m-4, \alpha} (n-m+1) [J^\alpha v_{n}, J_\alpha u_{k-(n-m+1)-1}] \\
& \hspace{-2cm} - a \sum_{k-2 <2n \leq k+m-3, \alpha} (k-n-1) [J^\alpha v_n, J_\alpha u_{(k-n-1)+m-1}] \\
& \hspace{-2cm} + a\sum_{k+m-2 < 2n \leq k+2m-3, \alpha} (n-m+1) [J^\alpha u_{n}, J_\alpha y_{k-(n-m+1)-1}] \\
& \hspace{-2cm} - a \sum_{k-1 <2n \leq k+m-2, \alpha} (k-n-1) [J^\alpha u_n, J_\alpha y_{m+k-n-2}]
\end{aligned}
\end{equation}
To finish the proof, we notice that by the formulas in Lemma \ref{lem:casimirouu} each of the sums in Equation \eqref{eq:ExplicitErrorTerms} vanishes, unless the indices of the two functions involved sum to either $0$ or $-1$. 
Thus, in particular, we see that all the error terms vanish unless $k+m=1$ or $k+m=2$. In all cases, even if $k=1-m$ or $k=2-m$, each of the sums in \eqref{eq:ExplicitErrorTerms} is a linear combination of rational multiples of $1, a, a^2$. We have thus proven that for $k \in \mathbb{Z}$ and $m \geq 1$ the following equality holds:
\[
\begin{aligned}
(u_m \partial) \Sug_k^{(2)}= 2(k-m)\Sug^{(2)}_{k+m-\frac{1}{2}} + (k-m)a \Sug^{(2)}_{k+m-1} + f_1(m,k) \delta_{m+k,1} + f_2(m,k) \delta_{m+k,2},
\end{aligned}
\]
where $f_1(m,k)$ and $f_2(m,k)$ are polynomials of degree at most $2$ in $a$ with rational coefficients. For $m=0$, one checks without difficulty that the error sums vanish for all values of $k$.

Finally, when $m=2$ and $k=-1$, there is at most one value of $n$ that contributes to each of the error sums, so it is easy to check that in this case \eqref{eq:ExplicitErrorTerms} evaluates to
\[
\begin{aligned}
\sum_{\alpha} & 2(-1) [J^\alpha v_{0}, J_\alpha v_{-1}] - \sum_{\alpha} 2(-1) [J^\alpha v_{-1}, J_\alpha v_{0}] = -4 \sum_{\alpha} [J^\alpha v_{0}, J_\alpha v_{-1}] = 6.
\end{aligned}
\]
\end{proof}

\subsection{The operators $\Sug_{k}^{(2)}$ are central}\label{sec:2SugawaraAreCentral}
In this section we prove Proposition \ref{prop:LkAreCentralPartOne}, i.e., we show
that all the operators $\Sug_{k}^{(2)}$ are central in $\hat{U}_2$. In order to do this we need a series of lemmas.

\begin{lemma}\label{lem:Commutator}
Let $f,g,h \in K_2$, $e \in \gog$. Define
\[
\omega_1(f,g,h) := f \Res_2(gh')  + g \Res_2(fh')
\]
and
\[
\omega_2(f,g,h,e) := \sum_{\alpha}  ( [J^\alpha,e] fh) (J_\alpha g) - \sum_{\alpha}  ([J^\alpha,e]f ) (J_\alpha gh).
\]
Then
\[
\sum_{\alpha} [J^\alpha f \, J_\alpha g, eh]= -\omega_1(f,g,h) C_2 e  + \omega_2(f,g,h,e).
\]
\end{lemma}
\begin{proof}
We have
\[
\begin{aligned}
\sum_{\alpha} [J^\alpha f \, J_\alpha g, eh] & = 
\sum_{\alpha} J^\alpha f \, [ J_\alpha g, eh] + \sum_{\alpha}[J^\alpha f, eh] \, J_\alpha g \\
&= \sum_{\alpha} J^\alpha f \, \left( [J_\alpha,e]gh  - C_2 \, \kappa( J_\alpha,e ) \Res_2(gh') \right) 
\\
&\quad + \sum_{\alpha}\left( [J^\alpha,e]fh  - C_2 \kappa (J^\alpha,e ) \Res_2(fh') \right) \, J_\alpha g \\
& = -C_2 e \left( f \Res_2(gh')  + g \Res_2(fh') \right) \\ 
& \quad + \sum_{\alpha}  (J^\alpha f) (  [J_\alpha,e] gh) + \sum_{\alpha}  ( [J^\alpha,e] fh) (J_\alpha g) \\
& = -C_2 e \left( f \Res_2(gh')  + g \Res_2(fh') \right) \\ 
& \quad - \sum_{\alpha}  ([J^\alpha,e]f ) (J_\alpha gh) + \sum_{\alpha}  ( [J^\alpha,e] fh) (J_\alpha g),
\end{aligned}
\]
where we have used the properties of the Casimir element recalled at the beginning of Section \ref{sez:Casimiro}.
\end{proof}

We now study the function $\omega_2$ from Lemma \ref{lem:Commutator} for the special values of $f, g, h$ that are involved in the operators $\Sug^{(2)}_{k}$. It will be useful to employ the following notation:
\begin{notation}\label{not:ABCD}
Let $n$ be a half-integer. For fixed values of $k\in \frac{1}{2}\mathbb{Z}$ and $m\in\mZ$ we set
\[
A_{n} :=  \sum_{\alpha} \left( [J^\alpha,e] w_n u_m \right) \left( J_\alpha z_{-n-\frac 12}w_k \right) \quad \text{ and } \quad B_{n} := \sum_{\alpha} \left( [J^\alpha,e] w_n\right)\left( J_{\alpha} z_{-n-\frac 12}w_ku_m \right),
\]
so that $\omega_2(w_n,z_{-n-\frac 12}w_k,u_m,e)=A_{n} - B_{n}$. Similarly we set
\[
B'_{n} := \sum_{\alpha} \left([J^\alpha,e]z_{-n-\frac{1}{2}} w_ku_m\right) \left(J_\alpha w_n\right) \quad \text{ and } \quad A'_{n} := \sum_{\alpha} \left([J^\alpha,e] z_{-n-\frac{1}{2}} w_k\right)\left( J_\alpha w_nu_m \right),
\]
so that $\omega_2(z_{-n-\frac{1}{2}}w_k,w_n,u_m,e)=B'_{n}-A'_{n}$.
\end{notation}

We need one last identity satisfied by the quantities $A_n, B_n, A'_n, B'_n$:
\begin{lemma}\label{lemma:APlusD}
For every $n \in \frac{1}{2}\mathbb{Z}$ we have 
\[
A_n + A'_n = B_n+B'_n = -ew_ku_m w_nz_{-n-\frac{1}{2}} = -ew_k u_m \begin{cases}
y_{-1}, \text{ if } n \in \mathbb{Z} \\
v_{-1}, \text{ if } n \in \frac{1}{2}\mathbb{Z} \setminus \mathbb{Z}
\end{cases}
\]
\end{lemma}
\begin{proof}
We show the identity for $A_n+A'_n$, the case of $B_n+B'_n$ being virtually identical. Using the properties
of the Casimir element recalled in Section \ref{sez:Casimiro}
and the fact  that we can assume $J^\alpha=J_\alpha$ we may rewrite $A'_{n}$ as follows:
\[
\begin{aligned}
A'_{n} & = \sum_{\alpha} \left([J^\alpha,e] z_{-n-\frac{1}{2}} w_k\right)\left( J_\alpha w_nu_m \right) \\
& = - \sum_{\alpha} \left(J^\alpha z_{-n-\frac{1}{2}} w_k\right)\left( [J_\alpha,e] w_nu_m\right)\\
& = - \sum_{\alpha} \left(J_\alpha z_{-n-\frac{1}{2}} w_k\right)\left( [J^\alpha,e] w_nu_m\right),
\end{aligned}
\]
which is simply the opposite of the expression obtained from $A_{n}$ by exchanging the two factors in the universal enveloping algebra. Hence 
\[
\begin{aligned}
A_n+A'_n & = \sum_{\alpha} \left[ [J^\alpha,e] w_n u_m , J_\alpha z_{-n-\frac 12}w_k \right] \\
& =   \sum_{\alpha} \left( \left[ [J^\alpha,e], J_\alpha \right]w_n u_mz_{-n-\frac 12}w_k - \kappa \left( [J^\alpha,e], J_\alpha \right) \Res_2(w_n u_m (z_{-n-\frac 12}w_k) ' ) C_2 \right).
\end{aligned}
\]
Now recall that for every $e \in \mathfrak{g}$, we have $\sum_{\alpha} \kappa \left( [J^\alpha,e], J_\alpha \right)=0$  and $ \sum_{\alpha} \left[ [J^\alpha,e], J_\alpha \right]=-e$, so that we obtain
\[
A_n + A'_n = -ew_nu_mz_{-n-\frac{1}{2}}w_k;
\]
the lemma now follows by remarking that $w_nz_{-n-\frac{1}{2}}$ equals $y_{-1}$ or $v_{-1}$ according to whether $n$ is an integer of a half-integer.
\end{proof}

\begin{proof}[Proof of Proposition \ref{prop:LkAreCentralPartOne}]
We claim that $L := \Sug_{k}^{(2)}$ commutes with all elements of the form $eu_m$ for $m \in \mathbb{Z}$ and $e \in \gog$.
We begin by showing that this statement implies the proposition. Since $\hU_2$ is topologically generated by $\gog \otimes R_2$, it clearly suffices to show that $L$ also commutes with $ev_m$ for every $m \in \mathbb{Z}$ and every $e \in \gog$. By Lemma \ref{lemma:DerivativesOfL}, the derivative $\partial L$ is a linear combination of 2-Sugawara operators, so by the claim we have $[\partial L, eu_m]=0$, as well as $[L,eu_m]=0 \Rightarrow \partial [L, eu_m]=0$. The relation $u_m'=2mv_{m-1}+mau_{m-1}$ then gives
\[
0=\partial [L, e u_m]=[\partial L,eu_m]+[L,eu_m']=ma[L,eu_{m-1}]+2m[L,ev_{m-1}]=2m[L,ev_{m-1}].
\]
This implies $[L,ev_m]=0$ for all $m\neq -1$. Finally, the Jacobi identity yields
\[
0=[L,[eu_1,e'v_{-2}]]=[L,[e,e']v_{-1}]+[L,\kappa(e,e')\Res_2(u_1'v_{-2})C_2]=[L,[e,e']v_{-1}],
\]
which shows that $L$ also commutes with $\gog \otimes v_{-1}$, hence that it is central.

We now prove the claim. We need to show that the commutator
\[  
 [L,eu_m] = \sum_{n, \alpha} \, [ \, \nop{J^\alpha w_n J_\alpha z_{-n-\frac12} w_k}\ ,\  eu_m ]
\]
vanishes for all $m \in \mathbb{Z}$ and $e \in \gog$.
Here, and in what follows, the index $n$ in the sums is understood to run over $\tfrac 12 \mZ$. 
Notice that, if $n,k$ are not both integers, then $z_{-n-\frac 12}w_k=w_{k-n-\frac12}$, while for $n,k \in \mathbb{Z}$ we have $z_{-n-\frac 12}w_k=y_{k-n-1}$. By Lemma \ref{lem:SplitTheSumsWithY} we then get
\[  
 [L,eu_m] \, = \sum_{n\leq k-n-\frac 12,\  \alpha} [ J^\alpha w_n J_\alpha z_{-n-\frac12} w_k ,\  eu_m ] + \sum_{n> k-n-\frac 12,\  \alpha} [ J_\alpha z_{-n-\frac12}w_k J^\alpha w_n ,  eu_m ].
\]
Using Lemma \ref{lem:Commutator} and the fact that $\omega_1$ is symmetric with respect to its first two arguments we may rewrite this as
\begin{equation}\label{eq:CommutatorOriginal}
\begin{split}
[L,eu_m] = & -\sum_{n} \omega_1(w_n,z_{-n-\frac 12}w_k,u_m)C_2e \\ 
& +  \sum_{n\leq k-n-\frac 12} \omega_2(w_n,z_{-n-\frac 12}w_k,u_m,e) \ + \sum_{n> k-n-\frac 12} \omega_2(z_{-n-\frac12}w_k, w_n, u_m,e) .
\end{split}
\end{equation}
Equation \eqref{eq:resduale}
implies that
$
\sum_n \omega_1(w_n, z_{-n-\frac12}w_k, u_m) = 2w_ku_m',
$
so the first summand in \eqref{eq:CommutatorOriginal} is $-2w_k u_m' C_2e$.
We now compute the rest of the commutator,
 
which using Notation \ref{not:ABCD} can be rewritten as
\begin{equation}\label{eq:RestOfCommutator2}
\sum_{n\leq k-n-\frac 12} (A_n-B_n) + \sum_{n> k-n-\frac 12} (B'_n-A'_n).
\end{equation}

Since $A_n = B_{n+m}$ and $A'_n=B'_{n+m}$ hold for every $n \in \frac 12 \mathbb{Z}$ and $m \in \mathbb{Z}$, the infinite sums appearing in \eqref{eq:RestOfCommutator2} are telescopic, so that they may be written as finite combinations of $A_n, B_n, A'_n, B'_n$. One may easily check that \eqref{eq:RestOfCommutator2} is equal to \begin{equation}\label{eq:CommutatorWhenMIsInteger}
\sum_{ \substack{n > \frac{k}{2}-\frac{1}{4}-m \\ n \leq \frac{k}{2}-\frac{1}{4}}} (A_n + A'_n).
\end{equation}
Notice that here the sum should be interpreted using the same convention that is normally employed for integrals, namely, if $m < 0$ the symbol $\displaystyle \sum_{ \substack{n > \frac{k}{2}-\frac{1}{4}-m \\ n \leq \frac{k}{2}-\frac{1}{4}}}$ denotes the \textit{opposite} of the sum $\displaystyle \sum_{ \substack{n \leq \frac{k}{2}-\frac{1}{4}-m \\ n > \frac{k}{2}-\frac{1}{4}}}$.

By Lemma \ref{lemma:APlusD}, every summand appearing in \eqref{eq:CommutatorWhenMIsInteger} is either $-ew_ku_m y_{-1}$ or $-ew_ku_m v_{-1}$, according to whether the index $n$ is an integer or a half-integer. Hence it suffices to count how many summands of each type there are: taking into account our convention for sums in the case $m < 0$, the result is immediately found to be
\[
-ew_ku_m(my_{-1}+mv_{-1})=-ew_ku_m'
\]
Thus we see that \eqref{eq:RestOfCommutator2} is equal to $-ew_ku_m'$ and  \eqref{eq:CommutatorOriginal} is equal to $[L,eu_m] = e \, (-2C_2-1)w_ku_m'$, which as claimed vanishes at the critical level $C_2=-\frac{1}{2}$.
\end{proof}

\subsection{Expansion of the Sugawara operators}
In this section we determine the first terms in the expansion of the operators $\Sug_k^{(2)}$.

\begin{lemma}\label{lemma:ExpansionTwoVariablesSugawara}
The two-variables Sugawara operators $\Sug_{k}^{(2)}$ have the following expansions (with notation as in Definition \ref{def:JRJR})
\begin{enumerate}
\item for $k=2j$ with $j \in \mathbb{Z}$: 
\[
E(\Sug^{(2)}_{k}) \equiv  2 a^{2j} \sum_{\alpha} J_\alpha  s^{j-1} \, J^{\alpha} s^{j} \bmod \hJ_{t,s}^{\leq 2}[s^{j-1}, s^j] \, ;
\]
\item for $k=2j+1$ with $j \in \mathbb{Z}$:
\[
E(\Sug^{(2)}_{k}) \equiv  a^{2j+1} \sum_\alpha J^\alpha s^j J_\alpha s^j   \bmod \hJ_{t,s}^{\leq 2}[t^{j-1},t^j] \, ;
\]
\item for $k=2j+\frac{1}{2}$ with $j \in \mathbb{Z}$:
\[
E(\Sug^{(2)}_{k})   \equiv \sum_\alpha \Big(  a^{2j} J^\alpha s^j J_\alpha s^j +    2(-a)^{2j+1}J_\alpha t^{j-1} J^\alpha t^j  \Big) \bmod \hJ_{t,s}^{\leq 2}[t^{j-1},t^j] \, ;
\]
\item for $k=2j+\frac{3}{2}$ with $j \in \mathbb{Z}$:
\[
E(\Sug^{(2)}_{k}) \equiv \sum_\alpha  a^{2j+2}  J^\alpha t^j J_\alpha  t^j  \bmod \hJ_{t,s}^{\leq 2}[t^j,t^j] \, .
\]
\end{enumerate}
\end{lemma}
\begin{proof}
Lemma \ref{lem:hE1} gives the following expressions for the expansions of $J^\alpha\, u_n, J^\alpha\, v_n, J^\alpha\, y_n$:
\begin{enumerate}
\item $E(J^\alpha \,  u_n)\coinc  J^\alpha \,  (-a)^n t^n  +  J^\alpha \,  a^n s^n \bmod \hJ^{\leq 1}_{t,s}[t^n]$;
\item $E(J^\alpha \,  v_n) \coinc J^\alpha \,  (-a)^{n+1}t^n +  J^\alpha \,  a^ns^{n+1} \bmod \hJ^{\leq 1}_{t,s}[s^{n+1}]$;
\item $E(J^\alpha \, y_n)\coinc J^\alpha \, a^{n+1}s^n \bmod \hJ_{t,s}^{\leq 1}[t^n]$.
\end{enumerate}
Notice in particular that $E(J^\alpha \,  u_n) \in \hJ_{s,t}^{\leq 1}[t^{n-1}]$, $E(J^\alpha \,  v_n) \in \hJ_{t,s}^{\leq 1}[s^{n}]$ and $E(J^\alpha \,  y_n) \in \hJ_{t,s}^{\leq 1}[t^{n-1}]$.
Consider now the expansion of a 2-variables Sugawara operator
\begin{equation}\label{eq:SugawaraExpansion}
\begin{aligned}
E(\Sug^{(2)}_k) & = \sum_{n \in \frac{1}{2}\mathbb{Z},\alpha} E\left(  \nop{J^\alpha w_n \; J_\alpha z_{-n-\frac{1}{2}}w_k } \right) \\
& = \sum_{2n \leq k - \frac{1}{2}} \sum_\alpha E\left( J^\alpha w_n \; J_\alpha z_{-n-\frac{1}{2}}w_k  \right) +  \sum_{2n > k - \frac{1}{2}} \sum_\alpha E\left( J_\alpha z_{-n-\frac{1}{2}}w_k \; J^\alpha w_n   \right).
\end{aligned}
\end{equation}
The four cases in the statement are treated similarly. We only give details for the third one, as it is slightly more complicated.
Let $k=2j+\frac{1}{2}$ with $j \in \mZ$.
The condition $2n \leq k -\frac 12$ is equivalent to $n \leq j$. We claim that all summands in \eqref{eq:SugawaraExpansion} belong to $\hJ_{t,s}^{\leq 2}[t^{j-1},t^j]$, with the exception of those corresponding to $n \in \{j-\frac 12, j, j+\frac 12\}$. Indeed, if $n \geq j + 1$ we know that $E(J^\alpha w_n) \in \hJ_{t,s}^{\leq 1}[t^j]$, and Lemma \ref{lem:filtrazioneJmn} part c) implies that the corresponding summand in \eqref{eq:SugawaraExpansion}, namely $\sum_{\alpha} E\left( J_\alpha z_{-n-\frac{1}{2}}w_k \; J^\alpha w_n   \right)$, lies in $\hJ_{t,s}^{\leq 2}[t^{j-1},t^j]$. Similarly, if $n \leq j-1$ then $k-n-\frac 12 \geq j+1$, and therefore
\[  
E(J_\alpha z_{-n-\frac 12}w_k) = E( J_\alpha w_{k-n-\frac 12} ) \in \hJ_{t,s}^{\leq 1}[t^j],
\]
and again by Lemma \ref{lem:filtrazioneJmn} part c) the corresponding summand belongs to $\hJ_{t,s}^{\leq 2}[t^{j-1},t^j]$. 
The remaining terms are:
\begin{enumerate}
\item for $n=j-\frac{1}{2}$ we have $\sum_\alpha E\Big( J^\alpha w_{j-\frac 12} \, J_\alpha w_{j+\frac 12}   \Big) = \sum_\alpha E\Big( J^\alpha w_{j-\frac 12}\Big) E\Big( J_\alpha w_{j+\frac 12}   \Big)$, which belongs to $\sum_\alpha \Big(J^\alpha (-a)^j t^{j-1}  + \hJ_{t,s}^{\leq 1}[t^{j-1}] \Big) \Big(  J_\alpha (-a)^{j+1}t^j  + \hJ_{t,s}^{\leq 1}[t^{j}]   \Big) $;
\item for $n=j$ we have $\sum_\alpha E\Big( J^\alpha w_{j} \, J_\alpha w_{j}   \Big) = \sum_\alpha E\Big( J^\alpha w_{j}\Big) E\Big( J_\alpha w_{j}   \Big)$, which belongs to $\sum_\alpha \Big(  J^\alpha a^js^j +  J^\alpha (-a)^j t^{j}  + \hJ_{t,s}^{\leq 1}[t^{j}] \Big) \Big(   J_\alpha a^js^j +  J_\alpha (-a)^j t^{j}  + \hJ_{t,s}^{\leq 1}[t^{j}]   \Big) $;
\item for $j=n+\frac{1}{2}$ we have $\sum_\alpha E\Big( J_\alpha w_{j-\frac 12} \, J^\alpha w_{j+\frac 12}   \Big) = \sum_\alpha E\Big( J_\alpha w_{j-\frac 12}\Big) E\Big( J^\alpha w_{j+\frac 12}   \Big)$, which belongs to $\sum_\alpha \Big(J_\alpha (-a)^j t^{j-1}  + \bJ_{t,s}^{\leq 1}[t^{j-1}] \Big) \Big(  J^\alpha (-a)^{j+1}t^j  + \hJ_{t,s}^{\leq 1}[t^{j}]   \Big)$.
\end{enumerate}
The sum of these terms is congruent modulo $\hJ_{t,s}^{\leq 2}[t^{j-1},t^j]$ to
\begin{align*}
E(\Sug^{(2)}_{k}) &  \equiv \sum_\alpha \Big(  a^{2j} J^\alpha s^j \, J_\alpha s^j +    2(-a)^{2j+1}J_\alpha t^{j-1} \, J^\alpha t^j  \Big)
\end{align*}
as claimed.
\end{proof}

Notice in particular that for every $j \in \mZ$ we get
\[
E\left( \Sug^{(2)}_{2j+1}-a\Sug^{(2)}_{2j+\frac 12}     \right) \equiv  2a^{2j+2}\sum_\alpha  J_\alpha t^{j-1} J^\alpha t^j  \bmod \hJ_{t,s}^{\leq 2}[t^{j-1},t^j].
\]

\begin{definition}
For $k \in \frac 12 \mZ$ we let $\Sstorto_k$ be the following $Q$-linear combination of operators $\Sug_k^{(2)}$:
\begin{enumerate}
\item $\Sstorto_k=a^{-\lceil k \rceil} \Sug_k^{(2)}$, if $k \in \mathbb{Z}$ or $k=2j+\frac{3}{2}$ for $j \in \mathbb{Z}$;
\item $\Sstorto_k=a^{-2j-2} \left( \Sug^{(2)}_{2j+1}-a\Sug^{(2)}_{2j+\frac{1}{2}} \right)$, if $k=2j+\frac{1}{2}$ for $j \in \mathbb{Z}$.
\end{enumerate}
\end{definition}

\begin{remark}\label{rmk:LeadingTermLStorti}
It follows from Lemma \ref{lemma:ExpansionTwoVariablesSugawara} and Lemma \ref{lemma:ExpansionOneVariableSugawara} that 
\[   
\lt(E(\Sstorto_k)) =
\begin{cases}
\lt(\Sug_k^{(s)}) & \mathrm{ if } \ k \in \mZ \\
\lt(\Sug_{k-\frac 12}^{(t)}) & \mathrm{ if } \  k \notin \mZ.
\end{cases}
\]
\end{remark}

\subsection{Proof of Theorem \ref{thm:CentreU2}}\label{sec:centrofuoridiagonale}
In this section we prove that $\phi_a:\widehat  {A[X_i]}[a^{-1}] \lra Z_2[a^{-1}]$
is an isomorphism. Together with Lemma \ref{lemma:IsomorphismAlongTheDiagonal} 
this completes the proof of Theorem \ref{thm:CentreU2}. 
As already announced, to prove that $\phi_a$ is an isomorphism we rely on the fact that we can leverage results on the 1-variable case of our problem to get a good understanding of $Z_{t,s}$.
In particular, from the description of $Z_t, Z_s$ given in Lemma \ref{lemma:FreFei1} and Remarks \ref{rmk:LnGenerateInDegreeOne} and \ref{rmk:Zt} we obtain the following characterisation of $Z_{t,s}$.

\begin{lemma}\label{lemma:centrots}
For all positive integers $N$ we have an isomorphism
\begin{equation}\label{eq:centrots}
 \frac{Z_{t,s}}{\hJ_{t,s}(N)\cap Z_{t,s}}\simeq 
 \frac{Z_{t}}{\hJ_{t}(N)\cap Z_{t}}\otimes _Q \frac{Z_{s}}{\hJ_{s}(N)\cap Z_{s}},
\end{equation}
and the right hand side is a polynomial algebra generated by the images of the elements $\Sug_n^{(t)}$, $\Sug_n^{(s)}$ for $n<2N$. Moreover, the image of $Z_{t,s}^{\leq 2}$ in the quotient $\frac{Z_{t,s}}{\hJ_{t,s}(N)\cap Z_{t,s}}$ corresponds, under the previous isomorphism, to the subspace of polynomials of degree at most 1 in the generators $\Sug_n^{(t)}$, $\Sug_n^{(s)}$.
\end{lemma}
\begin{proof}
We know from Section \ref{sec:examples} that $\hU_t\otimes_Q \hU_s$ is a dense subalgebra of $\hU_{t,s}$, hence $Z_t\otimes Z_s$ is a subalgebra of $Z_{t,s}$. Moreover, the isomorphism of Equation \eqref{eq:UtsAndUtUs} shows that the multiplication map from the right hand side to the left hand side of \eqref{eq:centrots} is injective.
Now observe that $\hgog_{t,s}^+ \cong \hgog_t^+ \oplus \hgog_s^+$, and that the isomorphism \eqref{eq:UtsAndUtUs} is compatible with this decomposition. It follows that
\[
\begin{aligned}
\frac{Z_{t,s}}{\hJ_{t,s}(N)\cap Z_{t,s}} \hookrightarrow \left( \frac{\hat{U}_{t,s}}{\hJ_{t,s}(N)} \right)^{\hgog_{t,s}^+} & \simeq 
\left( \frac{\hat{U}_{t}}{\hJ_{t}(N)}\right)^{\hgog_{t}^+} \otimes_Q \left( \frac{\hat{U}_{s}}{\hJ_{s}(N)} \right)^{\hgog_{s}^+} \\
& \simeq \frac{Z_t}{\hJ_t(N)\cap Z_t} \otimes_Q \frac{Z_s}{\hJ_s(N)\cap Z_s},
\end{aligned}
\]
where the last isomorphism is guaranteed by the remarks following Lemma \ref{lemma:FreFei1}. In addition, all these isomorphisms are compatible with the PBW filtration, so the statement follows from Lemma \ref{lemma:FreFei1} and Remark \ref{rmk:LnGenerateInDegreeOne}.
\end{proof}

Recall from Lemma \ref{lem:hE2} that the expansion map $E:\hU_2[a^{-1}]\lra \hU_{t,s}$ is injective, and that the topology on 
$\hU_2[a^{-1}]$ and $\hU_2$ is the subspace topology induced by the topology on $\hU_{t,s}$ under the embedding given by $E$. To simplify the notation, we will identify $\hU_2$ and $\hU_2[a^{-1}]$ with subspaces of $\hU_{t,s}$, and we will omit the map $E$. 

We also notice that the centre of $\hU_2[a^{-1}]$ is equal to $Z_2[a^{-1}]$, and recall from Lemma \ref{lem:hE2} that 
$\hU_2[a^{-1}]$ is dense in $\hU_{t,s}$, so that $Z_2[a^{-1}]\subset Z_{t,s}$. 
We now use Lemma \ref{lemma:centrots} to get information about $Z_2[a^{-1}]$.

We begin by comparing the elements $\Sstorto_n\in Z_2[a^{-1}]$ with the elements $\Sug^{(t)}_n$ and  $\Sug^{(s)}_n$ in $Z_{t,s}$:
\begin{lemma}\label{lemma:TriangularChangeOfBasis}
For every $N \geq 1$ and every $n \in \frac{1}{2}\mathbb{Z}$ with $n<2N$, there exist $q_i^{(s)}, q_i^{(t)}, q_n \in Q$ such that
\[
\Sstorto_n \equiv \Sug^{(s)}_n + q_n+ \sum_{i=n}^{2N-1} q^{(t)}_i \Sug^{(t)}_i + \sum_{i=n+1}^{2N-1} q^{(s)}_i \Sug^{(s)}_i \; \bmod{\hJ_{t,s}(N)},
\]
if $n \in \mathbb{Z}$, and
\[
\Sstorto_n \equiv \Sug^{(t)}_{k} + q_n+ \sum_{i=k+1}^{2N-1} q^{(t)}_i \Sug^{(t)}_i + \sum_{i=k+1}^{2N-1} q^{(s)}_i \Sug^{(s)}_i \; \bmod{\hJ_{t,s}(N)},
\]
if $n=k+\frac{1}{2}, k \in \mathbb{Z}$.
In particular, for every $\ell \leq 2N-1$ the matrix representing the change of basis from $\{1,\Sug_i^{(t)},\Sug_i^{(s)} \bigm\vert i=\ell,\ldots,2N-1 \}$ to $\{1,  \Sstorto_n  \bigm\vert n= \ell, \ldots, 2N-\frac{1}{2} \}$ is triangular with diagonal coefficients equal to 1.
\end{lemma}
As a consequence of this lemma, for every $N \geq 1$ the elements $\Sstorto_n$ for $n<2N$, together with $1$, form a basis of the $Q$-vector subspace of $Z_{t,s}/\hJ_{t,s}(N) \cap Z_{t,s}$ generated by the 
images of $1$ and of the operators $\Sug_i^{(t)}, \Sug_i^{(s)}$ for $i \leq 2N-1$.
\begin{proof}
As already noticed, the element $\Sstorto_n$ is central in $\hat{U}_{t,s}$, 
and it belongs to $\hat{U}_{t,s}^{\leq 2}$. It follows from Lemma \ref{lemma:centrots} that $\Sstorto_n$ is congruent modulo $\hJ_{t,s}(N)$ to a polynomial $P$ in $\Sug_i^{(t)}, \Sug_i^{(s)}$ of degree at most 1. Since the operators $\Sug_i^{(t)}, \Sug_i^{(s)}$ for $i>2N-1$ vanish modulo $\hJ_{t,s}(N)$, we may assume that $P$ is of the form
\begin{equation}\label{eq:PIsASum}
P = q_n + \sum_{i=i_0}^{2N-1} q_i^{(t)} \Sug_i^{(t)} +\sum_{i=i_0}^{2N-1} q_i^{(s)} \Sug_i^{(s)}
\end{equation}
for some $q_n, q_i^{(t)}, q_i^{(s)} \in Q$ and $i_0 \in \mathbb{Z}$.

We now observe that the degrees of the leading terms of the operators $\Sug_i^{(t)}, \Sug_i^{(s)}$ are all distinct by Lemma \ref{lemma:ExpansionOneVariableSugawara}, so they cannot cancel with each other. Thus the leading term of the right hand side of \eqref{eq:PIsASum} is equal to the leading term of one of its summands. On the other hand, $P \equiv \Sstorto_n \bmod \hJ_{t,s}(N)$, so its leading term may be computed modulo $\hJ_{t,s}(N)$ and is equal to either the leading term of $\Sug_n^{(s)}$ (if $n \in \mathbb{Z}$) or to that of $\Sug_{n-\frac{1}{2}}^{(t)}$ (if $n \in \frac{1}{2}\mathbb{Z} \setminus \mathbb{Z}$), see Remark \ref{rmk:LeadingTermLStorti}. It follows that the corresponding operator must appear with coefficient $1$ in the right hand side of \eqref{eq:PIsASum}, and that $\Sug_n^{(s)}$ (respectively $\Sug_{n-\frac{1}{2}}^{(t)}$) is the first operator to appear with nonzero coefficient.
\end{proof}

\begin{lemma}\label{lem:centrocalLdensi}
The closure of the $Q$-subalgebra $Q[\Sstorto_n]$ in $\hat{U}_{t,s}$ is $Z_{t,s}$, hence a fortiori the closure of $Q[\Sstorto_n]$ in $\hat{U}_2[a^{-1}]$ coincides with $Z_2[a^{-1}]$.
\end{lemma}
\begin{proof}

It suffices to show that, for all integers $N \geq 1$, the natural inclusion $Q[\Sstorto_n] \subseteq Z_{t,s}$ induces an isomorphism
\[
\frac{Q[\Sstorto_n]}{\hJ_{t,s}(N) \cap Q[\Sstorto_n]} \cong \frac{Z_{t,s}}{\hJ_{t,s}(N) \cap Z_{t,s}}.
\]
This follows from Lemmas \ref{lemma:centrots} and \ref{lemma:TriangularChangeOfBasis}: the quotient $\frac{Z_{t,s}}{\hJ_{t,s}(N) \cap Z_{t,s}}$ is a polynomial algebra over the generators $\Sug^{(s)}_n, \Sug^{(t)}_n$ for $n<2N$, which belong to $\frac{Q[\Sstorto_n]}{\hJ_{t,s}(N) \cap Q[\Sstorto_n]}$ by the last statement in Lemma  \ref{lemma:TriangularChangeOfBasis}.
\end{proof}

Recall that we denote by $\goz$ the image of the map $\phi$ of Theorem \ref{thm:CentreU2}
and, for $N\in\mZ$, we write $\goz_{< N}$ for the image via $\phi$ of $A[X_i:i<2N]$. 

\begin{lemma}\label{lem:gozainv}
The following hold:
\begin{enumerate}[\indent 1)]
 \item $\phi$ is injective and $\goz$ is closed;
 \item $\goz[a^{-1}]$ is a closed subset of $Z_2[a^{-1}]$.
\end{enumerate}
\end{lemma}
\begin{proof}Let $N\in \mZ$ and let $\pi_N$ be the projection from $\goz[a^{-1}]$ to $Z_{t,s}/\hJ_{t,s}(N)\cap Z_{t,s}$. We prove that the restriction of $\pi_N\circ \phi$ to the algebra $Q[X_i:i<2N]$ is injective. 
This is equivalent to showing that the images of the operators $\Sug^{(2)}_i$ for $i<2N$ under the map $\pi_N$ are algebraically independent. Given that the operators $\Sug^{(2)}_i$ and $\Sstorto_i$ only differ by a linear change of variables, this is also equivalent to proving that the images of the operators $\Sstorto_i$ for $i<2N$ under the map $\pi_N$ are algebraically independent. 
This last fact holds since, by Lemma \ref{lemma:TriangularChangeOfBasis}, 
a nontrivial polynomial relation among the $\Sstorto_{n}$ with $n<2N$ induces a corresponding nontrivial polynomial relation among the $\Sug_i^{(t)}, \Sug_i^{(s)}$,
and, by Lemma \ref{lemma:centrots}, the quotient $Z_{t,s} / \hJ_{t,s}(N) \cap Z_{t,s}$ is a polynomial algebra with generators the images of $\Sug_{n}^{(t)}, \Sug_n^{(s)}$ for $n < 2N$. 

Since 
$\widehat {A [X_i]}=A[X_i:i<2N]\oplus \calI(2N)$, where $\calI(2N)$ is the ideal generated by $X_i$ for $i\geq 2N$,
we also deduce that the inverse image of $\hJ_{t,s}(N)$ through the map $\phi$ is the ideal 
$\calI(2N)$. This implies that $\phi$ is injective and that the topology
on $\widehat {A[X_i]}$ is the same as that induced by $\hU_2$ through the map $\phi$. 
In particular, as $\widehat {A[X_i]}$ is complete, its image $\goz$ is closed. This proves 1). 

In order to deduce from 1) that $\goz[a^{-1}]$ is closed in $Z_2[a^{-1}]$ it is enough to check that
$$
\big(\goz+\hJ_2(N)\cap Z_2\big)\cap a Z_2=a\big(\goz+\hJ_2(N)\cap Z_2\big).
$$
By the previous discussion we have $\goz+\hJ_2(N)\cap Z_2=
\goz_{< N}\oplus \hJ_2(N)\cap Z_2$, hence the claim follows from Lemma \ref{lem:goz1}.
\end{proof} 

The following proposition completes the proof of Theorem \ref{thm:CentreU2}.

\begin{proposition}\label{prp:centrofuoridiagonale}
The map $\phi_a:\widehat {A[X_i]}[a^{-1} ]\lra Z_2[a^{-1}]$ is an isomorphism.
\end{proposition}
\begin{proof}
We have already proved that $\phi$ is injective. As $Z_2$ has no $a$-torsion, it immediately follows that $\phi_a$ is injective. 

We now prove surjectivity. The image of $\phi_a$ is the space $\goz[a^{-1}]$, which is closed in $Z_2[a^{-1}]$ by the previous lemma. It remains to prove that $\goz[a^{-1}]$ is dense in $Z_2[a^{-1}]$; we do this by showing that it is dense in $Z_{t,s}$. Equivalently, we need to prove that for all natural $N$ we have
$\goz[a^{-1}]+\hJ_{t,s}(N)\cap Z_{t,s}=Z_{t,s}$. As in the previous lemma we have
$$\goz[a^{-1}]+\hJ_{t,s}(N)\cap Z_{t,s} = \goz_{< N}[a^{-1}]+\hJ_{t,s}(N)\cap Z_{t,s}=
Q[\Sstorto_i:i<N]+\hJ_{t,s}(N)\cap Z_{t,s}:$$
the desired equality of this space with $Z_{t,s}$ now follows from Lemma 
\ref{lem:centrocalLdensi}.
\end{proof}

 \section{A Feigin-Frenkel theorem with two singularities}\label{sez:FF2}
In this section we give a version in our context of a well-known theorem of Feigin and Frenkel on the center of the enveloping algebra at the critical level \cite{FF92}. We begin by recalling their result.

\begin{theorem}[Feigin and Frenkel, \cite{FF92}]\label{thm:isomorfismo1sing}
There exists a unique continuous isomorphism $$\FF_1 :\Funct(\Op_1^*)\lra Z_1$$ of topological $\mC$-algebras that is compatible with the action of derivations on both sides and satisfies $\FF_1(\zz_n)=2\,\Sug^{(1)}_{-n-1}$.
\end{theorem}

The theorem generalizes to the case of $Z_t$ and $Z_s$ in the obvious way. 
The morphisms $\FF_t $ and $\FF_s$, together with their inverses, induce a continuous isomorphism 
$\FF_t\otimes \FF_s:\Funct(\Op_t^*)\otimes_Q\Funct(\Op^*_s)\lra Z_t\otimes_Q  Z_s$, where $\Funct(\Op_t^*)\otimes_Q\Funct(\Op^*_s)$ is given the tensor product topology (see Section \ref{sez:funzionioper}).
By Lemma \ref{lemma:centrots} $Z_t\otimes_Q Z_s$ is dense in $Z_{t,s}$, and similarly, as noticed in Section \ref{sez:funzionioper}, the space 
$\Funct(\Op^*_t)\otimes_Q\Funct(\Op^*_s)$ is dense in $\Funct(\Op^*_{t,s})$. Moreover, in both cases the subspace topology  coincides with the tensor product topology.
It follows that the isomorphism $\FF_t\otimes\FF_s$ and its inverse extend by continuity to all of $\Funct(\Op^*_t\times_{\Spec Q} \Op^*_s)$ and $Z_{t,s}$ respectively. We denote
by $$
\FF_{t,s}:\Funct(\Op^*_t\times_{\Spec Q} \Op^*_s)\lra Z_{t,s} 
$$
the isomorphism obtained in this way.
We are now in a position to describe the center $Z_2$ of $\hU_2$ in terms of functions on $\Op_2^*$. Recall from section \ref{ssez:campivettorifunzioni}
that $\gra_n$, $\grb_n$ are the coordinates of an oper with respect to the topological basis $u_n$, $y_n$. 

\begin{theorem}\label{thm:isomorfismo}
 There exists a unique isomorphism $\FF_2:\Funct(\Op_2^*)\lra Z_2$ of topological $A$-algebras such that the following two diagrams commute:
\begin{equation}\label{eq:commutazioni}
\begin{gathered}
\xymatrix{
  \Funct(\Op_2^*)\ar_{\calE}[d] \ar^-{\FF_2}[rr] & & Z_2 \ar^{E}[d] & & \Funct(\Op_2^*)\ar_{\calSp }[d] \ar^-{\FF_2}[rr] & &Z_2 \ar^{Sp}[d]\\
  \Funct(\Op_t^*\times_{\Spec Q}\Op_s^*) \ar^-{\FF_{t,s}}[rr] & &  Z_{t,s} & & \Funct(\Op_1^*) \ar^-{\FF_1}[rr] & & Z_1. 
 }
 \end{gathered}
 \end{equation}
Moreover, $\FF_2$ is $\Der_2$-equivariant, and for all $n\in \mZ$ we have 
 \begin{equation}\label{eq:FFLab}
 \FF_2(\grb_{n})=2 \Sug^{(2)}_{-1-n}\quad \text{ and }\quad\FF_2(\gra_{n})=2 \Sug^{(2)}_{-n-1/2}.
\end{equation}
 \end{theorem}

\subsection{Proof of Theorem \ref{thm:isomorfismo}}
Write $\Phi:Z_2\lra \Funct(\Op_t^*\times_{\Spec Q}\Op_s^*)$ for the composition $\FF^{-1}_{t,s}\circ E$ and $\Psi:\Funct({\Op_2^*})\lra Z_{t,s}$ for the composition $\FF_{t,s}\circ\calE$. 
We begin by showing that the existence and continuity of a morphism of $A$-algebras $\FF_2$ 
with the required properties follows from 
\begin{equation}\label{eq:phiE}
\Phi(\Sug^{(2)}_{n})=\frac 12 \calE(\grb_{-1-n})\quad \text{ and }\quad\Phi(\Sug^{(2)}_{n-1/2})=\frac 12 \calE(\gra_{-n}).
\end{equation}
Indeed, if these formulas hold, then $\Phi(A[\Sug^{(2)}_k:k\in \tfrac 12\mZ])$ is contained in the image of $\calE$. Since the topology of $\Funct(\Op^*_2)$ is induced by the topology of $\Funct(\Op_t^*\times_{\Spec Q}\Op_s^*)$ through the map $\calE$, and $\Funct(\Op^*_2)$ is complete, the image of $\calE$ is closed. It then follows from the continuity of $\Phi$ that its image is contained in that of $\calE$. In particular, as $\calE$ is injective (see Section \ref{sez:funzionioper}), there exists a continuous map $H:Z_2\lra\Funct(\Op_2^*)$ such that $\calE\circ H=\Phi$.  Similar considerations apply to the map $\Psi$. Indeed,
if equations \eqref{eq:phiE} hold, then we also have $\Psi(\grb_{-1-n}) =2 \, E(\Sug^{(2)}_{n})$
and 
$\Psi(\gra_{-n})= 2\,E(\Sug^{(2)}_{n-1/2})$. Arguing as above, we can construct a continuous map $\FF_2$ of $A$-algebras such that $E\circ \FF_2=\Psi$, and by continuity we obtain that $\FF_2$ and $H$ are inverse to each other. 

The commutativity of the second diagram in \eqref{eq:commutazioni} 
follows from equations \eqref{eq:FFLab}, from 
$\Sp(\Sug^{(2)}_k)=\Sug_{2k}^{(1)}$ for all $k\in \tfrac 12 \mZ$, from 
$\calSp(\gra_n)= \zz_{2n+1}$ and $\calSp(\grb_n)=\zz_{2n}$, and from 
$\FF_1(\zz_n)=2 \Sug^{(2)}_{-1-n}$. Indeed, from these formulas we obtain that
the maps $\Sp\circ\FF_2$ and $\FF_1\circ \calSp$ coincide on the elements $\gra_n$, $\grb_n$, and by continuity we get
$\Sp\circ\FF_2=\FF_1\circ \calSp$.
Finally, the uniqueness of $\FF_2$ and its $\Der_2$-equivariance follow
from the injectivity of $E$, from the  commutativity of the first diagram in \eqref{eq:commutazioni}, and 
from the $\Der_2$-equivariance of $\Psi$ and $E$.

\medskip

To prove that the equalities in \eqref{eq:phiE} hold we use the $\Der_2$-equivariance of $\calE$ and $\Phi$. For the sake of simplicity, we consider the map $\calE$ as an inclusion and omit it from the notation. For every integer $n$ we set
$$
\tilde \gra_n= 2 \Phi(\Sug^{(2)}_{-n-1/2})\quad \text{ and }\quad \tilde \grb_n= 2 \Phi(\Sug^{(2)}_{-n-1}).
$$
Since $\Phi$ is $\Der_2$-equivariant, Lemma \ref{lemma:DerivativesOfL} gives significant information on the action of $\Der_2$ on the functions $\tilde{\alpha}_n, \tilde{\beta}_n$. Our strategy is now to show that this action is sufficiently similar to the action of $\Der_2$ on $\alpha_n, \beta_n$ (described by Lemma \ref{lemma:DerivativesOfAlphaBeta}) to force $\alpha_n=\tilde{\alpha}_n$ and $\beta_n=\tilde{\beta}_n$.

By Lemma \ref{lemma:DerivativesOfL}, we have in particular
	\begin{equation}\label{eq:equivfg}
	s\partial \tilde \gra_{-1}= 2\Phi (s\partial \Sug^{(2)}_{1/2} ) =  0, 
	\end{equation}
	Moreover, $E(\Sug_{1/2}^{(2)})$ is in $Z_{t,s}^{\leq 2}$, so by Lemma \ref{lemma:centrots} it must be an (in general infinite) linear combination of the operators $\Sug_{n}^{(t)}$ and $\Sug_{n}^{(s)}$, possibly with a constant term.
	Since the maps $\FF^{-1}_t$ and $\FF^{-1}_s$ and the relations between the coordinates $\zz^{(t)}_i,\zz^{(s)}_i$ and $\gra_i,\grb_i$ are linear, we get that  $\tilde{\gra}_{-1}=\mathcal{F}_{t,s}^{-1}(E(\Sug_{1/2}^{(2)}))$ may be expressed as
	\begin{equation}\label{eq:bpqr}
	\tilde \gra_{-1}=\sum_{i \leq M}p_i \gra_i+\sum_{i\leq N} q_i\grb_i+r 
	\end{equation}
	for some $M,N \in \mathbb{Z}$ and $p_i,q_i,r\in Q$. 
	We start by showing that conditions \eqref{eq:equivfg} and \eqref{eq:bpqr} 
almost completely determine the element $\tilde{\gra}_{-1}$.

\begin{lemma}\label{lemma:f}
There exists a unique element of the form 
$$
f=\sum_{n\leq -2}p_n\gra_n+\sum_{n\leq -2}q_n\grb_n
$$ 
with $p_n,q_n\in Q$, $p_{-2}=1$ and $s\partial f=0$. Moreover, the solutions of the equation
$s\partial g=0$ with $g\in \overline {\langle 1,\gra_i,\grb_i\rangle_Q }\subset \Funct(\Op^*_t\times\Op^*_s)$ form a $Q$-vector space with basis $1, \alpha_{-1}, f$.
\end{lemma}
\begin{proof}
Since $s\partial$ is $Q$-linear and satisfies $s\partial \gra_{-1}=s\partial 1=0$, it suffices to show that the solutions $g=\sum_{i\leq M}p_n\gra_n+\sum_{n\leq N}q_n\grb_n$ to the equations
\begin{equation}\label{eq:spartialgVanishes}
\begin{cases}
s\partial g = 0 \\
p_{-1}=0
\end{cases}
\end{equation}
form a 1-dimensional vector space (over $Q$). To avoid ambiguity in the series representation of $g$, we may assume that $M$ is either $- \infty$ (if no term $\alpha_i$ has a nonzero coefficient) or is such that $p_M \neq 0$ (respectively, $N$ is either $-\infty$ or such that $q_N \neq 0$). Notice in particular that $M \neq -1$ since $p_{-1}=0$ by assumption.
By Lemma \ref{lemma:DerivativesOfAlphaBeta}, the coefficients of $\grb_{n}$ and $\gra_{n+1}$ in $s\partial g$ are given respectively by
\begin{equation}\label{eq:CoefficientsOfspartialg}
(n+1)ap_n -(2n+3)q_n \quad \text{ and}\quad -(n+1)a^2p_n+(n+1)aq_n-2(n+2)p_{n+1}.
\end{equation}
Let now $g$ be a nonzero solution of \eqref{eq:spartialgVanishes}, so that these two quantities vanish for every $n$.
We now prove that we have $M=N \leq -2$.
Suppose by contradiction that $N>M$ holds: then the coefficient of $\grb_{N}$ in $s\partial g$ is $-(2N+3)q_N$, which cannot vanish since $q_N \neq 0$.
Conversely, suppose that $M>N$. By the formulas above, the coefficient of $\gra_{M+1}$ in $s\partial g$ is then $-(M+1)a^2p_{M}$, which again cannot vanish (recall that $M \neq -1$ by our convention), contradiction. Hence we must have $N=M \neq -1$. Finally, suppose that $N=M$ is strictly greater than $-1$: then the vanishing of the expressions \eqref{eq:CoefficientsOfspartialg} for $n=M$ yields
$p_M=q_M=0$, again a contradiction.

Hence as claimed we have $M, N\leq -2$, and from \eqref{eq:CoefficientsOfspartialg} we get
\begin{equation}\label{eq:Recursionpnqn}
q_n=\frac{n+1}{2n+3}ap_n\qquad p_n=-2\frac{2n+3}{(n+1)a^2}p_{n+1}
\end{equation}
for all integers $n$. In particular, the value of $p_{-2}$ determines $p_n$ and $q_n$ for all $n \leq -2$, which shows that the space of solutions to \eqref{eq:spartialgVanishes} has dimension at most 1. On the other hand, setting $p_{-2}=1$ and $p_n=q_n=0$ for $n>-2$, and solving \eqref{eq:Recursionpnqn} recursively, we get the coefficients of a series $f$ that lies in $\overline{ \langle \alpha_i, \beta_i \rangle_Q}$ and solves \eqref{eq:spartialgVanishes}, hence the space of solutions to this system has dimension exactly 1. Combined with our previous remarks, this concludes the proof of the lemma.
\end{proof}

As a consequence of the above lemma, we get that $\tilde{\gra}_{-1} = p \gra_{-1} + qf + r$ for some $p,q,r \in Q$. Our next goal is to show that the relations of Lemma \ref{lemma:DerivativesOfL} force $q=r=0$.
We will need to know explicitly the first few terms in the series expansion of the element $f$. Using the recursion obtained in the proof of Lemma \ref{lemma:f} we get
\begin{equation}\label{eq:primiterminif}
f=a\grb_{-2}+\gra_{-2}-\frac{2}{a}\grb_{-3}
-\frac{3}{a^2}\gra_{-3} +\frac{6}{a^3}\grb_{-4} +\frac{10}{a^4}\gra_{-4}
-\frac{20}{a^5}\grb_{-5} -\frac{35}{a^6}\gra_{-5}+
\cdots
\end{equation}

\begin{lemma}\label{lemma:bbba}
Writing $\tilde \gra_{-1}=p\gra_{-1}+qf+r$ as above, we have 
$\tilde \grb_{i}=p\grb_{i}$ for $i=-2,-1,0$ and
$\tilde \gra_{0}=p\gra_0$.
\end{lemma}
\begin{proof}
We compute $\tilde \grb_{-1}, \tilde \grb_{-2}$, $\tilde \grb_{0}$ and $\tilde \gra_0$ in terms of $\tilde{\alpha}_{-1}$ using the 
following relations, which follow from Lemma \eqref{lemma:DerivativesOfL} arguing as for formula
\eqref{eq:equivfg}:
\begin{equation}\label{eq:bbba}
\begin{aligned}
\tilde \grb_{-1}&=\partial \tilde \gra_{-1},
&
\partial (u_{2}\partial) \tilde \grb_{-1}&= -12\tilde \grb_{-2}-2 a^2\tilde \grb_{-1},
\\
(u_{-1}\partial) \tilde \grb_{-2}  &= 2(2\tilde \gra_0+a\tilde \grb_0),
& 
u_1\partial (2\tilde \gra_0+a\tilde \grb_0)&= -6\tilde \grb_{-1}-2a^2\tilde \grb_{0}.
\end{aligned}
\end{equation}
Using Lemma \ref{lemma:DerivativesOfAlphaBeta} and the explicit expression of $f$ given by \eqref{eq:primiterminif} we deduce first that $\tilde \grb_{-1}$ is a series in $\grb_{-1},\gra_{-1},\grb_{-2},\dots$ without constant term, and then that $\tilde \grb_{-2}$ is a series in $\grb_{-2},\gra_{-2},\grb_{-3},\dots$ whose first terms are as follows:
$$\tilde \grb_{-2}=
(p+qa^2)\grb_{-2}+qa\gra_{-2}-q\grb_{-3}-\frac {2q}{a}  \gra_{-3}+\frac{3q}{a^2}\grb_{-4}+\frac{6q}{a^3}\gra_{-4}+\cdots
$$
From this we easily obtain that 
$Y := (u_{-1}\partial) \tilde \grb_{-2}$ is a series without constant term whose first terms are
\[
Y=(2ap+2a^3q)\grb_0+(4p+2a^2q)\gra_0+\cdots
\]
Finally, from this we deduce that $Z:=\tfrac 12(u_1\partial) Y$ is a series in $\grb_0,\gra_0,\grb_{-1},\gra_{-1},\grb_{-2},\dots$ without constant term and with coefficient of 
$\gra_0$ equal to $-2a^3q$. In particular, since by \eqref{eq:bbba} we have
\[
\tilde{\beta}_0 = \frac{1}{2a^2}\left(-6\tilde{\beta}_1 - Z \right) \quad\text{ and }\quad \tilde{\alpha}_0 = \frac{1}{4}Y - \frac{1}{2}a \tilde{\beta}_0,
\]
it follows that $\tilde \gra_0$ and $\tilde \grb_0$ have no constant term and that the coefficient of $\gra_0$ in $\tilde \grb_0$ is equal to $aq$.
Finally, from Lemma \ref{lemma:DerivativesOfL} we also have
\begin{equation}\label{eq:qIsZero}
v_1\partial \tilde \beta_0=2a\tilde\gra_0-5\tilde\grb_{-1}.
\end{equation}
From the explicit expressions of Lemma \ref{lemma:DerivativesOfAlphaBeta} we see that $v_1\partial(\beta_i)$ never involves any constant terms, while $v_1\partial(\alpha_i)$ has a nonzero constant term (equal to $3$) only for $i=0$. Since the right hand side of \eqref{eq:qIsZero} has no constant term, this equation implies $3aq=0$, that is, $q=0$.
From Lemma \ref{lemma:DerivativesOfAlphaBeta} and equations \eqref{eq:bbba} it is now easy to obtain $\tilde \grb_i=p\grb_i$
for $i=-2,-1,0$ and $\tilde \gra_0=p\gra_0$ as claimed. 
\end{proof}

\begin{lemma}\label{lemma:tildeugualenontilde}
We have $\tilde \grb_i=\grb_i$ and $\tilde \gra_i=\gra_i$ for all integers $i$.
\end{lemma}
\begin{proof}
Let $p$ be as in the statement of Lemma \ref{lemma:bbba}.
We begin by proving that the equalities $\tilde{\beta}_i=p\beta_i$ and $\tilde{\alpha}_i=p\alpha_i$ hold for all integers $i$.
From Lemma \ref{lemma:DerivativesOfL} we have
\begin{equation}\label{eq:induzionebeta}
\partial (u_{m+1}\partial) \tilde \grb_{-1}= -(m+1)\; \big( 2(2m+1)\tilde \grb_{-m-1}+a^2m \tilde \grb_{-m}\big)
\end{equation}
for all integers $m$, and analogous equations hold for the functions $\grb_n$ by Lemma \ref{lemma:DerivativesOfAlphaBeta}.
Using these relations for $m>0$ and 
arguing by decreasing induction, starting with $n=-1$ (notice that $\tilde{\beta}_{-1}=p\beta_{-1}$ holds by the previous lemma), we get 
$\tilde \grb_n =p\grb_n$ for all $n<0$. 
Consider now the following relation, which follows again from Lemma \ref{lemma:DerivativesOfL}:
\begin{align*}
\partial (2\tilde \gra_0+a\tilde \grb_0)&=-2\tilde \grb_0 -a^2\tilde \grb_1\ .
\end{align*}
Lemma \ref{lemma:DerivativesOfAlphaBeta} implies that the analogous relation holds for the functions $\gra_0, \grb_0, \grb_1$. Using Lemma \ref{lemma:bbba}, from the comparison of these formulas we obtain 
 $\tilde \grb_1=p\grb_1$.
Using equation \eqref{eq:induzionebeta} again, and arguing by induction (starting with $n=1$), 
we now get $\tilde \grb_n=p\grb_n$ for all $n$. 
Finally, from $\partial \tilde \grb_i=-(i+1)(2\tilde \gra_{i+1}+a\tilde \grb_{i+1})$ and the analogous relations for the functions $\gra_i$, $\grb_i$ we obtain $\tilde \gra_i=p\gra_i$ for all $i\neq 0$. The only remaining equality $\tilde \gra_0=p\gra_0$ is part of the previous lemma.
It remains only to show that $p=1$.
In order to do this, we apply the operator $u_2\partial$ to the functions $\tilde \grb_{0}$ and
$\grb_{0}$. By Lemmas \ref{lemma:DerivativesOfL} and \ref{lemma:DerivativesOfAlphaBeta}
we have
$$
u_2\partial \grb_0=-6\gra_{-1}-3a\grb_{-1}+12
\quad\text{ and }\quad
u_2\partial \tilde \grb_0=-6\tilde \gra_{-1}-3a\tilde \grb_{-1}+12.
$$
As we have $\tilde{\beta}_0=p\beta_0$, $\tilde{\alpha}_{-1}=p\alpha_{-1}$, and $\tilde{\beta}_{-1}=p\beta_{-1}$, these equations immediately imply $p=1$ as desired.
\end{proof}

This completes the proof of the relations \eqref{eq:phiE}, hence of Theorem \ref{thm:isomorfismo}.

\section{Weyl modules and their endomorphism rings}\label{sez:Weylmodules}
We start by recalling the definition of the Weyl modules for the affine Lie algebra 
$\hgog_1$. Let $\lambda$ be a dominant integral weight and let $V^\lambda$ be the 
corresponding irreducible representation of $\gog$. The \textbf{Weyl module} of 
weight $\lambda$ is
\[ \V_1^\lambda = \operatorname{Ind}_{\hgog_1^+}^{\hgog_1} V^\lambda, \]
where $\hgog_1^+$ acts on $V^\lambda$ as follows:
\begin{equation}\label{eq:azionegp}
  f(t)x \cdot u = f(0) \, xu, \quad C_1 \cdot u = -\frac{1}{2} u. 
 \end{equation}
Note that $\mV_1^\lambda$ has a natural structure of $\hU_1$-module. Frenkel 
and Gaitsgory described the endomorphism ring of this module in \cite{FG6}.
Let now $V_Q^{\grl}=Q\otimes _\mC V^\grl$. We can define a $Q$-linear 
action of $\hgog^+_t$ on $V_Q^{\grl}$ as in Equation \eqref{eq:azionegp}, and the induction of this representation from $\hgog^+_t$ to $\hgog_t$ yields a $\hgog_t$-module that we denote by $\mV_t^\grl$. We similarly construct the $\hgog_s$-module
$\mV_s^\grl$. We remark that these representations are obtained from 
$\mV_1^\grl$ by tensoring with $Q$, and that the action of $U_t$ (respectively $U_s$) induces an 
action of $\hU_t$ (respectively $\hU_s$). 

\begin{theorem}
[{Frenkel-Gaitsgory \cite[Theorem 1]{FG6}}]\label{thm:EndoVlambda}
There is a commutative diagram
\[ \xymatrix{ \Funct(\Op_1^*) \ar[d]_{\rho}
		 \ar[r]^-{\simeq}_-{\FF_1}  & Z_1 \ar[d]^{\omega_1} \\ \Funct(\Op_1^\lambda)\ar[r]^-{\simeq}_-{\calG_1}&    \End_{\hgog_1}(\V_1^\lambda),} \]
where the isomorphism $\FF_1 : \Funct(\Op_1^*) \to Z_1$ is that of Theorem 
\ref{thm:isomorfismo1sing}, $\rho$ is the natural restriction of functions from
$\Op_1^*$ to its subscheme $\Op_1^\lambda$, and $\omega_1$ is the natural map 
induced by the action of $\hU_1$ on $\V_1^\lambda$.
An analogous statement holds for $\mV^\grl_t$ and for $\mV^\grl_s$.
\end{theorem}

We now construct the Weyl modules for the algebra $\hgog_2$.

\begin{definition}\label{def:weyl2}
Let $\lambda, \mu$ be dominant integral weights and write $V_A^\lambda=V^\lambda 
\otimes_\mC A$, $V_A^\mu=V^\mu \otimes_\mC A$. The \textbf{Weyl module} of 
weights $\lambda,\mu$ is
\[
\V_2^{\lambda, \mu} = \operatorname{Ind}_{\hgog_2^+}^{\hgog_2} \left( 
V_A^\lambda \otimes_A V_A^\mu \right),
\]
where $\hgog_2^+$ acts on $V_A^\lambda \otimes_A V_A^\mu$ as follows:
\[
f(t,s)x \cdot (u \otimes v) = E_t(f)|_{t=0} \, xu \otimes v + E_s(f)|_{s=0} \, u 
\otimes xv, \quad C_2 \cdot (u \otimes v) = -\frac{1}{2} \, u \otimes v.
\]
Notice that $I_2(1) \otimes \gog$ acts trivially on $V_A^\lambda \otimes_A 
V_A^\mu$. This implies that for every $u \otimes v \in \V_2^{\lambda, \mu}$ 
there exists $m \in \mathbb{N}$ such that $I_2(m)\otimes \gog$ acts trivially on $u \otimes 
v$, hence $\V_2^{\lambda,\mu}$ has a natural structure of $\bU_2$-module. As 
$C_2$ acts as $-\frac{1}{2}$, this also induces a structure of $\hU_2$-module.
\end{definition}

To state the next lemma, we notice that both $C_t$ and $C_s$ act as $-1/2$ on $\mV^\grl_t\otimes_Q \mV^\mu_s$. Moreover, for each $w\in 
\mV^\grl_t\otimes _Q \mV^\mu_s$ there exists $n \in \mathbb{N}$ such that $J_t(n)\otimes U_s$ and $U_t\otimes J_s(n)$ act
trivially on $w$. 
In particular this shows that 
the action of $U_t\otimes U_s\subset U_{t,s}$ is continuous, hence it determines an action of  
$\hU_{t,s}$ (and of its subalgebra $\hU_2[a^{-1}]$) on $\mV^\grl_t\otimes _Q \mV^\mu_s$.
Similarly, the action of $\hU_2$ on 
$\mV_2^{\grl,\mu}$ determines by continuity an action of $\hU_{t,s}$ on 
$\mV_2^{\grl,\mu}[a^{-1}]$. In particular, we can define an action of 
$\hU_t\otimes \hU_s$ on $\mV^{\grl,\mu}_2[a^{-1}]$
through the natural map $\hU_t\otimes \hU_s\lra \hU_{t,s}$.

\begin{lemma}\label{lemma:VlambdamuOutsideAndAlongTheDiagonal}
For all integral dominant weights $\lambda,\mu$ we have the following 
isomorphisms of $\hU_2$ modules:
\begin{enumerate}
\item $\V_2^{\lambda, \mu}/a \V_2^{\lambda, \mu} \cong 
\operatorname{Ind}_{\hgog_1^+}^{\hgog_1} (V^\lambda \otimes_\mC V^\mu)$;
\item $\V_2^{\lambda, \mu}[a^{-1}] \cong \V_t^\lambda \otimes_Q \V_s^\mu$.
\end{enumerate}
\end{lemma}
\begin{proof}
For 1), one has $V_A^\lambda / 
aV_A^\lambda \cong V_A^\lambda \otimes_A \mC \cong V^\lambda$, and similarly 
for $V_A^\mu$. Moreover, we have isomorphisms $\hgog_2 / a \hgog_2 \cong \hgog_1$ and $\hgog_2^+ / a 
\hgog_2^+ \cong \hgog_1^+$. It follows that
\[
\begin{aligned}
\V_2^{\lambda, \mu}/a \V_2^{\lambda, \mu} & \cong \V_2^{\lambda, \mu} \otimes_A 
\mC \cong \left( U_2 \otimes_{U_A(\hgog_2^+)} (V_A^\lambda \otimes_A V_A^\mu)  
\right) \otimes_A \mC \\ 
& \cong (U_2 \otimes_A \mC) \otimes_{U_A(\hgog_2^+) \otimes \mC} \left( 
(V_A^\lambda \otimes_A V_A^\mu) \otimes_A \mC \right) \\
& \cong U_1  \otimes_{U_\mC(\hgog_1^+) } (V^\lambda \otimes_\mC V^\mu) = 
\operatorname{Ind}_{\hgog_1^+}^{\hgog_1} (V^\lambda \otimes_{\mC} V^\mu).
\end{aligned}
\]

To prove 2) notice that, by the discussion before the lemma, $\hgog_2[a^{-1}]$ 
acts on 
$\mV^\grl_t\otimes _Q \mV^\mu_s$ and $\hgog_t\times \hgog_s$ acts
on $\mV_2^{\grl,\mu}[a^{-1}]$. Hence, by the definition of Weyl modules as 
induced modules, to define a morphism between these modules $\V_2^{\lambda, \mu}[a^{-1}]$ and $\V_t^\lambda \otimes_Q \V_s^\mu$, and conversely, it is enough to 
define a $\hgog_2^+[a^{-1}]$-equivariant morphism from 
$V_A^\grl\otimes_A V^\mu_A[a^{-1}]$ to $\mV_t^\grl\otimes_Q \mV_s^\mu$
and a $(\hgog_{t}^+\times \hgog_{s}^+)$-equivariant morphisms from
$V_Q^\grl\otimes_Q V^\mu_Q$ to $\mV_2^{\grl,\mu}[a^{-1}]$. Since 
$V_A^\grl\otimes_A V^\mu_A[a^{-1}]\simeq V_Q^\grl\otimes_Q V^\mu_Q$ these morphisms 
are easily constructed, and their composition is the identity.
\end{proof}

Generalising Theorem \ref{thm:EndoVlambda} we now prove:
\begin{theorem}\label{teo:endomorfismi}
There is a commutative diagram
\[
\xymatrix{
\Funct(\Op_2^*) \ar[d]_{\rho} \ar[r]^-{\simeq}_-{\FF_2}  & Z_2\ar[d]^{\omega_2} \\
\Funct(\Op_2^{\lambda, \mu}) \ar[r]_-{\calG_2}^-{\simeq}  & \End_{\hgog_2}(\V_2^{\lambda,\mu})
}
\]
where the isomorphism $\mathcal{F}_2 : \Funct(\Op_2^*) \to Z_2$ is that of 
Theorem \ref{thm:isomorfismo}, $\rho$ is the natural restriction of functions 
from $\Op_2^*$ to its subscheme $\Op_2^{\lambda, \mu}$, and $\omega_2$ is the 
natural map induced by the action of $\hU_2$ on $\V_2^{\lambda, \mu}$.
\end{theorem}

\begin{proof} 
We begin by showing that the composition 
$\Psi=\omega_2\circ\FF_2:\Funct(\Op_2^*)\lra \End_{\hgog_2}(\V_2^{\lambda,\mu})$ 
factors through the quotient 
$\Funct(\Op_2^{\lambda, \mu})$ of $\Funct(\Op_2^*)$. Since every module 
we consider has no $a$-torsion, it suffices to prove this after inverting $a$. 
By Theorem \ref{thm:isomorfismo} and Lemma \ref{lemma:VlambdamuOutsideAndAlongTheDiagonal} 
we have the following commutative diagram:
\[
\xymatrix{
\Funct(\Op_t^*) \otimes_Q \Funct(\Op_s^*) \ar@{^{(}->}[r] \ar[d]_{\simeq}^-{\FF_{t}\otimes \FF_{s}}& 
\Funct(\Op_t^*\times \Op_s^*) \ar[d]_{\simeq}^-{\FF_{t,s}}  & 
\Funct(\Op_2^*)[\frac{1}{a}] \ar[d]_{\simeq}^-{\FF_{2,a}}\ar@{_{(}->}[l]_-{\calE}\\
Z_t \otimes_Q Z_s \ar[d]^{\omega_t \otimes \omega_s} \ar@{^{(}->}[r] & Z_{t,s} & 
Z_2[\frac{1}{a}]  \ar@{_{(}->}[l]_{E} \ar[d]^{\omega_{2,a}} \\
\End_{\hgog_t}(\V^\lambda_t) \otimes_Q \End_{\hgog_s}(\V^\mu_s) \ar[rr]^-{\simeq} & & 
\End_{\hgog_2}(\V_2^{\lambda,\mu})[\frac{1}{a}].
}
\]
To justify the existence of the isomorphism in the bottom row, we notice that the image of $\hgog_2$ in 
$\hgog_{t,s}$ is dense. Thus, if a continuous endomorphism commutes with $\hgog_2$, then it also commutes with $\hgog_{t,s}$, and hence with $\hgog_t\times \hgog_s$. Moreover, we may identify 
$\End_{\hgog_2^+}(\V_2^{\lambda,\mu})[a^{-1}]$ with 
$\End_{\hgog_2^+}(\V_2^{\lambda,\mu}[a^{-1}])$ 
and 
$\End_{\hgog_t}(\V^\lambda_t) \otimes_Q \End_{\hgog_s}(\V^\mu_s)$
with
$\End_{\hgog_t\times \hgog_s}(\V^\lambda_t \otimes_Q \V^\mu_s)$,
because the modules $\V_2^{\lambda,\mu}, \V^\lambda_t$, and $\V^\lambda_s$ are all finitely generated. The existence of the desired isomorphism then follows from the second part of Lemma \ref{lemma:VlambdamuOutsideAndAlongTheDiagonal}.

Now notice that, by Theorem \ref{thm:EndoVlambda}, the composition of the morphisms appearing in the first column of the diagram factors through $\Funct(\Op^\lambda_t)\otimes_Q\Funct(\Op^\mu_s)$. In particular, it factors through 
$\Funct\big((\Op_t^*)_{\geq -2}\big)\otimes_Q \Funct\big((\Op_s^*)_{\geq -2}\big)$. 
By Lemma \ref{prop:FactorisationProperties} 2), this implies that the composition of the morphisms in the rightmost column factors through $\Funct\big((\Op_2^*)_{\geq -2}|_{a\neq 0}\big)$. 
We claim that the diagram above then induces the following commutative diagram, where the morphisms $\mathcal{G}_t, \mathcal{G}_s$ are provided by Theorem \ref{thm:EndoVlambda}:
\[
\xymatrix{
\Funct \big((\Op_t^*)_{\geq -2})\big) \otimes_Q \Funct \big((\Op_s^*)_{\geq -2})\big)
 \ar^-{\simeq}_-{E^\sharp}[r] \ar[d]^{\rho}  & 
\Funct\big((\Op_2^*)_{\geq -2}|_{a\neq 0}\big) \ar[d]^{\rho} 
\\
\Funct(\Op_t^\grl) \otimes_Q \Funct(\Op_s^\mu) \ar^-{\simeq}_-{E^\sharp}[r] 
\ar[d]_-{\simeq}^-{\calG_t\otimes \calG_s}& 
\Funct(\Op_2^{\grl,\mu}|_{a\neq 0})  \ar[d]^-{\calG_{2,a}}
\\
\End_{\hgog_t}(\V^\lambda_t) \otimes_Q \End_{\hgog_s}(\V^\mu_s) \ar[r]^-{\simeq} & 
\End_{\hgog_2}(\V_2^{\lambda,\mu}[a^{-1}]).
}
\]
Indeed, we have just discussed the existence of the morphisms in the first column; the existence of those in the second column follows 
from this and the definition of $\Op_2^{\grl,\mu}$. 
By the remarks at the beginning of the proof, the previous diagram implies that there exists a map 
$\calG_2:\Funct(\Op_2^{\grl,\mu})\lra \End_{\hgog_2}(\mV_2^{\grl,\mu})$ such that 
$\Psi=\calG_2\circ \rho$ and whose localisation $\calG_{2,a}$ is an isomorphism. By Lemma \ref{lem:isoMN}, to finish the proof of the theorem it is enough to show that 
the specialization $\overline{\calG_2}$ is injective. We now consider the composition 
\[
\phi:\Funct\left(\Op^{\grl,\mu}_2|_{a=0}\right) \stackrel{\overline{\calG_2}}{\lra} 
\frac{\End_{\hgog_2}(\mV_2^{\grl,\mu})}{a\End_{\hgog_2}(\mV_2^{\grl,\mu})} \lra
\End_{\hgog_2}\bigg(\frac{\mV_2^{\grl,\mu}}{a\mV_2^{\grl,\mu}}\bigg)
\]
and prove that $\phi$ is injective, which implies the desired injectivity of $\overline{\mathcal{G}_2}$. 
The 
action of $\hgog_2$ on the quotient $\overline{\mV}=\mV^{\grl,\mu}_2/a\mV^{\grl,\mu}_2$
factors through the specialisation map $\Sp : \hgog_2 \to \hgog_1$, and $\End_{\hgog_2}(\overline{\mV})=\End_{\hgog_1}(\overline{\mV})$. By the commutativity of the second diagram in Theorem \ref{thm:isomorfismo}, this is compatible with the specialisation map among coordinate rings induced by the 
isomorphism $\Op^*_2|_{a=0}\simeq \Op^*_1$. 
This shows in particular that $\omega_1\circ \FF_1$ factors through $\phi$, where we consider $\Op^{\grl,\mu}_2|_{a=0}$ as a subscheme of $\Op^*_1$,
as in the following commutative diagram: 
\[
\xymatrix{
 &  & \Funct\left( \Op_2^{\lambda,\mu} \right) \ar[d]^{\calG_2} \\
\Funct(\Op_2^*) \ar@{->>}@/^1pc/[rru]^{\rho} \ar[r]^-{\calF_2}_-{\simeq} \ar[d] & Z_2 \ar[d] \ar[r]^-{\omega_2} & \End_{\hgog_2}\left(\mV_2^{\lambda, \mu} \right) \ar[d] \\
\Funct(\Op_1^*) \ar@{->>}@/_1pc/[rrd]_{\rho \circ \calSp} \ar[r]_-{\calF_1}^-{\simeq} & Z_1 \ar[r]_-{\omega_1} & \End_{\hgog_1}\left(\overline{\mV} \right) \\
&& \Funct\left( \Op_2^{\lambda,\mu}|_{a=0} \right) \ar[u]_{\phi}
}
\]
By Theorem \ref{thm:restrizionediagonale} and Lemma 
\ref{lemma:VlambdamuOutsideAndAlongTheDiagonal}
we are thus reduced to proving 
that 
\[
\phi: \Funct \left( \coprod_{\substack{|\mu-\lambda| \leq \nu \leq \lambda+\mu \\ \nu 
\equiv \lambda + \mu \bmod{2}}} \Op_1^\nu \right) \to 
\End_{\hgog_1}\left(\operatorname{Ind}_{\hgog_1^+}^{\hgog_1} (V^\lambda_1 \otimes_\mC 
V_1^\mu)\right)=
\End_{\hgog_1}\left(
\bigoplus
_{\substack{|\mu-\lambda| \leq \nu \leq \lambda+\mu \\ \nu \equiv \lambda + \mu \bmod{2}}} 
\mV^\nu_1 \right)
\]
is injective, where the last equality follows from the well-known decomposition of the tensor product of two $\mathfrak{sl}_2$-modules. It is enough to prove that the composition of this morphism with the projection $\End_{\hgog_1}\left(\bigoplus
_{\nu} 
\mV^\nu_1 \right) \to \bigoplus_\nu \End_{\hgog_1}\left( \mV^\nu_1 \right)$
is injective. Hence, it is enough to show that the resulting map 
$$
\psi:\Funct \left( \coprod_{\substack{|\mu-\lambda| \leq \nu \leq \lambda+\mu \\ \nu 
\equiv \lambda + \mu \bmod{2}}} \Op_1^\nu \right)\lra 
\bigoplus_{\substack{|\mu-\lambda| \leq \nu \leq \lambda+\mu \\ \nu 
\equiv \lambda + \mu \bmod{2}}} \End_{\hgog_1}\left( \mV^\nu_1 \right)
$$
is injective. For each weight $\xi$ appearing in the direct sum on the right, let $\pi^{\xi}$ be the projection $\bigoplus_{\substack{|\mu-\lambda| \leq \nu \leq \lambda+\mu \\ \nu 
\equiv \lambda + \mu \bmod{2}}} \End_{\hgog_1}\left( \mV^\nu_1 \right) \to \End_{\hgog_1}\left( \mV^\xi_1 \right)$, and let $\psi^\xi := \pi^{\xi} \circ \psi$.
By construction and Theorem \ref{thm:EndoVlambda}, for each $\xi$ we have that $\psi^\xi$
agrees with the projection $\Funct \left( \coprod_{\substack{|\mu-\lambda| \leq \nu \leq \lambda+\mu \\ \nu 
\equiv \lambda + \mu \bmod{2}}} \Op_1^\nu \right) \to \Funct \left( \Op_1^\xi \right) $ followed by the isomorphism $\calG_1:\Funct \left( \Op_1^\xi \right)\xrightarrow\sim \End_{\hgog_1}\left( \mV^\xi_1 \right)$.
The claim follows. 
\end{proof}

\bibliographystyle{acm}
\bibliography{biblio}

\begin{thebibliography}{10}

\bibitem{BD:quantization}
{\sc Beilinson, A., and Drinfeld, V.}
\newblock {Quantization of Hitchin's integrable system and Hecke eigensheaves},
  1991.

\bibitem{BD:chiralalgebras}
{\sc Beilinson, A., and Drinfeld, V.}
\newblock {\em Chiral algebras}, vol.~51 of {\em American Mathematical Society
  Colloquium Publications}.
\newblock American Mathematical Society, Providence, RI, 2004.

\bibitem{HigherTranscendentalFunctions}
{\sc Erd\'{e}lyi, A., Magnus, W., Oberhettinger, F., and Tricomi, F.~G.}
\newblock {\em Higher transcendental functions. {V}ol. {I}}.
\newblock Robert E. Krieger Publishing Co., Inc., Melbourne, Fla., 1981.
\newblock Based on notes left by Harry Bateman, With a preface by Mina Rees,
  With a foreword by E. C. Watson, Reprint of the 1953 original.

\bibitem{FF92}
{\sc Feigin, B., and Frenkel, E.}
\newblock Affine {K}ac-{M}oody algebras at the critical level and
  {G}elfand-{D}iki\u{\i} algebras.
\newblock In {\em Infinite analysis, {P}art {A}, {B} ({K}yoto, 1991)}, vol.~16
  of {\em Adv. Ser. Math. Phys.} World Sci. Publ., River Edge, NJ, 1992,
  pp.~197--215.

\bibitem{GiorgiaPhD}
{\sc Fortuna, G.}
\newblock {The Beilinson-Bernstein Localization Theorem for the affine
  Grassmannian}.
\newblock PhD thesis, Massachusetts Institute of Technology. Available at
  \url{https://dspace.mit.edu/handle/1721.1/82437?show=full}, 2013.

\bibitem{Frenkel_Langlands_loop_group}
{\sc Frenkel, E.}
\newblock {\em Langlands correspondence for loop groups}, vol.~103 of {\em
  Cambridge Studies in Advanced Mathematics}.
\newblock Cambridge University Press, Cambridge, 2007.

\bibitem{FG1}
{\sc Frenkel, E., and Gaitsgory, D.}
\newblock {$D$}-modules on the affine {G}rassmannian and representations of
  affine {K}ac-{M}oody algebras.
\newblock {\em Duke Math. J. 125}, 2 (2004), 279--327.

\bibitem{FG3}
{\sc Frenkel, E., and Gaitsgory, D.}
\newblock Fusion and convolution: applications to affine {K}ac-{M}oody algebras
  at the critical level.
\newblock {\em Pure Appl. Math. Q. 2}, 4, Special Issue: In honor of Robert D.
  MacPherson. Part 2 (2006), 1255--1312.

\bibitem{FG2}
{\sc Frenkel, E., and Gaitsgory, D.}
\newblock Local geometric {L}anglands correspondence and affine {K}ac-{M}oody
  algebras.
\newblock In {\em Algebraic geometry and number theory}, vol.~253 of {\em
  Progr. Math.} Birkh\"{a}user Boston, Boston, MA, 2006, pp.~69--260.

\bibitem{FG5}
{\sc Frenkel, E., and Gaitsgory, D.}
\newblock Geometric realizations of {W}akimoto modules at the critical level.
\newblock {\em Duke Math. J. 143}, 1 (2008), 117--203.

\bibitem{FG7}
{\sc Frenkel, E., and Gaitsgory, D.}
\newblock Local geometric {L}anglands correspondence: the spherical case.
\newblock In {\em Algebraic analysis and around}, vol.~54 of {\em Adv. Stud.
  Pure Math.} Math. Soc. Japan, Tokyo, 2009, pp.~167--186.

\bibitem{FG4}
{\sc Frenkel, E., and Gaitsgory, D.}
\newblock Localization of {$\hat{\mathfrak{g}}$}-modules on the affine
  {G}rassmannian.
\newblock {\em Ann. of Math. (2) 170}, 3 (2009), 1339--1381.

\bibitem{FG6}
{\sc Frenkel, E., and Gaitsgory, D.}
\newblock Weyl modules and opers without monodromy.
\newblock In {\em Arithmetic and geometry around quantization}, vol.~279 of
  {\em Progr. Math.} Birkh\"{a}user Boston, Boston, MA, 2010, pp.~101--121.

\bibitem{GaitsgoryKM}
{\sc Gaitsgory, D.}
\newblock {Day IV, talk 3. Kac-Moody representations, notes of the Workshop
  ``Towards the proof of the geometric Langlands conjecture''}.
\newblock Available at
  \url{https://sites.google.com/site/geometriclanglands2014/notes}, 2014.

\bibitem{Raskin}
{\sc Raskin, S.}
\newblock {Day II, talk 2. Factorization I, notes of the Workshop ``Towards the
  proof of the geometric Langlands conjecture''}.
\newblock Available at
  \url{https://sites.google.com/site/geometriclanglands2014/notes}, 2014.

\bibitem{stacks-project}
{\sc {The Stacks project authors}}.
\newblock {The Stacks project}.
\newblock \url{https://stacks.math.columbia.edu}, 2020.

\end{thebibliography}

\Addresses

\end{document}